\documentclass[11pt]{amsart}

\usepackage[dvipsnames]{xcolor}
\usepackage{amssymb,verbatim}
\usepackage{amsmath,amsfonts,enumitem}
\usepackage[mathscr]{euscript} 
\usepackage{amsthm,bm}
\usepackage{url}
\usepackage{graphicx} 
\usepackage{enumitem} 
\usepackage{hyperref} 
\hypersetup{colorlinks} 
\usepackage{bbm} 
\usepackage{bm} 
\usepackage{mathrsfs} 
\usepackage{cancel} 

\usepackage[colorinlistoftodos]{todonotes}

\RequirePackage{cleveref}
\usepackage{hypcap}
\hypersetup{colorlinks=true, citecolor=darkblue, linkcolor=darkblue}
\definecolor{darkblue}{rgb}{0.0,0,0.7}
\newcommand{\darkblue}{\color{darkblue}}

\definecolor{darkred}{rgb}{0.68,0,0}
\newcommand{\darkred}{\color{darkred}}

\definecolor{darkgreen}{rgb}{0,.38,0}
\newcommand{\darkgreen}{\color{darkgreen}}

\newcommand{\defn}[1]{\emph{\darkblue #1}}
\newcommand{\defna}[1]{\emph{\darkred #1}}
\newcommand{\defnb}[1]{\emph{\darkblue #1}}
\newcommand{\defng}[1]{\emph{\darkgreen #1}}

\makeatletter
\def\th@plain{%
	\thm@notefont{}
	\itshape 
}
\def\th@definition{%
	\thm@notefont{}
	\normalfont 
}
\makeatother

\newtheorem{thm}{Theorem}[section]

\newtheorem*{claim*}{Claim}
\newtheorem{cor}[thm]{Corollary}
\newtheorem{prop}[thm]{Proposition}
\newtheorem{conj}[thm]{Conjecture}
\newtheorem{conv}[thm]{Convention}
\newtheorem{question}[thm]{Question}

\newtheorem{op}[thm]{Open Problem}
\newtheorem{dl}[thm]{Definition/Lemma}

\theoremstyle{definition}
\newtheorem{ex}[thm]{Example}
\newtheorem{rem}[thm]{Remark}
\newtheorem{definition}[thm]{Definition}

\numberwithin{figure}{section}
\numberwithin{equation}{section}

\def\wh{\widehat}
\def\bu{\bullet}
\def\emp{\nothing}

\def\zz{\mathbb Z}
\def\nn{\mathbb N}

\def\gg{{\mathcal{G}}}
\def\rr{\mathbb R}

\def\ap{\om}
\def\apr{\om}

\def\rA{ {{A} } }
\def\rC{ {{C} } }
\def\rH{ {\textrm{H} } }
\def\rHr{ {\text{\em H} } }

\def\rrs{\mathbb R_{>0}}

\def\ov{\overline}

\def\sm{\smallsetminus}

\def\Ga{\Gamma}
\def\Om{\Omega}

\def\De{\Delta}
\def\Si{\Sigma}
\def\la{\lambda}
\def\ga{\gamma}
\def\si{\sigma}
\def\de{\delta}
\def\ep{\epsilon}
\def\al{\alpha}
\def\be{\beta}
\def\bebo{\boldsymbol{\beta}}
\def\om{\omega}

\def\ve{\varepsilon}

\def\vk{\varkappa}

\def\cA{\mathcal A}
\def\cB{\mathcal B}

\def\cC{\mathcal C}

\def\cF{\mathcal F}

\def\cL{\mathcal L}

\def\cO{\mathcal O}

\def\cP{\mathcal P}
\def\PP{\mathcal P}

\def\cS{\mathcal S}
\def\cT{\mathcal T}

\def\bP{\mathbf{P}}
\def\bE{\mathbf{E}}
\def\Eb{\mathbf{E}}

\def\ssu{\subset}

\def\<{\langle}
\def\>{\rangle}

\def\sign{\text{\rm sign}}

\def\width{\text{{\rm width}}}

\def\rR{ {\text {\rm R} } }

\def\rA{ {\text {\rm A} } }

\def\Cat{ {\text {\rm Cat} } }
\def\Aut{ {\text {\rm Aut} } }

\def\ups{\upsilon}

\def\mix{{\text {\rm mix} } }

\def\0{{\mathbf 0}}

\def\LL{{\mathcal L}}

\def\nothing{\varnothing}

\def\.{\hskip.06cm}
\def\ts{\hskip.03cm}

\def\vol{{\text {\rm vol}}}

\def\ba{\textbf{\textit{a}}}

\def\bc{\textbf{\textbf{c}}}

\def\bq{\textbf{\textit{q}}}

\def\bz{{\textbf{\textit{z}}}}

\def\bM{\textbf{\textrm{M}}}

\def\aD{\la}

\def\La{\Lambda}

\def\di{{\small{\ts\diamond\ts}}}

\def\ze{{\zeta}}

\newcommand{\comp}{\mathrm{comp}}
\newcommand{\maj}{\mathrm{maj}}

\newcommand{\Vol}{\textnormal{Vol}} 

\def\aF{\textrm{F}}
\def\aFr{\textrm{\em F}}

\def\aN{\textrm{N}}
\def\aNr{\textrm{\em N}}

\def\.{\hskip.06cm}
\def\ts{\hskip.03cm}

\def\nin{\noindent}

\newcommand{\LE}{\textsc{\#LE}}

\def\poly{{\textsf{P}}}
\newcommand{\textsu}[1]{\textup{\textsf{#1}}}

\newcommand{\SI}{\textup{{\rm {SI}}}}
\newcommand{\ComCla}[1]{\textup{\textsu{#1}}}

\newcommand{\sharpP}{\ComCla{\#P}}
\newcommand{\SP}{\ComCla{\#P}}

\newcommand{\GapP}{\ComCla{GapP}}

\newcommand{\Sigmap}{\ensuremath{\Sigma^{{\textup{p}}}}}

\newcommand{\NP}{\ComCla{NP}}

\newcommand{\coNP}{\ComCla{coNP}}
\renewcommand{\poly}{\ComCla{P}}
\newcommand{\CeqP}{\ComCla{C$_=$P}}

\newcommand{\PH}{\ComCla{PH}}
\newcommand{\PSPACE}{\ComCla{PSPACE}}

\newcommand{\FP}{\ComCla{FP}}

\def\SP{{\sharpP}}
\def\CEP{{\CeqP}}

\def\Pb{{\text{\bf P}}}

\newcommand{\inv}{\operatorname{{\rm inv}}}
\newcommand{\INV}{\operatorname{{\rm INV}}}
\newcommand{\en}{{\ve}}

\DeclareMathOperator{\cb}{\mathbf{c}} 
\DeclareMathOperator{\Ec}{\mathcal{E}} 

\DeclareMathOperator{\Rb}{\mathbb{R}} 
\def\precc{\preccurlyeq} 
\DeclareMathOperator{\vb}{\mathbf{v}} 
\DeclareMathOperator{\wb}{\mathbf{w}} 
\DeclareMathOperator{\zb}{\mathbf{z}} 

\newcommand{\mn}{{{\text{min}}}} 

\title{Linear extensions of finite posets}

\begin{document}

\author[Swee Hong Chan \ ]{Swee Hong Chan \ }
\address[Swee Hong Chan]{Department of Mathematics, Rutgers University,  Piscatway, NJ 08854.}
\email{\texttt{sweehong.chan@rutgers.edu}}

\author[ \  Igor Pak]{ \ Igor Pak}
\address[Igor Pak]{Department of Mathematics, UCLA,  Los Angeles, CA 90095.}
\email{\texttt{pak@math.ucla.edu}}
	
\date{\today}



\begin{abstract}
We give a broad survey of inequalities for the number of linear
extensions of finite posets.  We review many examples, discuss
open problems, and present recent results on the subject.
We emphasize the bounds, the equality conditions of the
inequalities, and the computational complexity aspects of
the results.
\end{abstract}
	
\maketitle
	

\section{Introduction}\label{s:intro}

\subsection{Foreword} \label{ss:intro-foreword}
The world of linear extensions of finite posets is a microcosm of contradictions.
Although the counting problem is $\SP$-hard in general, it is polynomially easy
in many special cases (see~$\S$\ref{ss:CS-SP}).  Although posets themselves
do not have a geometric structure, the number of linear extensions is given
by the volume of two different polytopes with distinct geometric applications
(see~$\S$\ref{ss:proof-comb-opt} and~$\S$\ref{ss:proof-geom}).  Finally, although
there is a large number of correlation inequalities that hold in full generality,
direct injective proofs are rare and difficult to construct, sometimes provably so.

We straddle the boundary between inequalities which hold for general posets
and specialized inequalities for classes of posets, with the emphasis on the
former.  There is a dearth of powerful tools for study of linear extensions,
and yet the sheer volume of results continues to astonish us.  This is hardly
reflected in our presentation style which continues the dry tradition
of stating results and letting the reader judge for themselves.  But please
be assured --- many results in this survey are extremely surprising and worthy of
elaboration, reflection and contemplation.

\smallskip

\subsection{Content} \label{ss:intro-content}
In this paper we give a broad survey of inequalities for the
\defna{number of linear extensions} \ts of finite posets, both as a function of
the poset and when closely related posets are compared.  We also discuss several
closely related inequalities for the \defna{order polynomial}, and
\defna{correlation inequalities} \ts for probabilities of various events
associated with linear extensions.


Of the ocean of inequalities for linear extensions, we single out several
key inequalities which are treated in separate sections along with their many
extensions and variations:

\smallskip

$\circ$ \. \defng{Sidorenko's inequality}, see Section~\ref{s:sid},

\smallskip

$\circ$ \. \defng{Bj\"orner--Wachs inequality}, see Section~\ref{s:BW},

\smallskip

$\circ$ \. \defng{Fishburn's inequality}, see Section~\ref{s:Fish},

\smallskip

$\circ$ \. \defng{XYZ inequality}, see Section~\ref{s:corr},

\smallskip

$\circ$ \. \defng{Stanley's inequality}, see Sections~\ref{s:ineq-Sta},~\ref{s:eq-Sta}
and~\ref{s:proof-crit-simp}.

\smallskip

\nin
We emphasise the equality conditions of the inequalities, so e.g.\ Section~\ref{s:eq-Sta}
is dedicated to equality conditions of Stanley type inequalities.  Examples of various
families of posets are given in Section~\ref{s:ex}.  We discuss various computational
complexity aspects around the numbers of linear extension in Section~\ref{s:CS}.

In our presentation we aim to be as complete as possible, but for reasons
of space and readability we abandoned the hope to include any proofs
(with few exceptions, see Section~\ref{s:proof-crit-simp}). Instead,
we include Section~\ref{s:proof-ideas} with some vague ideas on what
is going on, and some additional references.
We conclude with final remarks in Section~\ref{s:finrem}.

\medskip


\section{Definitions and notation}\label{s:def}

We start by introducing  basic definitions and notations which appear throughout
the survey.  We also include many references to monographs, surveys and other
background reading on the subject.

\subsection{General notation}\label{ss:def-gen}
Let \ts $[n]=\{1,\ldots,n\}$, \ts $\nn=\{0,1,2,\ldots\}$ \ts 
and \ts $\rr_+=\{x\ge 0\}$.
For a subset \ts $S\subseteq X$ \ts and element \ts $x\in X$,
we write \ts $S+x:=S\cup \{x\}$ \ts and \ts $S-x:=S\sm\{x\}$.

For a sequence \ts $\ba =(a_1,\ldots,a_m)$, denote
\ts $|\ba| := a_1 + \ldots + a_m$\ts.  This sequence is
\defn{unimodal} \ts if \.
$a_1 \le a_2 \le \ldots \le a_\ell \ge a_{\ell+1} \ge \ldots a_m$ \.
for some \ts $\ell$. Sequence \ts $\ba$ \ts is \defn{log-concave} \ts if \ts
\. $a_i^2\ge a_{i-1} \ts a_{i+1}$ \. for all \ts $1< i < m$.  We refer to
\cite{Bra15,Bre89,Sta-log-concave} for many examples of log-concave sequences
in Algebraic and Enumerative Combinatorics.

We use the \defn{$q$-analogues} \. $(n)_q:=1+q+\ldots + q^{n-1}$, \, $n!_q := (1)_q \cdots (n)_q$ \, and
$$
\binom{n}{k}_q \, := \, \frac{n!_q}{k!_q \. (n-k)!_q}\,.
$$
These can be viewed either polynomials in \ts $\nn[q]$, or as
real numbers for a fixed \ts $q\in \rr$.

For an inequality \ts $f\geqslant g$, the difference
\ts $(f-g)$ \ts is called the \defn{defect}.
For polynomials \ts $f,g\in \rr[z_1,\ldots,z_n]$, we write \ts $f(\bz)\geq g(\bz)$ \ts
for the inequality between the values at a given \ts $\bz=(z_1,\ldots,z_n) \in \rr^n$, and \ts $f\geqslant_\bz g$ \ts
for the stronger inequality which holds coefficient-wise.

\smallskip

\subsection{Posets} \label{ss:def-posets}
Suppose \ts $P=(X,\prec)$ \ts and \ts $Q=(Y,\prec')$ \ts are two posets on
sets \ts $X\subseteq Y$, such that \ts $x \prec' y$ \. implies \. $x\prec y$.
Then we say that $P$ is a \defn{subposet} \ts of~$Q$.  If \. $x \prec' y$ \.
$\Leftrightarrow$ \. $x\prec y$, we say that $P$ is an
\defn{induced subposet} \ts of~$Q$.
For a poset \ts $P=(X,\prec)$ \ts and a subset \ts $A \subseteq X$, denote
by \ts $P|_A:=(A,\prec)$ \ts the induced subposet of~$P$.

We use \ts $(P-z)$ \ts to denote a subposet \ts $P|_{X-z}$\ts, where
\ts $z\in X$.  Element \ts $x\in X$ \ts is \defn{minimal} \ts in~$\ts P$,
if there exists no element  \ts $y \in X-x$ \ts such that
\ts $y \prec x$.  Define \defn{maximal} \ts elements similarly.
Denote by \ts $\min(P)$ \ts and \ts $\max(P)$ \ts the set
of minimal and maximal elements in~$P$, respectively.

When~$\ts P$ \ts has a unique minimal element, we use \ts $\wh 0$ \ts
to denote it. Similarly, when \ts $P$ \ts has a unique maximal element,
we use \ts $\wh 1$ \ts to denote it. Element \ts $x\in X$ \ts is said to \defn{cover}
\ts $y\in X$, if \ts $y\prec x$ \ts and there are no elements \ts $z\in X$ \ts
such that \. $y\prec z \prec x$. For two subsets \ts $A,B\subseteq X$,
\ts $A\cap B = \emp$, we write \. $A\prec B$ \. if \. $x \prec y$ \.
for all \. $x\in A$ \ts and \ts $y \in B$.

A subset \ts $A \subseteq X$ \ts is an \defn{upper ideal} \ts if
\ts $x \in A$ \ts and \ts $y \succ x$ \ts implies \ts $y \in A$.  Similarly,
a subset \ts $A \subseteq X$ \ts is a
\defn{lower ideal} \ts if \. $x \in A$ \ts and \ts $y \prec x$ \ts implies \ts $y \in A$.
Denote by \. $x\!\downarrow \ := \. \{y \in X\,{}:\,{}y\preccurlyeq x\}$ \. and \.
$x\!\uparrow \ := \. \{y \in X\,{}:\,{}y\succcurlyeq x\}$ \. the \defn{lower} \ts
and \defn{upper order ideals} \ts generated by~$x$, respectively.  Similarly,
for a subset \ts $B\subseteq X$, denote by
\. $B\!\downarrow{}  := \ts \cup_{b\in B} \, b\!\downarrow$ \. and
\. $B\!\uparrow{}  := \ts \cup_{b\in B} \, b\!\uparrow$ \. the \defn{lower}~{\ts and}
\defn{upper closure} of~$B$, respectively.  We use \.
$\al(x) := \left|x\!\downarrow\right|$, \. $\be(x) := \left|x\!\uparrow\right|$,
$\al(B) := \left|B\!\downarrow\right|$ \. and \. $\be(x) := \left|B\!\uparrow\right|$,
to denote their sizes.

In a poset \ts $P=(X,\prec)$, elements \ts $x,y\in X$ \ts are called
\defn{parallel} or \defn{incomparable} if \ts $x\not\prec y$ \ts
and \ts $y \not \prec x$.  We write \. $x\parallel y$ \. in this case.
Denote by \.  $\comp(x) := \{y \in X \, : \, x \prec y \ \text{or} \ x \succ y\}$ \.
the set of elements \ts $y\in X$ \ts comparable to~$x$.
\defn{Comparability graph} \. $\Ga(P) = (X,E)$ \ts is a graph on~$X$
with edges \ts $E=\{(x,y) \, : \, x\prec y \ \text{or} \ x\succ y \}$.

A \defn{chain} is a subset \ts $C\ssu X$ \ts of pairwise comparable elements.
Let \ts $\cC(P)$ \ts denote the set of chains in \ts $P$.  The \defn{height} \ts
of poset \ts $P=(X,\prec)$ \ts is the maximum size \ts $|C|$ \ts of a chain
\ts $C\in \cC(P)$.  Chain \ts $C\in \cC(P)$ \ts  is called \defn{maximal} \ts
if there is no chain \ts $C'\in \cC(P)$ \ts such that \ts $C\ssu C'$.

An \defn{antichain} is a subset \ts $A\ssu X$ \ts of pairwise incomparable elements.
Let \ts $\cA(P)$ \ts denote the set of antichains in \ts $P$.  The \defn{width} of
poset  \ts $P=(X,\prec)$ \ts is the maximal size \ts $|A|$ \ts of a antichain \ts $A\in \cA(P)$.
Antichain \ts $A$ \ts is called \defn{maximal} \ts if there is no antichain \ts
$A'\in \cC(P)$ \ts such that \ts $A\ssu A'$.

A \defn{dual poset} \ts is a poset \ts $P^\ast=(X,\prec^\ast)$, where
\ts $x\prec^\ast y$ \ts if and only if \ts $y \prec x$.
A \defn{product} \ts $P\times Q$ \ts of posets \ts $P=(X,\prec)$ \ts
and \ts $Q=(Y,\prec^\circ)$ \. is a poset \ts $(X\times Y,\prec^\di)$,
where the relation \. $(x,y) \preccurlyeq^\di (x',y')$  \. if and only if
\. $x \preccurlyeq x'$ \. and \. $y \preccurlyeq^\circ y'$,
for all \ts $x,x'\in X$ \ts and \ts $y,y'\in Y$.

A \defn{disjoint sum} \ts $P+Q$ \ts of posets \ts $P=(X,\prec)$ \ts
and \ts $Q=(Y,\prec')$ \. is a poset \ts $(X\cup Y,\prec^\di)$,
where the relation $\prec^\di$ coincides with $\prec$ and $\prec'$ on
$X$~and~$Y$, and \. $x\.\|\. y$ \. for all \ts $x\in X$, $y\in Y$.
A \defn{linear sum} \ts $P\oplus Q$ \ts of posets \ts $P=(X,\prec)$ \ts
and \ts $Q=(Y,\prec')$ \. is a poset \ts $(X\cup Y,\prec^\di)$,
where the relation $\prec^\di$ coincides with $\prec$ and $\prec'$ on
$X$~and~$Y$, and \. $x\prec^\di y$ \. for all \ts $x\in X$, $y\in Y$.

Posets constructed from one-element posets by recursively taking
disjoint and linear sums are called \defn{series-parallel}.
Both \defn{$n$-chain} \ts $C_n$ \ts and \defn{$n$-antichain} \ts $A_n$
\ts are examples of series-parallel posets.  These posets can be
characterized by not having poset \ts $N=\{x\prec y \succ z \prec w\}$
as induced subposet (thus they are also called \defn{$N$-free}).
\defn{Forest} \ts is
a series-parallel poset formed by recursively taking disjoint sums
(as before), and linear sums with one element: \ts $C_1 \oplus P$.

Let \ts $P=(X,\prec)$ \ts and \ts $P'=(X,\prec')$ \ts be two posets
on the same ground set.  We say that \ts $P$ \ts and \ts $P'$ \ts
are \defn{consistent} \ts if there are no elements \ts $x,y\in X$, s.t.\
\ts $x\prec y$ \ts and \ts $y \prec' x$.  For consistent posets, the
\defn{intersection} \ts $P\cap P' = (X,\prec^\di)$ \ts is
the poset with the union of the relations: \ts $x\prec y$ \ts
if \ts $x \prec y$ \ts or \ts $x \prec' y$.

Poset \ts $P=(X,\prec)$ \ts is \defn{graded} \ts
 if there is a rank function \ts $\rho: X\to \nn$ \ts such
 that \ts $\rho(x)\le \rho(y)$ \ts for all \ts $x\prec y$, and
 $\rho(y)= \rho(x)+1$ \ts for all covers~$y$ of~$x$. If $P$ and $Q$
 are graded, then so are \ts $P+Q$, \ts $P \oplus Q$ \ts and \ts $P\times Q$.
\defn{Boolean algebra} \ts $B_m := C_2 \times \cdots \times C_2$ \ts ($m$ times)
and the \defn{grid poset} \ts $G_{k\ell} := C_k \times C_\ell$ \ts are examples of graded
posets.

Let \ts $P=(X,\prec)$ \ts be a poset. For elements \ts $x,y\in X$,
the \defn{greatest lower bound} \ts $x\wedge y$ \ts is an element
s.t.\ \. $z\prec x, \ts z \prec y \, \Rightarrow \, z \prec (x\wedge y)$.
Similarly, the \defn{least upper bound} \ts $x \vee y$ \ts is an element
s.t.\ \. $z\succ x, \ts z \succ y \, \Rightarrow \, z \succ (x\wedge y)$.
Poset \ts $P$ \ts is a \defn{lattice} \ts if \ts $x\wedge y$ \ts and
\ts $x \vee y$ \ts are well defined for all \ts $x,y \in X$.  We use
\ts $\cL=(X,\vee,\wedge)$ \ts notation in this case.  Lattice \ts $\cL$ \ts
is \defn{distributive} \ts if \.
$x\wedge (y \vee z) = (x \wedge y) \vee (x \wedge z)$ \. for all \ts $x,y,z\in X$
(note that the dual identity follows).  Both \ts $B_m$ \ts and \ts $G_{k\ell}$ \ts
are examples of distributive lattices.

We refer to \cite[Ch.~3]{Sta-EC} and \cite[Ch.~12]{West21} for accessible
textbook introductions to posets,
to surveys \cite{BW00,Tro95} for further definitions and standard
results, and to \cite{CLM12} for a recent monograph.\footnote{See also \ts
\href{https://en.wikipedia.org/wiki/Glossary_of_order_theory}{w.wiki/7Yy6}
\ts for a quick guide to the terminology.}

\smallskip

\subsection{Linear extensions and order polynomial}\label{ss:def-LE}
Let \. $P=(X,\prec)$ \. be a poset with \. $|X|=n$ \. elements.
Denote \. $[n]:=\{1,\ldots,n\}$.
A \defn{linear extension} of $P$ is a bijection \. $f: X \to [n]$,
such that
\. $f(x) < f(y)$ \. for all \. $x \prec y$.
Denote by \ts $\Ec(P)$ \ts the set of linear extensions of $P$,
and let \. $e(P):=|\Ec(P)|$.
%

For an integer \ts $t\ge 1$, denote by \. $\Omega(P,t)$ \. the number of
order preserving maps \. $g: X\to [t]$, i.e.\ maps which satisfy
\. $g(x)\le g(y)$ \. for all \. $x \prec y$.  This is the
\defn{order polynomial} \ts corresponding to poset~$P$.

Let \ts $P=(X,\prec)$, where \ts $X=\{x_1,\ldots,x_n\}$.  We will always assume that \ts $X$
\ts has a \defn{natural labeling}, i.e.\ \ts $f: x_i\to i$ \. is a linear extension.
A \defn{$P$-partition} \ts is an order preserving map \. $h: X\to \nn$, i.e.\
maps which satisfies \. $h(x)\le h(y)$ \. for all \. $x \prec y$.
Denote by \ts $\PP(P)$ \ts the set of $P$-partitions and let  \ts $\PP(P,t)$ \ts be
the set of $P$-partitions with values at most~$t$.\footnote{In \cite{Sta-thesis,Sta-EC},
Stanley uses $P$-partitions to denote order-reversing rather than order-preserving
maps.}

Let
\begin{equation}\label{eq:OP-q-def}
\Omega_q(\cP) \, := \, \sum_{h\ts \in \ts \PP(P)} \. q^{h(x_1) \. + \. \ldots \. +  \. h(x_n)} \quad \text{and} \quad
\Omega_q(\cP,t) \, := \, \sum_{h \ts \in \ts \PP(P,t)} \. q^{h(x_1) \. + \. \ldots \. +  \. h(x_n)}\..
\end{equation}
Stanley showed, see \cite[Thm~3.15.7]{Sta-EC}, that there is a statistics \. $\maj: \Ec(\cP) \to \nn$,
such that
\begin{equation}\label{eq:Sta-PP}
\Omega_q(P) \ = \ \frac{e_q(P)}{(1-q)(1-q^2)\cdots (1-q^n)} \,,
\end{equation}
where
\begin{equation}\label{eq:Sta-PP-q}
e_q(P) \, := \, \sum_{f\ts \in \ts \Ec(P)} \. q^{\maj(f)}\..
\end{equation}
More generally, let
\begin{equation}\label{eq:OP-bq-def}
\Omega_\bq(P,t) \, := \, \sum_{h\in \PP(P,t)} \. q_1^{h(x_1)} \ts \cdots \. q_n^{h(x_n)}\..
\end{equation}
We call this GF the \defn{multivariate order polynomial}.
Note that Stanley gave a generalization
of \eqref{eq:Sta-PP} and \eqref{eq:Sta-PP-q} for \ts $\Omega_\bq(\cP)$,  see \cite[Thm~3.15.5]{Sta-EC}. Finally,
for integer \ts $t\geq 0$, define
\begin{equation}\label{eq:OP-rK-def}
\Phi_\bz(P,t) \, := \, \sum_{h \ts\in \ts \PP(P,t)} \. z_0^{m_0(h)} \. z_1^{m_1(h)} \. \cdots \. z_t^{m_t(h)}\.,
\end{equation}
where \ts $m_i(h):= |h^{-1}(i)|$ \ts is the number of values $i$ in the \ts $P$-partition~$h$.  Clearly,
\begin{equation}\label{eq:OP-rK-exp}
\Phi_\bz(P,t) \, = \, \Om_q(P,t), \quad \text{where} \quad \bz \. = \. (1,q,q^2,\ldots,q^t).
\end{equation}

Finally, for a fixed linear ordering \ts $X=\{x_1,\ldots,x_n\}$,
denote by \ts $\sign(f)$ \ts the sign of \ts $f\in \Ec(P)$ \ts viewed
as a permutation in~$S_n\ts$. Define the \defn{sign-imbalance}:
\begin{equation}\label{eq:def-SI}
\SI(P) \, := \, |\Si(P)|\ts, \quad \text{where} \quad \Si(P) \, := \, \sum_{f\in \Ec(P)} \. \sign(f).
\end{equation}
Clearly, the sign-imbalance \ts $\SP(P)$ \ts is independent of the ordering of~$X$.
Poset~$P$ is called \defn{sign-balanced} \ts if \ts $\SP(P)=0$.
We refer to \cite{Sta-sign} for the introduction to sign-(im)balance.

\smallskip

\smallskip

\subsection{Poset polytopes}\label{ss:def-poset-polytopes}
Let \ts $P=(X,\prec)$ \ts be a poset with \ts $|X|=n$ \ts elements.
The \defn{order polytope} \ts $\cO_P\subset \rr^n$ \ts is defined as
\begin{equation}\label{eq:order-def}
0\le \al_x \le 1 \quad \text{for all}  \quad x\in X, \qquad
\al_x \le \al_y \quad \text{for all}  \quad  x\prec y, \ \ x,y \in X.
\end{equation}
Similarly, the \defn{chain polytope} (also known as the \defn{stable set polytope})
 \ts $\cS_P\subset  \rr^n$ \ts  is defined as
\begin{equation}\label{eq:chains-def}
\be_x \ge 0 \quad \text{for all}  \quad x\in X\., \qquad
\be_x + \be_y + \ldots \le 1 \quad \text{for all}  \quad  C=\{x\prec y \prec \cdots \} \in \cC(P).
\end{equation}
In \cite{Sta-two}, Stanley computed the volume of both polytopes:
\begin{equation}\label{eq:two-poset}
\Vol \ts (\cO_P) \, = \, \Vol \ts (\cS_P) \, = \, \frac{e(P)}{n!}\..
\end{equation}
This connection is the key to many applications of geometry to
poset theory and vice versa.
Stanley also computed the \defn{Ehrhart polynomial} \ts
\begin{equation}\label{eq:two-poset-OP}
\big|t\cdot \cO_P \cap \zz^n\big| \, = \,\big|t\cdot \cS_P \cap \zz^n\big| \, = \, \Omega(P,t+1)\..
\end{equation}
Note that \eqref{eq:two-poset} and \eqref{eq:two-poset-OP} give:
\begin{equation}\label{eq:OP-asy}
\Omega(P,t) \ \sim \ \frac{e(P)\, t^n}{n!}   \quad \text{as} \ \ \ t \to \infty\ts.
\end{equation}

\smallskip

\subsection{Complexity} \label{ss:def-CS}
We assume that the reader is familiar with basic notions and results in
computational complexity and only recall a few definitions.  We use standard
complexity classes: \. $\poly$, \. $\FP$, \. $\NP$,\. $\coNP$, \. $\SP$, \. $\Sigmap_m$ \. and \. $\PH$.
The notation \. $\{a =^? b\}$ \. is used to denote the
decision problem whether \ts $a=b$.  We use the \emph{oracle notation} \ts
{\sf R}$^{\text{\sf S}}$ \ts for two complexity classes \ts {\sf R}, {\sf S} $\subseteq \PH$,
and the polynomial closure \ts $\<${\sc A}$\>$ for a problem \ts {\sc A} $\in \PSPACE$.
We will also use less common classes \.
$$
\GapP:= \{f-g \mid f,g\in \SP\} \quad \text{and} \quad
\CEP:=\{f(x)=^?g(y) \mid f,g\in \SP\}.
$$
Note that \ts $\coNP \subseteq \CEP$.

We also assume that the reader is familiar with standard decision and
counting problems: \ts {\sc 3SAT}, \ts {\sc \#3SAT} \ts and
\ts {\sc PERMANENT}.  Denote by \ts {\sc \#LE} \ts the problem of
computing the number \ts $e(P)$ \ts of linear extensions.
%
For a counting function \ts $f\in \SP$,
the \defn{coincidence problem} \ts is defined as:
$$\text{\sc C}_f \ := \ \big\{\ts f(x) \. = ^? \ts f(y) \ts \big\}.
$$
Note the difference with the \defn{equality verification problem}:
$$
\text{\sc E}_{f-g} \, := \, \big\{ \ts f(x) \. =^? \. g(x) \big\},
$$
where \ts $f,g\in \SP$ \ts are counting functions and
\ts $x\in X$ \ts is an input.
Clearly, we have both \ts $\text{\sc E}_{f-g}\in \CEP$ \. and \. $\text{\sc C}_f \in \CEP$.
Note also that \.$\text{\sc C}_\text{\#3SAT}$ \ts
is both \ts $\CEP$-complete \ts and \ts $\coNP$-hard.

The distinction between \emph{binary} \ts and \emph{unary} \ts presentation
will also be important.  We refer to \cite{GJ78} and \cite[$\S$4.2]{GJ79}
for the corresponding notions of $\NP$-completeness and \emph{strong} \ts $\NP$-completeness.
Unless stated otherwise, we use the word ``\emph{reduction}'' \ts to mean the
polynomial Turing reduction.

We refer to \cite{AB09,Gold,Pap} for definitions and standard results
in computational complexity, and to \cite{Aar16,Wig19} for extensive
surveys of computational complexity applications in mathematics.
See \cite{GJ79} for the classical introduction
and a long list of $\NP$-complete problems.
See also \cite[$\S$13]{Pak-OPAC}
for a recent overview of $\SP$-complete problems in combinatorics.



\medskip

\section{Basic inequalities for the numbers linear extensions}\label{s:basic}

Here by basic inequalities we mean inequalities for \ts $e(P)$ \ts in terms of
various poset parameters.  The inequalities themselves range from elementary
to highly nontrivial.

\subsection{Induced subsets} \label{ss:basic-induced}
Let \ts $P=(X,\prec)$ \ts be a finite poset.  Denote by \ts $k(P)$ \ts
the number of incomparable pairs of elements: \ts $x \ts \| \ts y$, where \ts $x,y\in X$.
Equivalently, \ts $k(P)$ \ts is the number of induced subposets isomorphic to \ts $A_2$\ts.
Similarly, denote by \ts $\ell(P)$ \ts and \ts $m(P)$ \ts the number of induced subposets
isomorphic to \ts $A_3$ \ts and \ts $(C_2+C_1)$, respectively.

\begin{prop}[{\rm Ewacha, Rival and Zaguia \cite{ERZ97}}{}] \label{p:basic-induced}
Let \ts $P=(X,\prec)$ \ts be a finite poset, and let \ts $k=k(P)$, \ts $\ell=\ell(P)$
\ts and \ts $m=m(P)$ \ts defined as above.  Then:
\begin{equation}\label{eq:basic-induces}
2^k \ts \left(\tfrac34\right)^{\ell+m} \, \le \, e(P) \, \le \, 2^k.
\end{equation}
\end{prop}

The upper bound is trivial and tight for linear sums of \ts $A_2$ \ts and \ts $C_1$\ts.
The lower bound is tight for \ts linear sums of \ts $A_3$ \ts and \ts $(C_2+C_1)$ \ts
and is weak when \ts $(\ell+m)$ \ts is large.  The authors have conjectured further
inequalities of this type.

\smallskip

\subsection{Partitions into chains and antichains} \label{ss:basic-part}
Suppose \ts $P=(X,\prec)$ \ts is a subposet of \ts $Q=(X,\prec')$, i.e.,
we have \. $x \prec' y$ \. implies \. $x\prec y$.  Clearly,
we have \ts $e(P)\le e(Q)$.  The following easy consequence is
especially notable:

\begin{prop}  \label{p:basic-chains}
Let \ts $P=(X,\prec)$ \ts be a poset with \ts $|X|=n$ \ts elements,
and let \ts $X=C_1 \cup\ldots \cup  C_\ell$ \ts be a partition into
disjoint chains of sizes \. $c_1,\ldots,c_\ell$\ts.  Then we have:
\begin{equation}\label{eq:basic-chains}
e(P) \, \le \, \tbinom{n}{\ts c_1\ts,\ts \ldots \ts,\ts c_\ell \ts}\ts.
\end{equation}
\end{prop}

We also have the following antichain version, straight by definition.

\begin{prop}  \label{p:basic-antichains}
Let \ts $P=(X,\prec)$, and let \ts $X=A_1 \cup\ldots \cup  A_m$ \ts be a partition
into disjoint antichains of sizes \. $a_1,\ldots,a_m$\ts,
such that \. $A_1 \prec \ldots \prec A_m$\ts. Then we have:
\begin{equation}\label{eq:basic-antichains}
e(P) \, \ge \, a_1! \. \cdots \. a_m!
\end{equation}
\end{prop}

The following is a surprising generalization that extends this to all
partitions into antichains.

\begin{thm}[{\rm Bochkov and Petrov \cite[Cor.~6]{BP21}}{}]  \label{t:basic-antichains-BO}
Let \ts $P=(X,\prec)$ \ts be a poset, and let \ts $X=A_1 \cup\ldots \cup  A_m$ \ts be a partition
into disjoint antichains of sizes \. $a_1,\ldots,a_m$\ts.  Then we have:
\begin{equation}\label{eq:basic-antichains-BP}
e(P) \, \ge \, a_1! \. \cdots \. a_m!
\end{equation}
\end{thm}

In \cite[Thm~1]{BP21}, both the upper bound \eqref{eq:basic-chains}
and the lower bound \eqref{eq:basic-antichains} are extended to the
\defng{Greene--Kleitman--Fomin parameters}, see e.g.\ \cite{BF01,GK78}.

\smallskip

\subsection{Recursion over antichains} \label{ss:basic-recursion}
The following result follows directly from the definition:

\begin{prop}  \label{p:basic-recursion}
Let \. $\min(P)$ \ts be the set of minimal elements of \ts $P=(X,\prec)$.  We have:
\begin{equation}\label{eq:basic-recursion}
e(P) \, = \, \sum_{x \ts \in \ts \min(P)} \. e(P-x)\ts.
\end{equation}
\end{prop}

By induction, this gives the following upper bound:

\begin{cor} \label{c:basic-width}
Let \ts $P=(X,\prec)$ \ts be a poset of width \ts $w$ \ts with \ts $|X|=n$ \ts elements.
Then we have:
\. $e(P) \le w^n$.
\end{cor}

The following inequality is a surprising generalization of the proposition:

\begin{thm}[{\rm Edelman--Hibi--Stanley~\cite{EHS89} and Stachowiak~\cite{Sta89a}}{}]\label{t:basic-EHS}
Let \ts $A$ \ts be an antichain in \ts $P=(X,\prec)$.  We have:
\begin{equation}\label{eq:basic-EHS}
e(P) \, \ge \, \sum_{x \ts \in \ts A} \. e(P-x)\ts.
\end{equation}
Moreover, this inequality is an equality \. \underline{if and only if} \.
$A$ \. intersects every maximal chain, i.e.\ \ts $|C\cap A|=1$ \ts
for all \ts $C\in \cC(P)$. Additionally, the defect of~\eqref{eq:basic-EHS}
is in~$\ts \SP$.
\end{thm}

Clearly, the set of minimal elements \ts $\min(P)$ \ts satisfies contains
an element from every maximal chain, so Theorem~\ref{t:basic-EHS} implies
Proposition~\ref{p:basic-recursion}. Note that not every maximal antichain
satisfies this condition. For example, take chains \ts $C_2$ \ts on \ts
$\{1,2\}$ \ts and \ts $C_3$ \ts on \ts $\{1',2',3'\}$, and let \ts
$P=C_2 \times C_3$. Consider an antichain \ts
$A=\big\{(1,3'), \ts (2,1')\big\}$ in~$P$.  Now note that a maximal chain
\. $C=\big\{(1,1'), \ts (1,2'), \ts (2,2'), \ts (2,3')\big\}$ \.
does not contain elements in~$A$.

Note that the original proof of Theorem~\ref{t:basic-EHS} in~\cite{EHS89}
is combinatorial and uses promotion operators (see~$\S$\ref{ss:proof-inject}).
We refer to \cite[Cor.~8.2]{CPP-effective} for a proof using Sidorenko's flow
(see~$\S$\ref{ss:proof-comb-opt}).  The proof in~\cite{Sta89a} uses a
simple induction, and the theorem is applied to obtain the
following result:

\begin{thm}[{\rm Stachowiak~\cite{Sta89a}}{}]\label{t:Stach}
Let \ts $P=(X,\prec)$ \ts and \ts $Q=(X,\prec')$ \ts be two posets on
the same ground set, such that their comparability graphs satisfy:
\ts $\Ga(P) \supseteq \Ga(Q)$.  Then \ts $e(P) \le e(Q)$.  In particular,
$e(P)$ \ts depends only on \ts $\Ga(P)$.  Additionally, if \ts $e(P)=e(Q)$,
then \ts $P=Q$.
\end{thm}

For posets of height two, Theorem~\ref{t:Stach} was proved in an earlier
paper \cite{Sta88}.
That \ts $e(P)$ \ts depends only on the comparability graph \ts $\Ga(P)$ \ts
was also proved in~\cite{EHS89}, and extended to the order polynomial.
Note that Theorem~\ref{t:Stach} follows easily from the volume formula
\eqref{eq:two-poset} and geometric description of vertices of the chains
polytope \ts $\cC_P(P)$, see below.   We refer to \cite{Sta-two} for the
introduction, and to \cite{Iri17} for a recent exploration of the connection
between \ts $e(P)$ \ts and orientations of \ts $\Ga(P)$.

\begin{cor} \label{c:Stach}
Let \ts $P=(X,\prec)$ \ts be a poset with \ts $|X|=n$ \ts elements,
and let \ts $x\in X$.  Then: \. $e(P) \le n \ts e(P-x)$.
\end{cor}

This follows immediately from Theorem~\ref{t:Stach} by taking \ts
$Q:=(P-x) + C_1$ \ts and noting that \ts $\Ga(P) \supseteq \Ga(Q)$,
\ts $e(Q) = n \ts e(P-x)$.

\smallskip

\subsection{Weighted chains} \label{ss:basic-weighted-chains}
There is a better way to give an upper bound for \ts $e(P)$, by assigning weights
to elements of antichains.

\begin{prop}\label{p:LYM-antichains}
Let \ts $\xi: X\to \rr_{>0}$ \ts be a positive function s.t.\
\begin{equation}\label{eq:LYM-antichain-condition}
\sum_{x\ts \in \ts A} \. \xi(x) \, \le \, 1 \quad \text{for all} \quad A\in \cA(P).
\end{equation}
Then:
\begin{equation}\label{eq:LYM-antichain}
e(P) \, \le \, \prod_{x\in X} \. \frac{1}{\xi(x)}\..
\end{equation}
\end{prop}

\smallskip

For example, \ts $\xi(x):=1/w$ \ts gives Corollary~\ref{c:basic-width}.
The following construction shows how this bound can be improved.
For $x\in X$, denote by \ts $\cC_x(P)\subseteq \cC(P)$ \ts the subset
of chains in~$P$ which contain~$x$.  Let \ts $v: \cC(P) \to \rr_{\ge 0}$ \ts
be a probability distribution on chains in~$P$.  Denote
$$
c_x \, := \, \sum_{C\in \cC_x(P)} \. v(C)\ts.
$$

\begin{cor}\label{c:LYM-chains-upper}
Suppose \ts $c_x>0$ \ts for all \ts $x\in X$.  Then:
\begin{equation}\label{eq:LYM-chains-upper}
e(P) \, \le \, \prod_{x\in X} \. \frac{1}{c_x}\..
\end{equation}
\end{cor}

\begin{proof} Take $\xi(x) := c_x$.  Observe that condition
\eqref{eq:LYM-antichain-condition} holds trivially for every chain \ts
$C\in \cC(P)$, and thus for every probability distribution~$v$ on $\cC(P)$.
Thus, \eqref{eq:LYM-chains-upper} follows from \eqref{eq:LYM-antichain}.
\end{proof}

\smallskip

\subsection{LYM property} \label{ss:basic-LYM}
Let \ts $P=(X,\prec)$ \ts be graded poset with the rank function \ts $\rho: X\to \{0,\ldots,\ell\}$.
Denote by \ts $X_r$ \ts the set of elements of rank~$r$, i.e.\ $X_r=\rho^{-1}(r)$,
and let \ts $n_r:=|X_r|$ \ts for all \ts $1\le r \le \ell$.  Clearly, $n_r \le w$,
where \ts $w$ \ts is the width of~$P$.

We say that $P$ has \defn{LYM property} \ts if for every antichain \ts $A\in \cA(P)$ \ts we have:
\begin{equation}\label{eq:LYM-def}\tag{LYM}
\sum_{x\in A} \. \frac{1}{n_{\rho(x)}} \, \le \, 1\ts.
\end{equation}
In particular, the width \ts $w = \max_r n_r$\ts.  Such posets are
called \defn{LYM posets}.

Let $G$ be a subgroup of $\Aut(P)$ which acts transitively on the set
\ts $\cC_{\text{max}}$ \ts of maximal chains in~$P$.
Let \ts $v$ \ts be a uniform distribution on $\cC_{\text{max}}$.
By transitivity, $c_x = \frac{1}{n_r}$ \. where \ts $r=\rho(x)$.
Since \ts $\cC_x \cap \cC_y=\emp$ \ts for all $x,y\in A$, we have \eqref{eq:LYM-def}.
We can now use \eqref{eq:LYM-chains-upper} to conclude:

\begin{cor}[{\rm \cite{SK87}}{}]\label{c:LYM-bounds}
Let \ts $P=(X,\prec)$ \ts be a graded poset with \ts $n_r$ \ts elements of rank~$r$,
and with a transitive action on the set of maximal chains in~$P$.  Then \ts $P$ \ts
is a LYM poset, and
\begin{equation}\label{eq:LYM-bounds}
e(P) \, \le \, \prod_{r=0}^\ell \. (n_r)^{n_r}\..
\end{equation}
\end{cor}

Note the sequence of implications here:  transitive action implies
\eqref{eq:LYM-def} and the upper bound in \eqref{eq:LYM-bounds}.  The following
result proves the implication directly.

\begin{thm}[{\rm Brightwell--Tetali~\cite[Thm~5.1]{BT03}}{}]\label{t:LYM-BT}
Let \ts $P=(X,\prec)$ \ts be a graded LYM poset with $n_r$ elements of rank~$r$,
\ts $0\le r \le \ell$.  Then:
\begin{equation}\label{eq:LYM-bounds-BT}
\prod_{r=0}^\ell \. (n_r)! \, \le \,
e(P) \, \le \, \prod_{r=0}^\ell \. (n_r)^{n_r}\..
\end{equation}
\end{thm}

Here the lower bound is given by \eqref{eq:basic-antichains}.  The upper bound
is an improvement over the \ts $e(P)\le w^n$ \ts bound in Corollary~\ref{c:basic-width}.
Corollary~\ref{c:LYM-bounds} was extended to all graded posets with LYM property.
See also \cite{Sha98} for the same result with slightly stronger assumptions.
We refer to \cite{GK76,GK78,Kle74} for equivalent definitions of LYM posets, and to
\cite[$\S$12.2]{West21} for further applications of~\eqref{eq:LYM-def}.

\begin{ex}[{\rm Boolean algebra}{}] \label{ex:LYM-Boolean}
Let \ts $B_k$ \ts be the poset of all subsets of $[k]$ by inclusion,
so $n=2^k$ and $\ell=k$.   Observe that the symmetric group $S_k$ acts
transitively on \ts $\cC(B_k)$, and thus \ts $B_k$ \ts
has the LYM property.
In particular, this implies \defng{Sperner's theorem}:
$$w \, := \, \width(B_k) \, = \, \tbinom{k}{\lfloor k/2\rfloor}.
$$
Now both Sha--Kleitman inequality~\eqref{eq:LYM-bounds} and Brightwell--Tetali inequality \eqref{eq:LYM-bounds-BT} give:
\begin{equation} \label{eq:LYM-Boolean-bounds}
\prod_{r=0}^k \. \tbinom{k}{r}! \, \le \,
e(B_k) \, \le \, \prod_{r=0}^k \. \tbinom{k}{r}^{\tbinom{k}{r}}  \, \le \, \tbinom{k}{\lfloor k/2\rfloor}^{2^k}.
\end{equation}
Lower bound and either upper bound give correct two leading terms of the asymptotics:
$$
\frac{\log_2 e(B_k)}{2^k} \, = \, \log_2 \tbinom{k}{\lfloor k/2\rfloor} \. + \. \Theta(1)
\, = \, k \. -  \. \tfrac12 \log_2 k \. + \. \Theta(1),
$$
see e.g.\ \cite{Coo09} for a careful calculation.
In \cite{BT03}, it was shown that the lower bound in~\eqref{eq:LYM-Boolean-bounds}
gives the exact value of the constant implied by the \ts $\Theta(1)$ \ts notation.
\end{ex}

\begin{ex}[{\rm Products of LYM posets}{}] \label{ex:LYM-products}
Let \ts $P,Q$ \ts be LYM posets with log-concave rank numbers.  It was shown in
\cite{Har74,HK73}, that the product \ts $P \times Q$ \ts satisfies the same properties.
In particular, this implies that the product of chains \. $C_p \times C_q \times \ldots$
are also LYM~posets. Thus, the bound in Theorem~\ref{t:LYM-BT} applies and generalized
the bounds in~\eqref{eq:LYM-Boolean-bounds}.   Note that there is no
transitive action on maximal chains in this case.
\end{ex}

\smallskip

\subsection{Entropy bounds} \label{ss:basic-entropy}
For a poset \ts $P=(X,\prec)$ \ts on \ts $|X|=n$ \ts elements,
define the \defn{entropy}
\begin{equation}\label{eq:basic-entropy}
\rH(P) \, := \, \min_{{{\bebo} \ts \in \ts \cS_P}} \. \Big( - \. \frac{1}{n} \. \sum_{x\in X} \. \log \be_x\Big),
\end{equation}
where the minimum is over all vectors \. $\bebo:X\to \rr_{>0}$ \. in the chain polytope \ts $\cS_P$\ts.

\begin{thm}[{\rm Kahn--Kim~\cite{KK95}}{}]\label{t:basic-entropy-KK}
We have:
\begin{equation}\label{eq:basic-entroy-KK1}
n\log_2 n \. - \. n \cdot \rHr(P) \, \ge \, \log_2 e(P)  \, \ge \, 0.09 \bigl(n\log_2 n \. - \. n \cdot \rHr(P)\bigr).
\end{equation}
Additionally,
\begin{equation}\label{eq:basic-entroy-KK2}
\log_2 e(P)  \, \ge \, \log_2 n! \. - \. n \cdot \rHr(P) \, \ge \, n \log_2 n \. - \. (\log_2 e) \ts n \. - \. n \cdot \rHr(P).
\end{equation}
\end{thm}

Since the entropy on convex bodies can be approximated in polynomial time,
this result can be viewed as a deterministic approximation algorithm for~$\ts e(P)$.

\smallskip

\subsection{Height two posets} \label{ss:basic-height-two}
Let \ts $P=(X,\prec)$ \ts be a poset of height two with \ts $n=|X|$ \ts elements.
Let \ts $X=Y\cup Z$ \ts be a partition into two antichains \ts $Y,Z\in \cA(P)$ \ts
corresponding to rank $0$ and~$1$, respectively.

\begin{thm}[{\rm Brightwell--Tetali~\cite[Thm~1.4]{BT03}}{}]\label{t:basic-height-two-BT}
Suppose there exist integers \ts $a,b\in \nn$, such that
\ts $\al(z)= a$ \ts and \ts $\be(y)= b$, for all \ts $y\in Y$ \ts
and \ts $z\in Z$. Then:
\begin{equation}\label{eq:basic-height-two-BT}
e(P) \, \le \, n! \. \tbinom{a+b}{a}^{-n/(a+b)}.
\end{equation}
\end{thm}

This inequality is sharp for \ts $k:=n/(a+b)\in \nn$, as can be seen for
a disjoint sum of $k$ copies of poset \ts $K_{ab} :=A_b \oplus A_a\ts$.
Curiously, \eqref{eq:basic-height-two-BT} fails if we instead
use \ts $\al(x)\ge a$ \ts and \ts $\be(y)\ge b$, see \cite{BT03}.
The authors deduce  Theorem~\ref{t:basic-height-two-BT} form
the following result and the asymptotic formula~\eqref{eq:OP-asy}.

\begin{thm}[{\rm Brightwell--Tetali~\cite[Thm~3.2]{BT03}}{}]\label{t:basic-height-two-BT-OP}
In conditions of Theorem~\ref{t:basic-height-two-BT}, for all \ts $t\ge 1$ \ts
we have:
\begin{equation}\label{eq:basic-height-two-BT-OP}
\Om(P,t) \, \le \, \Om(K_{ab},t)^{n/(a+b)}.
\end{equation}
\end{thm}

The authors prove the result using technical entropy computations.

\medskip



\section{Basic inequalities for order polynomials}\label{s:basic-OP}

Order polynomial is just as fundamental object as the number of linear extensions,
and in many cases easier to work with.  Additionally, it has a clear geometric
interpretation as the Ehrhart polynomial for poset polytopes, see
\eqref{eq:two-poset-OP}.

\subsection{Explicit lower bound}\label{ss:basic-OP-explicit}
The following inequality extends the asymptotic formula~\eqref{eq:OP-asy}:

\begin{thm}[{\rm \cite[Thm~1.4, Cor~6.3]{CPP-effective}}{}]\label{t:basic-OP}
Let \ts $P=(X,\prec)$ \ts be a poset with \ts $|X|=n$ \ts elements.  Then,
for all integer \ts $t\ge 1$, we have:
\begin{equation}\label{eq:basic-OP}
\Omega(P,t) \ \geq \ \frac{e(P)\, t^n}{n!}\,.
\end{equation}
Moreover, the equality holds for a given \ts $t\ge 1$ \. \underline{if and only if} \. $P=A_n$ \ts
is an antichain.  Additionally, we have:
$$\Omega(P,t) \cdot n! \, - \, e(P) \, t^n \ \in \ \SP\ts.
$$
\end{thm}

Note that \eqref{t:basic-OP} improves upon a straightforward inequality \.
$\Omega(P,t) \ge e(P) \binom{t}{n}$, where \ts $t\ge n$.  The authors
prove the inequality by an explicit injection.


\subsection{Log-concavity}\label{ss:basic-OP-log-concavity}
Evaluations of the order polynomial have additional properties:

\begin{thm}[{\rm \defn{log-concavity}, Brenti~\cite[Thm~7.6.5]{Bre89}}{}]
\label{t:basic-OP-log-concave}
	Let \ts $P=(X,\prec)$ \ts be a poset with \ts $|X|=n$ \ts elements.
	Then, for all integer \ts $t\ge 2$, we have:
	\begin{equation}\label{eq:basic-OP-log-concave}
	 \Omega(P,t)^2  \ \geq \  \Omega(P,t+1) \  \Omega(P,t-1).
	\end{equation}
\end{thm}


In other words, \eqref{eq:basic-OP-log-concave} gives log-concavity of
values of the order polynomial.  This inequality is always strict:


\begin{thm}[{\rm \cite[Thm~4.8]{CPP-effective}}{}]
\label{t:basic-OP-log-concave-strict}
	Let \ts $P=(X,\prec)$ \ts be a poset with \ts $|X|=n$ \ts elements.
	Then, for all integer \ts $t\ge 2$,  we have:
	\begin{equation}\label{eq:log-concave-strict}
	 \Omega(P,t)^2  \ \geq \, \Big(1 \. + \. \frac{1}{(t+1)^{n+1}} \Big) \,   \Omega(P,t+1) \   \Omega(P,t-1).
	\end{equation}
\end{thm}

Note that the \. $\frac{1}{(t+1)^{n+1}}$ \. term is far from optimal, see \cite[Rem.~4.11]{CPP-effective}.
We refer to \cite{FH23} for the background on log-concavity of the Ehrhart polynomials of
integral polytopes.  Let us emphasize that although the original proof of Theorem~\ref{t:basic-OP-log-concave}
is via direct injection, this approach does not extend to Theorem~\ref{t:basic-OP-log-concave-strict}
which is proved using the FKG inequality (see~$\S$\ref{ss:proof-FKG}).  Note that another direct
combinatorial proof of  \eqref{eq:basic-OP-log-concave} is given in \cite[Thm~5]{DDP}
(see also \cite[$\S$4.4]{Day84}).

\begin{thm}[{\rm \defn{$q$-log-concavity}~\cite[Thm~1.5]{CPP-effective}}{}] \label{t:basic-OP-q}
Let \ts $P=(X,\prec)$ \ts be a poset with \ts $|X|=n$ \ts elements.
Then, for every integer \ts $t \geq 2$, we have:
\begin{equation}\label{eq:basic-OP-q}
\Omega_q(P,t)^2  \  \geqslant_q \ \Omega_q(P,t+1) \. \cdot \. \Omega_q(P,t-1),
\end{equation}
where the inequality holds coefficient-wise as a polynomial in~$q$.
\end{thm}

We also have a multivariate version of this result:

\begin{thm}[{\rm \defn{$\bq$-log-concavity}~\cite[Cor.~9.5]{CP-multi}}{}] \label{t:basic-OP-bq}
Let \ts $P=(X,\prec)$ \ts be a poset with \ts $|X|=n$ \ts elements.
Then, for every integer \ts $t \geq 2$, we have:
\begin{equation}\label{eq:basic-OP-bq}
\Omega_\bq(P,t)^2  \  \geqslant_\bq \ \Omega_\bq(P,t+1) \. \cdot \. \Omega_\bq(P,t-1),
\end{equation}
where the inequality holds coefficient-wise as a polynomial in \ts $\bq=(q_1,\ldots,q_n)$.
\end{thm}

Again, the proof of both theorems uses the FKG inequality
(see~$\S$\ref{ss:proof-FKG}).  We conclude with a special case of an open
problem by Ferroni and Higashitani stated in the language of
Ehrhart polynomials of integral polytopes \cite[Question~5.10]{FH23},
which asks if negative values of Ehrhart polynomials of integral
polytopes are log-concave.

\begin{thm}[{\rm \defn{negative log-concavity} ~\cite[Thm~3]{DDP}}{}]
\label{t:basic-OP-strict-log-concave}
	Let \ts $P=(X,\prec)$ \ts be a poset with \ts $|X|=n$ \ts elements.
	Then, for all integer \ts $t\le -2$, we have:
	\begin{equation}\label{eq:basic-OP-strict-log-concave}
	 \Omega(P,t)^2  \ \geq \  \Omega(P,t+1) \  \Omega(P,t-1).
	\end{equation}
\end{thm}

Note that for negative \ts $t\in \zz$, the number
\ts $|\Omega(P,t)|$ \ts counts the number of integral points in the
relative interior of the expansion of the order polytope~$\cO(P)$.
We refer to \cite{BS18} and \cite[$\S$4.6]{Sta-EC} for an extensive
discussion of this connection.

\smallskip

\subsection{Monotonicity}\label{ss:basic-KS}
The following conjecture is mentioned in the solution to Exc.~3.163(b)
in~\cite{Sta-EC}, see also \cite[Conj.~4.12]{CPP-effective}.

\begin{conj}[{\rm \defn{Kahn--Saks monotonicity conjecture}}{}]\label{conj:KS-mon}
Let \ts $P=(X,\prec)$ \ts be a poset with \ts $|X|=n$ \ts elements.  Then, for all
integer \ts $t\ge 1$, we have:
	\begin{equation}\label{eq:basic-KS-mono}
\frac{\Omega(P,t)}{t^n} \, \ge \, \frac{\Omega(P,t+1)}{(t+1)^n} \,.
	\end{equation}
\end{conj}

The conjecture holds trivially when \ts $\Omega(P,t)$ \ts has positive coefficients.
We refer to \cite{LT19} for some explicit examples of order polynomials with
negative coefficients.  Stanley noted that the conjecture holds for \ts $t$ \ts large enough,
since the coefficient \. $[t^{n-1}] \. \Omega(P,t)>0$.
The proof is based on an elegant direct injection.  The following result
lend further support of the conjecture:


\begin{prop}[{\rm \cite[Prop.~4.14]{CPP-effective}}{}]
\label{p:KS-mono-multiple}
Let \ts $P=(X,\prec)$ \ts be a poset with \ts $|X|=n$ \ts elements.
Then, for all integer \ts $k, \ts t\ge 1$, we have:
$$
\frac{1}{t^n} \ \Omega(P,t) \, \geq \, \frac{1}{(kt)^n} \ \Omega(P,kt)\ts.
$$
Moreover, we have:
$$\Omega(P,t)\ts k^n \. - \. \Omega(P,kt) \. \in \. \SP\ts.$$
\end{prop}


Here the proof is elementary, via direct injection, and does not extend
to \ts $k\notin \nn$.
One way to approach the Kahn--Saks monotonicity conjecture is
to prove the following inductive inequality:


\begin{conj}[{\rm \cite[Conj.~4.17]{CPP-effective}}{}]\label{conj:KS-FKG}
	Let \ts $P=(X,\prec)$ \ts be a finite poset, and let
\. $t \geq k\ge 1$ \. be positive integers.
	Then there exists \ts $x \in X$, such that
	\begin{equation}\label{eq:KS-FKG}
	  \frac{\Omega(P,k)}{\Omega(P,t)} \ \geq \  \frac{k\, \Omega(P- x,k)}{t\, \Omega(P- x,t)}\,.
	\end{equation}
\end{conj}


\begin{prop}[{\rm \cite[Prop.~4.18]{CPP-effective}}{}] \label{p:basic-KS-mono-two-conj}
Conjecture~\ref{conj:KS-FKG} \ts implies \ts
Conjecture~\ref{conj:KS-mon}.
\end{prop}


We conclude with a curious counterpart of \eqref{eq:basic-KS-mono}:

\begin{thm}[{\rm \cite[Thm~4.8]{CPP-effective}}{}] \label{t:KS-mono-reverse-width}
Let \ts $P=(X,\prec)$ \ts be a finite poset of width~$w$. Then, for all integer \ts $t\ge 1$,
we have:
\begin{equation}\label{eq:basic-KS-mono-width}
\frac{\Omega(P,t)}{t^w} \, \le \, \frac{\Omega(P,t+1)}{(t+1)^w} \,.
\end{equation}
\end{thm}


This is asymptotically trivial, but not obvious for small~$t$ and large \ts $w \le n$.
When \ts $w=n$, we have $P=\rA_n\ts$, \ts $\Omega(P,t)= t^n$,
and both \eqref{eq:basic-KS-mono} and
\eqref{eq:basic-KS-mono-width} are equalities.
We note that proofs of both Proposition~\ref{p:basic-KS-mono-two-conj} and
Theorem~\ref{t:KS-mono-reverse-width} use the FKG inequality (see~$\S$\ref{ss:proof-FKG}).


\begin{conj}[{\rm Chan--Panova, 2023}{}]\label{conj:mono-CP}
Let \ts $P=(X,\prec)$ \ts be a finite poset that is not a chain.
Then, there exists elements \ts $x,y\in X$, s.t.\  \ts $y$~covers~$x$, and
for all positive integers \. $t \geq k\ge 1$, we have:
\begin{equation}\label{eq:mono-CP}
	  \frac{\Omega(Q,k)}{\Omega(P,k)} \ \leq \  \frac{\Omega(Q,t)}{\Omega(P,t)}\,.
\end{equation}
where \ts $Q=(X,\prec')$ \ts is a poset obtained from~$P$ by removing \ts $\{x\prec y\}$.
\end{conj}

By analogy with Proposition~\ref{p:basic-KS-mono-two-conj},
Conjecture~\ref{conj:mono-CP} \ts implies \ts Conjecture~\ref{conj:KS-mon}.

\medskip



\section{Sidorenko type inequalities}\label{s:sid}

\subsection{Sidorenko inequality} \label{ss:sid-original}
The following result is a poset theoretic version of polyhedral duality.

\begin{thm}[{\rm \defn{Sidorenko inequality}~\cite[Thm~11]{Sid}}]\label{t:sid}
Let \. $P=(X,\prec)$ \. and \. $Q=(X,\prec')$ \. be two posets on the same set
with \. $|X|=n$ \. elements.  Suppose
\begin{equation}\label{eq:sid-chain-condition}
\bigl|C \cap C'\bigr| \. \le \. 1 \quad \text{for all} \ \ C\in \cC(P), \ C' \in \cC(Q).
\end{equation}
Then:
\begin{equation}\label{eq:sid}\tag{Sid}
e(P) \. e(Q)\, \ge \, n!
\end{equation}
Moreover, \eqref{eq:sid} is an equality \. \underline{if and only if} \. $P$ \ts
is series-parallel.
\end{thm}

The assumption~\eqref{eq:sid-chain-condition} can also be written in
terms of comparability graphs: \.  $\Ga(P) \subseteq  \ov{\Ga(Q)}$,
see~$\S$\ref{ss:ex-perm} for examples.
Note that testing if a poset is series-parallel is in~$\poly$ since they are $N$-free,
see also \cite{VTL82}.  There are several proofs of Theorem~\ref{t:sid}.  The original
proof uses combinatorial optimization (see~$\S$\ref{ss:proof-comb-opt}).  More recent
proofs use direct surjection \cite{GG20} (cf.~\cite{MPP-phi}), and injection
\cite{CPP-effective,GG22}.  In particular, we have the following:

\begin{thm}[{\rm \cite[Thm~1.14]{CPP-effective} and \cite[Thm~3.8]{GG22}}{}]\label{t:sid-SP}
The defect of the Sidorenko inequality \eqref{eq:sid} is in $\ts \SP$.
\end{thm}

\begin{rem}[{\rm \cite{BBS99}}{}]\label{r:sid-correlation}
In condition of Theorem~\ref{t:sid}, suppose \ts $\Ga(P) =  \ov{\Ga(Q)}$.
Then one can view Sidorenko's inequality \eqref{eq:sid} as a negative correlation
result for uniform bijections \ts $g: X\to [n]$.  Indeed, note that \ts $e(P\cap Q)=1$.
Thus, we have:
$$\bP\big(g\in \Ec(P)\cap \Ec(Q)\big)  \, = \, \frac{1}{n!} \, \le \, \frac{e(P)\cdot e(Q)}{(n!)^2}  \, = \,
\bP\big(g\in \Ec(P)\big) \. \cdot \. \bP\big(g\in \Ec(Q)\big).
$$
\end{rem}

\smallskip

\subsection{Generalizations of the Sidorenko inequality} \label{ss:sid-gen}
The assumption~\eqref{eq:sid-chain-condition} in the theorem can be relaxed
to give the following result.

\begin{thm}[{\rm \cite[Thm~1.7]{CPP-effective}}{}]\label{t:sid-gen}
Let \. $P=(X,\prec)$ \. and \. $Q=(X,\prec')$ \. be two posets on the same set
with \. $|X|=n$ \. elements.  Suppose
$$
\bigl|C \cap C'\bigr| \. \le \. k \quad \text{for all} \ \ C\in \cC(P), \ C' \in \cC(Q).
$$
Then:
\begin{equation}\label{eq:sid-gen}
e(P) \. e(Q)\, \ge \,  \frac{n!}{k^{n-k}\. k!} \..
\end{equation}
\end{thm}

The following result is a natural generalization of the
Sidorenko inequality.

\begin{thm}[{\rm Sidorenko~\cite[Thm~14]{Sid}}]\label{t:sid-BT}
Let \. $P_1=(X,\prec_1)$, \ldots \., $P_k=(X,\prec_k)$ \. and
$Q=(X,\prec')$ \. be  posets on the same set.  Suppose
\begin{equation}\label{eq:sid-int-condition}
\bigcap_{i=1}^k \. \Ga(P_i) \ \subseteq \, \Ga(Q).
\end{equation}
Then:
\begin{equation}\label{eq:sid-BT}
e(P_1) \. \cdots \. e(P_k)\, \ge \, e(Q).
\end{equation}
\end{thm}

For example, take posets \ts $P_1,P_2$ \ts which satisfy \. $\Ga(P_1) \subseteq \ov{\Ga(P_2)}$.
and let \ts $Q\gets A_n$\ts. Then \eqref{eq:sid-int-condition} gives
\. $e(P_1) \. e(P_2) \ge e(Q)=n!$.  In other words, Theorem~\ref{t:sid-BT}
implies Theorem~\ref{t:sid}.

\smallskip

\subsection{Reverse Sidorenko inequality} \label{ss:sid-reverse}
It may come as a surprise that the lower bound in the Sidorenko inequality
is always sharp up to a simple exponential factor.  Formally, we have
the following:

\begin{thm}[{\rm \defn{Reverse Sidorenko inequality}~\cite{BBS99}}{}]\label{t:sid-BBS}
Let \. $P=(X,\prec)$ \. and \. $Q=(X,\prec')$ \. be two posets on the same set
with \. $|X|=n$ \. elements which satisfy \ts $\Ga(P) =  \ov{\Ga(Q)}$.
Denote by \ts $\om_n := \vol(B_n)$ \ts be the volume of a unit ball in \ts $\rr^n$.
Then:
\begin{equation}\label{eq:Sid-BBS}
e(P) \. e(Q)\, \le \, \frac{(n! \. \om_n)^2}{4^n}\,.
\end{equation}
\end{thm}

The Stirling formula and the asymptotics for \ts $\om_n$ \ts show that the Sidorenko
inequality is asymptotically sharp:
\begin{equation}\label{eq:Sid-BBS-asy}
n! \, \le \, e(P) \. e(Q)\, \le \, n! \. \bigl(\tfrac{\pi}{2}\bigr)^n \. O\big(\tfrac{1}{\sqrt{n}}\big).
\end{equation}

\begin{question}\label{q:sid-BBS}{\rm
Denote by \ts $\mu(n)$ \ts the maximal value of the product \. $e(P) \. e(Q)$ \, over all
posets $P,Q$ on $n$ elements which satisfy~\eqref{eq:sid-chain-condition}.
In~\cite{BBS99}, the authors ask to determine
$$\vk \,:=\, \limsup_{n\to \infty} \bigg(\frac{\mu(n)}{n!}\bigg)^{1/n}.
$$
They observe that \ts $\vk> 1.123$ \ts and conjecture that \ts $\vk<1.2$.
The upper bound \ts $\vk \le \frac{\pi}{2} \approx 1.571$ \ts given by \eqref{eq:Sid-BBS-asy},
remains the best known asymptotic upper bound.
}\end{question}

\begin{rem}\label{r:sid-mixed}
\defng{Mixed  Sidorenko inequality} \ts is another generalization of
Sidorenko's inequality to \defng{double posets} \ts is given in
\cite[Thm~6.2]{AASS}.  It would be interesting to see if this
inequality has a direct injective proof.
\end{rem}



\medskip

\section{Bj\"orner--Wachs type inequalities}\label{s:BW}

\smallskip

\subsection{Bj\"orner--Wachs inequality}\label{ss:BW-ineq}
The following inequality is elementary, but surprisingly rich in generalizations
and applications:


\begin{thm}[{\rm \defn{Bj\"orner--Wachs inequality}~\cite[Thm~6.3]{BW89}}{}]\label{t:ineq-BW}
Let \ts $P=(X,\prec)$ \ts be a poset with \ts $|X|=n$ \ts elements.
We have:
\begin{equation}\label{eq:BW} \tag{BW}
e(P) \ \geq  \  \. n! \.\cdot \. \prod_{x \in X} \, \frac{1}{\be(x)}\,.
\end{equation}
Moreover, the equality holds \. \underline{if and only if}  \. $P$ \ts is a forest.
Additionally, the defect of the \eqref{eq:BW} is in~$\ts \SP$.
\end{thm}

The original proof uses a direct injection.
This inequality was popularized by Stanley, who stated it without proof
or a reference in \cite[Exc.~3.57]{Sta-EC}.\footnote{Richard Stanley
informed us that he indeed took it from~\cite{BW89} (personal communication, March 27, 2022).}
Unaware of the provenance, in~\cite{HP08}, Hammett and Pittel gave a laborious proof in
the language of geometric probability.  Note that \eqref{eq:BW}  is asymmetric,
i.e.\ not invariant under poset duality, leading to the following:


\begin{cor}\label{c:BW-anti-hooks}
Let \ts $P=(X,\prec)$ \ts be a forest.  Then:
\begin{equation}\label{eq:BW-anti-hooks}
  \prod_{x \in X} \al(x) \, \ge \, \prod_{x \in X} \be(x)\ts.
\end{equation}
\end{cor}

This inequality follows immediately from Theorem~\ref{t:ineq-BW}, since \ts $\al(x)$ \ts and \ts $\be(x)$ \ts
switch role in dual posets:
$$
 \prod_{x \in X} \be(x) \ = \ \frac{n!}{e(P)} \ = \ \frac{n!}{e(P^\ast)} \ \le \ \prod_{x \in X} \al(x)\..
$$
The corollary was also proved combinatorially and generalized in~\cite{PPS20}.
The proof uses \defng{Karamata's inequality}, which does not lead to an injection
(cf.\ \cite[$\S$7.5]{IP22}).

\begin{rem}\label{r:BW-probab}
Theorem~\ref{t:ineq-BW} is a correlation inequality in the following sense.  Let \ts
$\Si(X)$ \ts denote the set of bijections \. $\si: X\to [n]$.  By definition, we have:
$$
\bP\bigl(\si(x) < \si(y) \ \, \forall \, x,y\in X, \, x\prec y\bigr) \, = \, \frac{e(P)}{n!}\,,
$$
where \ts $\bP$ is a uniform measure on \ts $\Si(X)$.
Denote by \ts $\cA_x \subseteq \Si(X)$ \ts the event that \ts $\si(x) \le \si(y)$
\ts for all \ts $x \prec y$.  Then \eqref{eq:BW} says that every collection
of \ts $\cA_x$ \ts is mutually positively correlated.  The second part implies that
for forests these events are mutually independent.
\end{rem}

\smallskip

\subsection{Reiner's inequality}\label{ss:BW-Reiner}
The following result is a natural $q$-analogue of the Bj\"orner--Wachs
inequality \eqref{eq:BW},
but was discovered only recently:

\begin{thm}[{\rm \defn{Reiner's inequality}~\cite[Thm~5.1]{CPP-effective} and \cite[Thm~6.2]{BW89}}{}] \label{t:ineq-BW-OP-Reiner}
Let \ts $P=(X,\prec)$ \ts be a poset with \ts $|X|=n$ \ts elements. Then:
\begin{equation}\label{eq:ineq-BW-OP-Reiner}
\Om_q(P) \ \geqslant_q  \  \. \prod_{x \ts \in X} \, \frac{1}{1-q^{\be(x)}}\,,
\end{equation}
where the inequality between two power series is coefficient-wise.
Moreover, this inequality is an equality \. \underline{if and only if} \.
\ts $P$ \ts is a forest.  Additionally,
the coefficient \. $[q^m]$ \. of the defect of this inequality is in~$\SP$, where
\ts $m$ \ts is given in binary.
\end{thm}

Reiner's inequality \eqref{eq:ineq-BW-OP-Reiner} was proved by Reiner by
remarkably short and direct proof, see below.  It was published by the
authors in \cite{CPP-effective}. The equality part was proved
in the original Bj\"orner--Wachs paper.  Multiplying both sides of
\eqref{eq:ineq-BW-OP-Reiner} by \. $(1-q)(1-q^2)\cdots (1-q^n)$ \.
and using the equality \eqref{eq:Sta-PP}, we conclude:

\begin{cor} \label{c:ineq-BW-q}
For all \. $0< q <1$, we have:
\begin{equation}\label{eq:ineq-BW-q}
e_q(P)  \ \geq  \  (1-q)(1-q^2)\cdots (1-q^n) \. \prod_{x \ts \in X} \, \frac{1}{1-q^{\be(x)}}\,.
\end{equation}
\end{cor}

Taking the limit \. $q\to 1-$, gives the
Bj\"orner--Wachs inequality \eqref{eq:BW}.  This is probably
the shortest and the most conceptual proof of~\eqref{eq:BW}.

\begin{proof}[Proof of Theorem~\ref{t:ineq-BW-OP-Reiner}]
Interpret the RHS of~\eqref{eq:ineq-BW-OP-Reiner} as the GF for maps \. $g\in \cP(P)$ \. which
are obtained as a nonnegative integer linear combination of characteristic
functions of upper order ideals in poset~$P\ts$:
$$g \.= \, \sum_{x\in X} \, m(x) \.\chi(x\!\uparrow)\., \ \ \. \text{where} \ \ m(x) \in \nn \ \ \text{for all} \ \ x\in X.
$$
Note that characteristic functions \ts $\chi(x\!\uparrow)$ \ts are linearly independent
because in the standard basis \ts $\{\chi(y) \. : \. y\in X\}$, the transition matrix is unitriangular.
Since \. $\sum_{x\in X} g(x) = \sum_{x\in X} \. m(x) \. \be(x)$, the result follows immediately from \eqref{eq:OP-q-def}.
\end{proof}

\smallskip

\subsection{Order polynomials version}\label{ss:BW-OP}
The following is the extension of the Bj\"orner--Wachs inequality for
order polytopes:

\begin{thm}[{\rm \cite[Thm~1.2]{CPP-effective}}{}]
\label{t:ineq-BW-OP-lower-bound}
Let \ts $P=(X,\prec)$ \ts be a poset with \ts $|X|=n$ \ts elements, and
let \ts $r=|\min(P)|$ \ts be the number of minimal elements.  Then, for all \ts $t\in \nn$, we have:
\begin{equation}\label{eq:ineq-BW-OP-lower-bound}
\Omega(P,t) \ \geq \ t^r \ts (t+1)^{n-r} \, \prod_{x \ts \in \ts X} \. \frac{1}{\be(x)} \..
\end{equation}
\end{thm}


By \eqref{eq:OP-asy}, the inequality \eqref{eq:ineq-BW-OP-lower-bound} implies \eqref{eq:BW}.
The following result shows that \eqref{eq:ineq-BW-OP-lower-bound} can be slightly improved
if Conjecture~\ref{conj:KS-mon} holds:

\begin{thm} [{\rm \cite[Thm~4.13]{CPP-effective}}{}]\label{t:ineq-BW-OP-lower-bound-conj}
	Let \. $P=(X,\prec)$, let \. $\min(P)\subseteq X$ \. be the subset of maximal elements,
	and let \. $r:=|\min(P)|$ \. be the number of maximal elements. If Conjecture~\ref{conj:KS-mon}
	holds, then we have:
	\begin{equation}\label{eq:ineq-BW-OP-lower-bound-conj}
		\Omega(P,t) \ \geq \
		t^r \prod_{x \ts \in\ts X\sm \min(P)} \. \bigg(\frac{t}{\be(x)}\,+\,\frac{1}{2}\bigg).
	\end{equation}
\end{thm}

It would be interesting to find an unconditional proof of this inequality.
Both Theorems~\ref{t:ineq-BW-OP-lower-bound} and~\ref{t:ineq-BW-OP-lower-bound-conj}
were proved using the FKG inequality (see~$\S$\ref{ss:proof-FKG}).

\medskip



\section{Fishburn type inequalities}\label{s:Fish}

\subsection{Two minimal elements} \label{ss:Fish-baby}
We start with the following special case which is already interesting
and hard to prove.

\begin{thm}[{\rm see \cite[Thm~1.1]{CP-corr}}{}]
\label{t:Fish-baby}
Let \ts $P=(X,\prec)$ \ts be a poset with \ts $|X|=n> 2$ \ts elements.
Let \ts $x,y\in \min(X)$ \ts  be distinct minimal elements of~$P$. Then:
\begin{equation}\label{eq:Fish-baby}
  \frac{n}{n-1} \ \leq \ \frac{e(P) \.\cdot\. e(P-x-y)}{e(P-x) \.\cdot\. e(P-y)} \ \leq \ 2 \ts.
\end{equation}
\end{thm}

\smallskip

This correlation inequality is the most natural
and the simplest to state.
The lower bound in~\eqref{eq:Fish-baby} is a special case of the
Fishburn's inequality~\eqref{eq:Fish-ineq} below, while the upper
bound is a special case of~\eqref{eq:ineq-delete} below.
%
Note that the lower bound is tight for \ts $P=A_n$ \ts and
the upper bounds is tight for the linear sum \ts $P=A_2\oplus C_{n-2}$\..

\begin{rem}\label{r:Fish-baby-prob}
Correlation inequalities are best understood in probabilistic notations.
The inequality~\eqref{eq:Fish-baby} can be rewritten as
\begin{equation}\label{eq:Fish-baby-prob}
  \frac{n}{n-1} \ \leq \ \frac{\Pb[f(x)=1, f(y)=2]}{\Pb[f(x)=1] \.\cdot\. \Pb[f(y)=1]}  \ \leq \ 2,
\end{equation}
where the probability \ts $\Pb$ \ts is over the uniform random
linear extension \ts $f\in\Ec(P)$.  The asymmetry in the
numerator is an illusion, since \. $\Pb[f(x)=1, f(y)=2] = \Pb[f(x)=2, f(y)=1]$.
\end{rem}

\begin{op}\label{op:Fish-baby}
For many examples of large posets, the lower bound in~\eqref{eq:Fish-baby} is tight.
Can one improve the upper bound for a large natural class of posets?
\end{op}

\smallskip

\subsection{Fishburn's inequality} \label{ss:Fish-ineq}
Let \ts $P=(X,\prec)$ \ts be a poset with \ts $|X|=n$ \ts elements.  Denote
$$
\en(P) \, := \, \frac{e(P)}{n!} \, = \, \bP\big(f\in \Ec(P)\big),
$$
where the probability is over uniform bijections \ts $f:X\to [n]$.

Let  \ts $A \subseteq X$ \ts be a subset of the ground set.  By a small abuse of
notation, denote by \. $e(A)$ \. the number of linear extensions
of the induced subposet \ts $P|_A=(A,\prec)$.

\begin{thm}[{\rm \defn{Fishburn's inequality}~\cite[Lemma,~p.~130]{Fis84}}{}]
\label{t:Fish}
Let \ts $P=(X,\prec)$ \ts be a finite poset, and let \ts $A,B \subseteq X$ \ts be lower ideals of~$P$.
Then:
\begin{equation}\label{eq:Fish-ineq}
	  \en(A \cup B) \.\cdot\. \en(A\cap B) \ \geq \ \en(A) \.\cdot\. \en(B)\ts.
\end{equation}
\end{thm}

Taking \ts $A:=X-x$ \ts and \ts $B:=X-y$ \ts gives \eqref{eq:Fish-baby}.
The original proof of Fishburn's inequality uses the AD inequality
(see~$\S$\ref{ss:proof-FKG}), and the approach by Shepp \cite{She80,She82}.
In fact, Theorem~\ref{t:Fish} follows easily from \cite[Thm~2]{She80},
as explained in \cite[Lem.~10]{Bri88}.
The following is a self-dual generalization.

\begin{thm}[{\rm \defn{generalized Fishburn's inequality}~\cite[Thm~3.4]{CP-multi}}{}]\label{t:Fish-CP}
Let \ts $P=(X,\prec)$ \ts be a finite poset. Let \. $A,B \subseteq X$ \. be lower ideals,
and let \. $C,D \subseteq X$ \. be upper ideals of~$P$, such that \.
$A \cap C = B \cap D = \varnothing$.
Then:
\begin{equation}\label{eq:Fish-ineq-CP}
\en(X-V) \.\cdot\. \en(X-W) \ \geq \ \en(X-A-C) \.\cdot\. \en(X-B-D)\ts,
\end{equation}
where \. $V:=(A \cap B) \cup (C \cup D)$ \. and \. $W:=(A \cup B) \cup (C \cap D)$.
\end{thm}

The proof of this generalization also uses the AD~inequality.

\smallskip

\subsection{Lam--Pylyavskyy extension} \label{ss:Fish-ineq-LP}
To simplify the notation, denote by \ts $\Om(A,t)$ \ts the order polynomial of the
induced subposet \ts $P|_A$\ts. Define \ts $\Om_q(A,t)$,  \ts $\Om_\bq(A,t)$ \ts and \ts $\Phi_\bz(A,t)$ \ts
in a similar way.

\begin{thm}[{\rm Lam--Pylyavskyy~\cite[Thm~3.6]{LP07}}{}]
\label{t:LP}
Let \ts $P=(X,\prec)$ \ts be a finite poset, and let \ts $A,B \subseteq X$ \ts be lower ideals of~$P$.
Then, for all integer \ts $t\ge 1$, we have:
\begin{equation}\label{eq:LP-ineq}
	  \Om(A \cup B,t) \.\cdot\. \Om(A\cap B,t) \ \geq \ \Om(A,t) \.\cdot\. \Om(B,t).
\end{equation}
Additionally, the defect of this inequality is in~$\SP$.
\end{thm}

The authors proved this result by an explicit injection, which they generalize
in several different ways.  Some of these generalization are natural from the
algebraic combinatorics point of view, but some are natural and apply to all
posets.

\begin{thm}[{\rm Lam--Pylyavskyy~\cite[Prop.~3.7]{LP07}}{}] \label{t:LP-multi}
Let \ts $P=(X,\prec)$ \ts be a finite poset, and let \ts $A,B \subseteq X$ \ts be lower ideals of~$P$.
Then, for all integer \ts $t\ge 1$, we have:
\begin{equation}\label{eq:LP-ineq-q}
	  \Om_q(A \cup B,t) \.\cdot\. \Om_q(A\cap B,t) \ \geqslant_q \ \Om_q(A,t) \.\cdot\. \Om_q(B,t),
\end{equation}
where the inequality holds coefficient-wise as a polynomial in~$\ts q$.
Moreover, we have:
\begin{equation}\label{eq:LP-ineq-multi}
	  \Phi_\bz(A \cup B,t) \.\cdot\. \Phi_\bz(A\cap B,t) \ \geqslant_\bz \ \Phi_\bz(A,t) \.\cdot\. \Phi_\bz(B,t),
\end{equation}
where the inequality holds coefficient-wise as a polynomial in~$\ts \bz=(z_0,z_1,\ldots)$.
Additionally, the defect of this inequality is in~$\ts  \SP$.
\end{thm}

Here the second part \eqref{eq:LP-ineq-multi} implies the first part \eqref{eq:LP-ineq-q} by
the substitution in \eqref{eq:OP-rK-exp}.  Letting \ts $t\to \infty$, using the equality \eqref{eq:Sta-PP},
and multiplying both sides by an appropriate product of the type \.
$(1-q)(1-q^2)\cdots$ \. gives a $q$-analogue of Fishburn's inequality in the
style of Corollary~\ref{c:ineq-BW-q}$\ts:$

\begin{cor}[{\rm \defn{$q$-Fishburn's inequality}}{}] \label{c:LP-ineq-maj-q}
Let \ts $P=(X,\prec)$ \ts be a finite poset, and let \ts $A,B \subseteq X$ \ts be lower ideals of~$P$.
Then, for all \. $0< q <1$, we have:
\begin{equation}\label{eq:LP-ineq-maj-q}
\frac{e_q(A\cup B) \cdot e_q(A\cap B)}{e_q(A) \cdot e_q(B)}   \ \geq  \
\frac{|A\cup B|!_q \cdot |A\cap B|!_q}{|A|!_q\cdot |B|!_q}
\,.
\end{equation}
\end{cor}

Taking the limit \ts $q\to 1-$ \ts gives back Fishburn's inequality \eqref{eq:Fish-ineq}.
The original proof of Theorem~\ref{t:LP-multi} uses an explicit injection.  The proof
in \cite{CP-multi} uses generalizations of the AD~inequality (see~$\S$\ref{ss:proof-FKG}).

\smallskip

\subsection{Self-dual extension of Fishburn's inequality}\label{ss:main-dual}
Note that Fishburn's inequality is defined to be asymmetric up to duality.
The following generalization is self-dual.

\begin{thm}[{\rm \cite[Thm~3.4]{CP-multi}}{}] \label{t:LP-dual}
Let \ts $P=(X,\prec)$ \ts be a finite poset. Let \ts $A,B \subseteq X$ \ts be lower ideals,
and let \ts $C,D \subseteq X$ \ts be upper ideals of~$P$, such that \.
$A \cap C = B \cap D = \varnothing$.
Then:
\begin{equation}\label{eq:LP-dual}
\en(X-V) \.\cdot\. \en(X-W) \ \geq \ \en(X-A-C) \.\cdot\. \en(X-B-D)\,.
\end{equation}
where \. $V:=(A \cap B) \cup (C \cup D)$ \. and \. $W:=(A \cup B) \cup (C \cap D)$.
\end{thm}

Fishburn's inequality \eqref{eq:Fish-ineq} is a special case of the theorem when \ts $C=D= \emp$.
Curiously, we are able to prove both multivariate analogues in this setting:

\begin{thm}[{\rm \cite[Thms~4.9~and~4.10]{CP-multi}}{}] \label{t:main-OP-multi}
Let \ts $\cP=(X,\prec)$ \ts be a finite poset. Let \ts $A,B \subseteq X$ \ts be lower ideals,
and let \ts $C,D \subseteq X$ \ts be upper ideals of~$\cP$, such that \.
$A \cap C = B \cap D = \varnothing$.  Fix and integer \ts $t \ge 1$.
Then:
\begin{equation}\label{eq:LP-dual-bq}
\Om_\bq(X-V,t) \.\cdot\. \Om_\bq(X-W,t) \ \geqslant_\bq \ \Om_\bq(X-A-C,t) \.\cdot\. \Om_\bq(X-B-D,t)\ts,
\end{equation}
where \. $V:=(A \cap B) \cup (C \cup D)$ \. and \. $W:=(A \cup B) \cup (C \cap D)$,
and the inequality holds coefficient-wise as a polynomial in \. $\bq=(q_1,q_2,\ldots)$.
Similarly, we have:
\begin{equation}\label{eq:LP-dual-bz}
		\Phi_\bz(X-V,t) \.\cdot\. \Phi_\bz(X-W,t) \ \geqslant_\bz \ \Phi_\bz(X-A-C,t) \.\cdot\. \Phi_\bz(X-B-D,t)\ts,
\end{equation}
and the inequality holds coefficient-wise as a polynomial in \. $\bz=(z_0,z_1,\ldots)$.
\end{thm}

The proof of Theorems~\ref{t:LP-dual} and~\ref{t:main-OP-multi} also
uses generalizations of the AD~inequality (see~$\S$\ref{ss:proof-FKG}),
and thus fundamentally non-injective.  This leaves open whether Lam--Pylyavskyy
injective arguments can be modified to answer the following:

\begin{question}\label{q:LP-SP}
Is the deficit of inequalities \eqref{eq:LP-dual-bq} and \eqref{eq:LP-dual-bz} in \ts $\SP$?
\end{question}



\medskip

\section{Correlation inequalities}\label{s:corr}

\subsection{GYY inequality} \label{ss:power-GYY}
Let \ts $P=\rC + \rC'$ \ts be a disjoint sum of two chains with \ts $\ell$ \ts
and \ts $(n-\ell)$ \ts elements, respectively, where
\. $\rC = \{u_1\prec \ldots \prec u_\ell\}$ \. and
\. $\rC'=\{v_1\prec \ldots \prec v_{n-\ell}\}$.  

Denote \ts $\La:= [\ell]\times [n-\ell]$.
For all \ts $S\subseteq \La$, let \.
$A_{S} \ts := \ts (X, \prec_S)$ \. be a poset with the relations
\ts $u_i \prec_S v_j$ \ts for all \ts $(i,j) \in S$.
Note that posets \ts $A_S\ts, A_{T}$ \ts are consistent
with each other and with~$P$, for all $S, T\subseteq \La$.


\begin{thm}[{\rm \defn{Graham--Yao--Yao inequality}~\cite[Thm~1]{GYY80}}{}]\label{t:GYY}
Let \ts $S, T\subseteq \La$. Then:
	\begin{equation}\label{eq:GYY}\tag{GYY}
	e(P\cap A_S \cap  A_T) \, e(P)  \ \geq \  e(P\cap A_S)  \, e(P\cap A_T)\ts.
	\end{equation}
Additionally, 
the defect of this inequality is in~$\ts \SP$.
\end{thm}


The theorem was originally proved by Graham, Yao and Yao in~\cite{GYY80} using
a lattice paths argument (see~$\S$\ref{ss:proof-lattice}).  A proof using the
FKG inequality was given in \cite{KS81}, and soon after in \cite{She80},
see the generalization below.


\begin{rem}\label{r:GYY-probab}
The result simplifies in a probabilistic setting.
Denote by \ts $\cA_S\subseteq \Ec(P)$ \ts the event that a linear
extension \ts $f\in \Ec(P)$ \ts satisfies relations in~$\prec_S$\ts.
Then \eqref{eq:GYY} can be rewritten as a positive correlation:
\begin{equation}\label{eq:GYY-probab}
	\Pb(\cA_S \cap  \cA_T) \ \geq \  \Pb(\cA_S) \, \Pb(\cA_T)\ts,
\end{equation}
where \ts $\Pb$ \ts is the uniform measure on \ts $\Ec(P)$.  Note that
\ts $\Pb(\cA_S) =0$ \ts if \ts $\cA_S$ \ts and \ts $P$ \ts are
inconsistent.
\end{rem}

\smallskip

\subsection{Shepp's inequality} \label{ss:ineq-Shepp}
Let \ts $P = Q + Q'$, where \ts $Q=(U,\prec)$ \ts and \ts $Q=(V,\prec')$.
Denote by \. $U=\{u_1,\ldots,u_\ell\}$ \ts and \ts $V=\{v_1,\ldots,v_{n-\ell}\}$ \.
the elements in these two posets.


\begin{thm}[{\rm \defn{Shepp's inequality}~\cite[Thm~2]{She80}}{}]\label{t:ineq-Shepp}
Let \ts $S, T\subseteq \La$. Then:
	\begin{equation}\label{eq:GYY-Shepp}
	e(P\cap A_S \cap  A_T) \, e(P)  \ \geq \  e(P\cap A_S)  \, e(P\cap A_T)\ts.
	\end{equation}
Similarly,
	\begin{equation}\label{eq:GYY-Shepp-dual}
	e(P\cap A_S \cap  A_T^\ast) \, e(P)  \ \leq \  e(P\cap A_S)  \, e(P\cap A_T^\ast)\ts,
	\end{equation}
where we use the notation \ts $e(P\cap R)=0$ \ts if posets $P$ and~$R$ are not consistent.
\end{thm}


This inequality was conjectured in \cite[p.~252]{GYY80}, which also mentioned that
it becomes false if \ts $P$ \ts has even one relation of the form \ts $u_i \prec v_j$.
In a note added in proof, the authors wrote that Theorem~\ref{t:ineq-Shepp} was
proved by Shepp~\cite{She80} using  ``an ingenious application of the FKG inequalities''
\cite[p.~258]{GYY80}.

\begin{op}\label{op:ineq-Shepp}
Prove of disprove:  \. the defect of inequality \eqref{eq:GYY-Shepp} is in~$\SP$.
\end{op}

If true, this would extend the second part of Theorem~\ref{t:GYY}.


\begin{thm}[{\rm \defn{Shepp's inequality for order polynomials}~\cite[Eq.~(2.12)]{She80}}{}]\label{t:ineq-Shepp-OP}
Let \ts $S, T\subseteq \La$. Then, for every integer \ts $t\ge 1$, we have:
\begin{equation}\label{eq:GYY-Shepp-OP}
	\Omega(P\cap A_S \cap  A_T, t) \. \cdot \. \Omega(P,t)  \ \geq \  \Omega(P\cap A_S\ts, t)  \. \cdot \. \Omega(P\cap A_T,t)\ts.
	\end{equation}
Similarly,
\begin{equation}\label{eq:GYY-Shepp-OP-dual}
	\Omega(P\cap A_S \cap  A_T^\ast, t) \. \cdot \. \Omega(P,t)  \ \leq \  \Omega(P\cap A_S\ts, t)  \. \cdot \. \Omega(P\cap A_T,t)\ts,
	\end{equation}
where we use the notation \ts $\Omega(P\cap R,t)=0$ \ts if posets $P$ and~$R$ are not consistent.
\end{thm}

By~\eqref{eq:OP-asy}, this theorem implies Shepp's inequality (Theorem~\ref{t:ineq-Shepp}).
The following $q$-analogue is the most general result in this direction.

\begin{thm}[{\rm \defn{$q$-analogue of Shepp's inequality} \cite[Thm~5.4]{CPP-effective}}{}]\label{t:ineq-Shepp-OP-q}
Let \ts $S, T\subseteq \La$. Then we have:
\begin{equation}\label{eq:GYY-Shepp-OP-q}
	\Omega_q(P\cap A_S \cap  A_T) \. \cdot \. \Omega_q(P)  \ \geqslant_q \  \Omega_q(P\cap A_S)  \. \cdot \. \Omega_q(P\cap A_T)\ts,
	\end{equation}
where the inequality holds coefficient-wise as a polynomial in~$\ts q$. More generally, for every integer \ts $t\ge 1$, we have:
\begin{equation}\label{eq:GYY-Shepp-OP-q-t}
	\Omega_q(P\cap A_S \cap  A_T, t) \. \cdot \. \Omega_q(P,t)  \ \geqslant_q \  \Omega_q(P\cap A_S\ts, t)  \. \cdot \. \Omega_q(P\cap A_T,t)\ts.
	\end{equation}
\end{thm}

The negative correlation version can be obtained in a similar way.  Theorem~\ref{t:ineq-Shepp-OP-q}
is proved using Bj\"orner's $q$-FKG inequality (see~$\S$\ref{ss:proof-FKG}).

\smallskip

\subsection{XYZ inequality} \label{ss:ineq-XYZ}
The following result is perhaps the most celebrated correlation inequality
for linear extensions:

\begin{thm}[{\rm \defn{XYZ inequality}~\cite{She82}}{}]\label{t:ineq-XYZ}
Let \ts $P=(X,\prec)$ \ts be a finite poset, and let \ts $x,y,z\in X$ \ts
be incomparable elements.  Denote \.
$P_{xy}:= P\cap \{x\prec y\}$, \. $P_{xz}:= P\cap \{x\prec z\}$ \. and \. $P_{xyz}:= P\cap \{x\prec y, x\prec z\}$.
Then:
\begin{equation}\label{eq:ineq-XYZ}\tag{XYZ}
e(P) \, e(P_{xyz}) \ \ge \
e(P_{xy}) \, e(P_{xz}).
\end{equation}
Moreover, for all integer \ts $t\ge 1$, we have:
\begin{equation}\label{eq:ineq-XYZ-OP}
\Om(P,t)  \. \cdot \.  \Om(P_{xyz},t) \ \ge \
\Om(P_{xy},t)  \. \cdot \.  \Om(P_{xz},t).
\end{equation}
\end{thm}

Inequality~\eqref{eq:ineq-XYZ} was first conjectured by Ivan~Rival
and Bill~Sands \cite[p.~806]{Riv82}.  Shepp's original proof of \eqref{eq:ineq-XYZ}
used the FKG inequality and goes through \eqref{eq:ineq-XYZ-OP}.  It was proved
by Fishburn \cite{Fis84}, that \eqref{eq:ineq-XYZ} is always strict.
A combinatorial (but not fully injective) argument was given in~\cite{BT02}.

\begin{conj}[{\rm \cite[Conj.~6.4]{Pak-OPAC}}{}]\label{conj:XYZ-SP}
The defect of \eqref{eq:ineq-XYZ} is not in~$\SP$.
\end{conj}

\smallskip

\begin{rem}\label{r:ineq-XYZ-multiple}
As with other correlation inequalities, the XYZ inequality is easier to understand in terms
of uniform random linear extensions \ts $f \in \Ec(P)$.  For incomparable
elements \ts $u,v\in X$, denote \. $\Ec_{uv} := \Ec(P_{uv})\ssu \Ec(P)$.  Then:
\begin{equation}\label{eq:XYZ-prob}
\bP(\Ec_{xy} \ts \cap \ts \Ec_{xz}) \, \ge \,  \bP(\Ec_{xy})
\. \cdot \. \bP(\Ec_{xz}).
\end{equation}

To simplify the notation, we write \. $\cA \vartriangle \cB$ \. if \.
the events \ts $\cA$ \ts and \ts $\cB$ \ts have \defn{positive correlation}:  \.
$\bP(\cA\cap\cB) \ge \bP(\cA)\cdot \bP(\cB)$.  Similarly,  we write \.
$\cA \, \triangledown \, \cB$ \. if \. these events have \defn{negative correlation}: \.
$\bP(\cA\cap\cB) \le \bP(\cA)\cdot \bP(\cB)$. In this notation, \eqref{eq:XYZ-prob} can be
written as \. $(\Ec_{xy}) \vartriangle (\Ec_{xz})$, or, equivalently, as  \.
$(\Ec_{xy})  \, \triangledown \,  (\Ec_{yz})$.

In \cite{Bri85}, Brightwell described all collections of inequalities for which
we have the analogue of \eqref{eq:XYZ-prob}.  Typical examples include:
$$
(\Ec_{xy} \cap \Ec_{uv}) \vartriangle (\Ec_{xv} \cap \Ec_{uy}), \quad
(\Ec_{xz} \cap \Ec_{yz}) \, \triangledown \, (\Ec_{zu} \cap \Ec_{zv}) \quad
\text{and} \quad (\Ec_{xw} \cap \Ec_{yw}\cap \Ec_{zu} \cap \Ec_{zv}) \, \triangledown \, (\Ec_{wz}).
$$
This resolved Colin McDiarmid's question and negatively resolved a conjecture
of Kahn and Saks, see e.g.~\cite[p.~168]{Win86}. For the GYY~inequality \eqref{eq:GYY},
the corresponding result was obtained by Winkler in \cite{Win83}.  See also
an extensive discussion in \cite{Bri85,Day84,Fis92}.
\end{rem}

\smallskip

\subsection{Average height} \label{ss:ineq-height}
The \defn{average height} \ts of an element $x$ in poset~$P$, is defined as
$$
h(P,x) \, := \, \bE[f(x)] \, = \, \frac{1}{e(P)} \. \sum_{f\ts \in \ts \Ec(P)} f(x)\ts.
$$

\begin{thm}[{\rm Winkler \cite{Win82}}{}]\label{t:ineq-height}
Let \ts $P=(X,\prec)$ \ts be a poset, and let \ts $x,y\in X$ \ts be
incomparable elements.  Then:
\begin{equation}\label{eq:ineq-height}
h(P,x) \, \ge \, h(P_{xy},x).
\end{equation}
Moreover, for all \ts $k\in \nn$, we have:
\begin{equation}\label{eq:ineq-height-maj}
\bP\big(\ts f(x)>k\big) \ \ge \ \bP\big(\ts f(x)>k \, \mid \, f(x) < f(y)\ts \big).
\end{equation}
\end{thm}


The proof follows easily from Shepp's proof of the XYZ inequality (Theorem~\ref{t:ineq-XYZ}).
Now, for a subset \ts $S\subseteq X$, denote
$$
h(S) \, := \, \bE\big[\ts \min_{x\in S} f(x) \ts \big].
$$
In particular, \ts $h(S)=h(P,x)$ \ts for \ts $S=\{x\}$, \ts and \ts $h(X)=1$.

\begin{thm}[{\rm Winkler \cite[Thm~4]{Win82}}{}]\label{t:ineq-height-product}
Let \ts $P=(X,\prec)$ \ts be a poset, and let \ts $U, V\subset X$ \ts
such that \ts $U \cup V = X$. Then:
\begin{equation}\label{eq:ineq-height-product}
h(U) \cdot h(V) \, \le \, h(U) \. + \. h(V).
\end{equation}
\end{thm}

The proof of~\eqref{eq:ineq-height-product}
is another elementary probabilistic application of the XYZ inequality.

\begin{cor}
Let \ts $x,y\in X$ \ts be the only two minimal elements in~$P$.
Then \. $h(P,x) \cdot h(P,y) \le h(P,x) + h(P,y)$.  In particular,
either \ts $h(P,x)\le 2$ \ts or \ts $h(P,y)\le 2$.
\end{cor}

Note that this is tight, since for \ts $P=C_\ell + C_\ell$ \ts and \ts $x\in \min(P)$,
we have \. $h(P,x) \to 2$ \. as \ts $\ell\to \infty$.

\begin{proof}
Let \ts $U:=x\!\uparrow$ \ts and \ts $V:=y\!\uparrow$.  Observe that \ts $U\cup V=X$.
By definition, we also have \ts $h(P,x) = h(U)$ \ts and \ts $h(P,y) = h(V)$.  The result
now follows from \eqref{eq:ineq-height-product}.
\end{proof}

\smallskip

\subsection{Deletion correlations} \label{ss:ineq-delete}
Let \ts $P=(X,\prec)$ \ts be a poset with \ts $|X|=n$ \ts elements, let \. $z\in X$ \. and \. $a\in [n]$.
Let \. $\Ec(P,z,a)$ \. be the set of linear extensions \. $f\in \Ec(P)$, such that \. $f(z)=a$.
Denote by \. $\aN(P, z,a):=\bigl|\Ec(P,z,a)\bigr|$ \. the number of such linear extensions.

\begin{thm}[{\rm \cite[Thm~6.3]{CP-corr}}{}]\label{t:ineq-delete}
Let \ts $P=(X,\prec)$ \ts be a poset with \ts $|X|=n>2$ \ts elements.
Fix an element \ts $z \in X$ \ts and integer  \ts $1 \leq a \leq n-2$.
Then, for all distinct minimal elements \ts $x,y\in \min(X-z)$, we have:
\begin{equation}\label{eq:ineq-delete}
\aNr(P, z,a) \cdot \aNr(P-x-y, z,a)  \ \leq \ 2 \, \aNr(P-x, z,a) \cdot \aNr(P-y, z,a).
\end{equation}
\end{thm}

Taking a disjoint sum \ts $P\gets P+ z$ \ts and \ts $a=1$, we get a special case
of the upper bound in \eqref{eq:Fish-baby}, a counterpart to the corollary of
Fishburn's inequality.  The original proof uses the combinatorial atlas
(see~$\S$\ref{ss:proof-LA}). The same holds for results for the rest
of the section.

\smallskip

\subsection{Subsets} \label{ss:ineq-subset}
Fix a nonempty subset $A\subseteq X$.  For a linear extension
\ts $f\in \Ec(P)$, define
\begin{equation}\label{eq:ineq-f-min-def}
f(A) \, := \, \big\{\ts f(x) \, : \, x \in A \ts\big\} \quad \text{and} \quad
f_\mn(A) \, := \, \min \ts f(A).
\end{equation}
Note that \ts $f_\mn(A)=f(x)$ \ts for all singletons \ts $A=\{x\}$, where \ts $x\in X$.
The following result is more natural in probabilistic notation:

\begin{thm}\label{t:ineq-deletion-Sta}
Let \ts $P=(X,\prec)$ \ts be a poset on \ts $|X|\ge 2$ \ts elements. Fix a nonempty
subset \ts $A \subseteq  X$.  Then:
\begin{equation}\label{eq:ineq-Stanley-subset1}
\Pb[1,2\notin f(A)] \ \leq \ \Pb[1\notin f(A)]^2 \qquad \text{and}
\end{equation}
\begin{equation}\label{eq:ineq-Stanley-subset2}
\Pb[1\in f(A)]  \.\cdot \. \Pb[1\notin f(A)] \ \leq \ \Pb[1\notin f(A), 2\in f(A)]\ts,
\end{equation}
\end{thm}

We return to this result in~$\S$\ref{ss:ineq-Sta-conj}. The following
corollary is an easy consequence of Theorem~\ref{t:ineq-deletion-Sta}.

\begin{cor}[{\rm \cite[Cor.~1.7]{CP-corr}}{}]\label{c:ineq-deletion-one}
Let \ts $P=(X,\prec)$ \ts be a poset on \ts $|X|\ge 2$ \ts elements, and let \ts $A \subseteq X$ \ts be
a nonempty subset of elements.  Then:
\begin{equation}\label{eq:Stanley-ext3}
\Pb[1,2\in f(A)] \. \cdot \. \Pb[1,2\notin f(A)] \ \le \ \Pb[1\in f(A), 2\notin f(A)]^2.
\end{equation}
\end{cor}

\begin{proof} 
Multiply \eqref{eq:ineq-Stanley-subset1} for subsets $A$ and $X\sm A$.
Then use \eqref{eq:ineq-Stanley-subset2}.
\end{proof}

Note that \ts $A$ \ts is an \emph{arbitrary} \ts nonempty subset of the ground set~$X$.
For a subset \ts $\textrm{W}\subseteq\Ec(P)$, we write
\ts $\aN\big(P,z,a  \. | \. \textrm{W}\big)$ \ts to denote the
number of linear extensions \ts $f\in \aN(P,z,a)$ \ts which satisfy condition~$\textrm{W}$.
The following result is a generalization of Corollary~\ref{c:ineq-deletion-one}.

\begin{thm}[{\rm \cite[Lem.~6.4]{CP-corr}}{}]
\label{t:ineq-deletion-Sta-arrow}
Let \ts $P=(X,\prec)$ \ts be a poset on \ts $|X| = n \ge 3$ \ts elements,
let \ts $z\in X$, \ts $a \in \{3,\ldots, n\}$, and let \ts $A \subseteq X-z$ \ts be
a nonempty subset.  Then:
\begin{equation}\label{eq:ineq-deletion-Sta-arrow}
\aligned
& \aNr\big(P,z,a \. | \.  1\in f(A), \ts 2 \in f(A\!\uparrow)\big)  \. \cdot \.
		\aNr(P,z,a \. | \.   1,2\notin f(A)) \\
& \hskip1.66cm \leq \ \aNr\big(P,z,a \. | \.  1\in f(A), \ts 2\notin f(A)\big)^2.
\endaligned
	\end{equation}
\end{thm}

Taking a disjoint sum \ts $P\gets P+z$ \ts and \ts $a=n$ \ts implies
\eqref{eq:Stanley-ext3}.  Note also that \.
$$\aN\big(P,z,a \. | \.  1\in f(A), \ts 2 \in f(A)\big) \ \le \
\aN\big(P,z,a \. | \.  1\in f(A), \ts 2 \in f(A\!\uparrow)\big),
$$
so \eqref{eq:ineq-deletion-Sta-arrow} gives a stronger inequality.

\smallskip


\subsection{Covariance inequalities} \label{ss:ineq-cov}
The following theorem gives a similar upper bound for the covariances:

\begin{thm}[{\rm \cite[Thm~1.2]{CP-corr}}{}]\label{t:poset-cov}
Let \ts $P=(X,\prec)$ \ts be a finite poset, and let \ts $x,y \in X$ \ts
be fixed poset elements. Then:
\begin{equation}\label{eq:poset-cov}
\frac{\bE[f(x) \ts f(y)]  \. + \. \bE\big[\min\{f(x),f(y)\}\big]}{\bE[f(x)] \.\cdot\. \bE[f(y)]} \ \le \, 2 \ts.
\end{equation}
\end{thm}

\smallskip
The following result generalized this to subsets:

\begin{thm}[{\rm \cite[Thm~1.8]{CP-corr}}{}]\label{t:poset-cov-multiple}
Let \ts $P=(X,\prec)$ \ts be a finite poset, and let  \ts $A,B\subseteq X$ \ts be nonempty
subsets. Then:
	\begin{equation}\label{eq:poset-cov-multiple}
		\frac{\Eb\big[f_{\min}(A) \ts  f_{\min}(B) \big] \. + \. \Eb\big[f_{\min}(A\cup B)\big]}{\Eb\big[f_{\min}(A)\big]  \.\cdot\.  \Eb\big[f_{\min}(B)\big]} \ \leq \ 2\ts.
	\end{equation}
\end{thm}

\smallskip

Let us emphasize that here \ts $A$ \ts and \ts $B$ \ts are \emph{arbitrary} \ts subsets of the ground set~$X$.
Recall that \ts
${B\!\uparrow} \. := \ts \cup_{b\in B} \, b\!\uparrow$ \. denotes the
\defn{upper closure} of a subset \ts $B\subseteq X$.  The following
is a symmetric generalization of Theorem~\ref{t:ineq-deletion-Sta-arrow} to
two disjoint subsets of minimal elements:

\begin{thm}[{\rm \cite[Thm~1.9]{CP-corr}}{}]
\label{t:ineq-cov-arrow}
Let \ts $P=(X,\prec)$ \ts be a finite poset, and let \ts $A,B \subset \min(P)$ \ts be
disjoint nonempty subsets of minimal elements.  Then:
	\begin{equation}\label{eq:poset-disjoint-logconcave-new}
    \aligned
	& \Pb\big[1\in f(A), \ts 2 \in f(A\!\uparrow)\ts\big] \. \cdot \. \Pb\big[1\in f(B), \ts 2 \in f(B\!\uparrow)\big] \ \leq \ \Pb\big[1\in f(A), 2 \in f(B)\big]^2\ts.
    \endaligned
	\end{equation}
\end{thm}

See also \cite[Thm~1.10]{CP-corr}, for a three element generalization of this inequality.
\smallskip

\subsection{Unique covers} \label{ss:ineq-four-elements}
Let \ts $P=(X,\prec)$ \ts be a poset, and let \ts $x,y\in X$.
Recall that element $y$ \defn{covers} \ts $x$, if \ts $x\prec y$,
and there is no \ts $v\in X$ \ts s.t.\ $x\prec v \prec y$.
For elements \ts $x\prec y$ \ts in~$X$, we say that $y$ is a
\defn{unique cover} \ts of~$x$, if \ts $y$ \ts covers \ts $x$ \ts
and does not cover any other elements in~$X$.

\begin{thm}[{\rm \cite[Cor.~3.10]{CP-corr}}{}]\label{t:ineq-four-elements}
Let \ts $P=(X,\prec)$ \ts be a finite poset, and let \ts $x,y\in \min(P)$ \ts be distinct
minimal elements.  Suppose element \ts $v\in X$ \ts is a unique cover of~$x$, and \ts $w \in X$ \ts
is a  unique cover of~$y$. Then:
\begin{equation}\label{eq:ineq-four-elements}
	  e(P-x-y)^2 \ \geq \ e(P-x-v) \.\cdot \. e(P-y-w).
\end{equation}
\end{thm}

This inequality is derived from \eqref{eq:poset-disjoint-logconcave-new} for
\ts $A=\{x\}$ \ts and \ts $B=\{y\}$.  We conclude with the following four element
inequality.

\begin{thm}[{\rm \cite[Cor.~3.11]{CP-corr}}{}]\label{c:ineq-four-elements-three}
Let \ts $P=(X,\prec)$ \ts be a finite poset, and let \ts $x,y,z\in \min(P)$ \ts
be distinct minimal elements.   Suppose element \ts $u\in X$ \ts is a unique
cover of~$z$.  Then:
\begin{equation}\label{eq:ineq-four-elements-three}
	  e(P-u-z) \.\cdot \. e(P-x-y) \ \leq \ 2 \, e(P-x-z) \.\cdot \. e(P-y-z).
\end{equation}
\end{thm}

This is a direct corollary of a three element generalization of Theorem~\ref{t:ineq-cov-arrow}
mentioned above.



\medskip

\section{Stanley type inequalities}\label{s:ineq-Sta}

In this section we present a collection of Stanley type inequalities.
In the next section, we discuss various equality conditions for some of
these inequalities.

\subsection{Stanley inequality}\label{ss:ineq-Sta}
Let \ts $P=(X,\prec)$ \ts be a poset with \ts $|X|=n$ \ts elements, let \. $x\in X$ \. and \. $a\in [n]$.
Recall that \.
$\Ec(P,x,a)$ \. denotes the set of linear extensions \. $f\in \Ec(P)$,
such that \. $f(x)=a$, and that \. $\aN(P, x,a):=\bigl|\Ec_{\bz\bc}(P,x,a)\bigr|$.
\defn{Stanley's inequality} \ts states that
\. $\big\{\aN(P, x,a)\big\}$ \. is log-concave:
%


\begin{thm}[{\rm \defn{Stanley's inequality}~\cite[Thm~3.2]{Sta-AF}}{}]\label{t:ineq-Sta}
We have:
\begin{equation}\label{eq:Sta}\tag{Sta}
\aNr(P, x,a)^2 \, \ge \, \aNr(P, x,a+1) \.\cdot \.  \aNr(P, x,a-1).
\end{equation}
\end{thm}


The unimodality of \. $\big\{\aN(P, x,a)\big\}$ \. was conjectured by
Kislitsyn \cite[$\S$4.4]{Kis68} and later independently by Rivest.
The log-concavity was conjectured by Chung, Fishburn and Graham \cite{CFG},
who established both conjectures for posets of width two.  The authors of \cite{CFG}
called Rivest's conjecture ``tantalizing'' and add a note characterizing
Stanley's then forthcoming proof using the \defng{Alexandrov--Fenchel inequality}
\ts as ``very ingenious'' (see~$\S$\ref{ss:proof-geom}).

\begin{conj}[{\rm \cite[Conj.~6.3]{Pak-OPAC}}{}]\label{conj:ineq-Sta-SP}
The defect of Stanley's inequality~\eqref{eq:Sta} is not in $\ts \SP$.
\end{conj}

In \cite[$\S$9.12]{CPP-effective}, we wrote ``At this point, it is even hard to guess which
way the answer would go. While some of us believe the answer should be negative,
others disagree.''  We have stronger convictions now.

\smallskip

\subsection{Kahn--Saks inequality}\label{ss:ineq-Sta-KS}
Let \ts $P=(X,\prec)$ \ts be a poset with \ts $|X|=n$ \ts elements, let \. $x,y\in X$ \. and \. $a\in [n]$.
Denote by \.
$\cF(P,x,y,a)$ \. the set of linear extensions \. $f\in \Ec(P)$,
such that \. $f(y)-f(x)=a$.  Let \.$\aF(P, x,y,a) := \bigl|\cF(P,x,y,a)\bigr|$\ts.


\begin{thm}[{\rm \defn{Kahn--Saks inequality}~\cite[Thm~2.5]{KS84}}{}]\label{t:ineq-KS}
We have:
\begin{equation}\label{eq:KS}\tag{KS}
\aFr(P, x,y,a)^2 \, \ge \, \aFr(P, x,y,a+1) \.\cdot \.  \aFr(P, x,y,a-1).
\end{equation}
\end{thm}


It is easy to see that \eqref{eq:KS} implies \eqref{eq:Sta} by taking
\ts $x\gets \wh 0$, \ts $y \gets x$, and \ts $a \gets a+1$.
The theorem is proved using the Alexandrov--Fenchel inequality
again (see~$\S$\ref{ss:proof-geom}).

\smallskip

\subsection{$\bq$-Stanley and $\bq$-KS inequalities}\label{ss:ineq-Sta-multi}
Let \ts $P=(X,\prec)$ \ts be a poset of width two with \ts $|X|=n$ \ts elements.
Fix a partition \. $X=\rC\sqcup \rC'$ \. into two chains, where
\. $\rC = \{u_1\prec \ldots \prec u_\ell\}$, and
\. $\rC'=\{v_1\prec \ldots \prec v_{n-\ell}\}$.

Let \. $\bq:= (q_1,\ldots,q_{\ell})$ \. be formal variables.
Define the \defn{$\bq$-weight} of \ts $\aN(P,x,a)$ \ts and \ts
$\aF(P, x,y,a)$ \ts as follows:
$$
\aN_\bq(a) \, := \, \sum_{f\ts\in\ts \Ec(P,x,a)} \, \bq^{f} \qquad
\text{and} \qquad \aF_\bq(a) \, := \, \sum_{f\ts\in\ts \cF(P,x,y,a)} \, \bq^{f},
$$
where \, $\bq^f \, := \,  q_1^{f(u_1)} \cdots q_\ell^{f(u_\ell)}$.

\begin{thm}[{\rm \defn{$\bq$-Stanley inequality}~\cite[Thm~7.1]{CPP-KS}}]\label{t:Sta-q}
In notation above, let \ts $x\in \rC$ \ts and \ts $a\in [n]$.  Then:
\begin{equation}\label{eq:Sta-ineq-bq}
\aNr_\bq(P,x,a)^2 \ \geqslant_\bq \ \aNr_\bq(P,x,a+1) \. \cdot \. \aNr_\bq(P,x,a-1),
\end{equation}
where the inequality between polynomials is coefficient-wise.
\end{thm}

More generally, we have:

\begin{thm}[{\rm \defn{$\bq$-KS inequality}~\cite[Thm~7.2]{CPP-KS}}]\label{t:Sta-KS-q}
In notation above, let \ts $x,y\in \rC$ \ts be distinct elements, and let \ts $a\in [n]$.  Then:
\begin{equation}\label{eq:KS-ineq-bq}
\aFr_\bq(P,x,y,a)^2 \ \geqslant_\bq \ \aFr_\bq(P,x,y,a+1) \. \cdot \. \aFr_\bq(P,x,y,a-1),
\end{equation}
where the inequality between polynomials is coefficient-wise.
\end{thm}


Note that \eqref{eq:KS-ineq-bq} implies \eqref{eq:Sta-ineq-bq} in a similar way that
\eqref{eq:KS} implies \eqref{eq:Sta}.
Taking all \ts $q_i \gets 1$ \ts in these two inequalities gives \eqref{eq:Sta}
and \eqref{eq:KS}, respectively.  Explicit equality conditions for both
inequalities are given in \cite[Thm~1.6]{CPP-KS} and \cite[Thm~1.7]{CPP-KS}.
Theorems~\ref{t:Sta-q} and~\ref{t:Sta-KS-q} are proved by an explicit injection.

\smallskip

\subsection{Weighted Stanley inequality}\label{ss:ineq-Sta-weighted}
Let \ts $\ap: X \to \rrs$ \ts
be a positive weight function on~$X$.  We say that \ts $\ap$ \ts is \defn{order-reversing}
if it satisfies
\begin{equation}\label{eq:Rev}
u \. \preccurlyeq\. v \quad \Rightarrow 	\quad  \ap(u) \,  \geq \, \ap(v)\ts,
\end{equation}
for all \ts $u,v\in X$.
Define
\begin{equation}\label{eq:sta-def-N-weighted}
\aN_\ap(P,x,a) \, := \, \sum_{f \ts \in \ts \Ec(P,x,a)} \. \ap(f,x)\.,  \quad \text{where}
\quad \ap(f,x)  \, := \, \prod_{y\in X \ : \ f(y) < f(x)} \. \ap(y)\ts.
\end{equation}


\begin{thm}[{\rm \defn{weighted Stanley inequality}, \cite[Thm~1.35]{CP-atlas}}{}]\label{t:Sta-weighted}
For every order-reversing weight function \ts $\apr$, we have:
\begin{equation}\label{eq:Sta-weighted}
\aNr_\apr(P,x,a)^2 \,\. \ge \,\. \aNr_\apr(P,x,a) \.\cdot \. \aNr_\apr(P,x,a)\ts,
\end{equation}
where \ts $\aNr_\apr(P,x,a)$ \ts is defined by~\eqref{eq:sta-def-N-weighted}.
\end{thm}


Taking all \ts $\ap(x) \gets 1$ \ts in the inequality \eqref{eq:Sta-weighted}
gives \eqref{eq:Sta}. Explicit equality conditions for \eqref{eq:Sta-weighted} are
given in \cite[Thm~1.40]{CP-atlas}, generalizing Theorem~\ref{t:Sta-equality}.
Theorem~\ref{t:Sta-weighted} is proved by using a combinatorial atlas
(see~$\S$\ref{ss:proof-LA}).

\smallskip

\subsection{Generalized Stanley inequality}\label{ss:ineq-Sta-gen}
Let \. $x,z_1,\ldots,z_k\in X$ \. and \. $a,c_1,\ldots,c_k\in [n]$;
we write \. $\bz =(z_1,\ldots,z_k)$ \. and \. $\bc =(c_1,\ldots,c_k)$,
and assume that \. $c_1<\cdots< c_k$\ts.

Let \. $\Ec_{\bz\bc}(P)$ \. be the set of linear extensions \. $f\in \Ec(P)$,
such that \. $f(z_i)=c_i$ \. for all \. $1\le i \le k$.
Similarly, let \. $\Ec_{\bz\bc}(P,x,a)$ \. be the set of linear
extensions \. $f\in \Ec_{\bz\bc}(P)$, such that \. $f(x)=a$.
Denote by \. $\aN_{\bz\bc}(P):=\bigl|\Ec_{\bz\bc}(P)\bigr|$ \.
and \. $\aN_{\bz\bc}(P, x,a):=\bigl|\Ec_{\bz\bc}(P,x,a)\bigr|$ \.
the number of such linear extensions.
The following result states that the sequence \. $\big\{\aN_{\bz\bc}(P, x,a), a\in [n]\big\}$ \. is log-concave:
%


\begin{thm}[{\rm \defn{generalized Stanley inequality}~\cite[Thm~3.2]{Sta-AF}}{}]\label{t:ineq-Sta-gen}
In notation above, for all \ts $k \ge 0$, we have:
\begin{equation}\label{eq:Sta-gen}
\aNr_{\bz\bc}(P, x,a)^2 \, \ge \, \aNr_{\bz\bc}(P, x,a+1) \.\cdot \.  \aNr_{\bz\ts\bc}(P, x,a-1).
\end{equation}
\end{thm}

The theorem is proved using the Alexandrov--Fenchel inequality (see~$\S$\ref{ss:proof-geom}).

\begin{op}\label{op:ineq-Sta-gen-weighted}
Find a weighted version of Theorem~\ref{t:ineq-Sta-gen}, i.e.\ a common generalization
of Theorems~\ref{ss:ineq-Sta-weighted} and~\ref{t:ineq-Sta-gen}.
\end{op}

\smallskip

\subsection{Order polynomial version of Stanley inequality}\label{ss:ineq-DDP}
%
%
The following is a natural generalization of Brenti's log-concavity
for the order polynomial \eqref{eq:basic-OP-log-concave} to the setting
of Stanley inequality~\eqref{eq:Sta}.

\begin{thm}[{\em{\rm\defn{Daykin--Daykin--Paterson inequality}}~\cite[Thm~4]{DDP}}{}]\label{t:ineq-DDP}
Let \ts $P=(X,\prec)$ \ts be a finite poset, and let \ts $x\in X$.
Denote by \ts $\Om(P,t \ts ; \ts x,a)$ \ts the number of order preserving
maps \ts $h: X\to [t]$, such that \ts $h(x)=a$.  Then,
for all integer \ts $t>a>1$, we have:
\begin{equation}\label{eq:ineq-DDP}
	\Om(P,t\ts ; \ts x,a)^2 \, \ge \, \Om(P,t \ts ; \ts x,a+1) \. \cdot \. \Om(P,t\ts ; \ts x,a-1).
\end{equation}
Additionally, the defect of this inequality is in~$\SP$.
\end{thm}

The inequality~\eqref{eq:ineq-DDP} was conjectured by Graham \cite[p.~129]{Gra83},
by analogy with Stanley's inequality \eqref{eq:Sta}.
The proof in \cite{DDP} uses an explicit injection.   The authors prove, in fact, a
stronger result, in the style of the generalized Stanley inequality \eqref{eq:Sta-gen}.

\begin{thm}[{\em {\rm \defn{generalized DDP inequality}~\cite[Thm~4]{DDP}}}{}]\label{t:ineq-DDP-gen}
Let \ts $P=(X,\prec)$ \ts be a finite poset, let \ts $x\in X$.  Fix \ts $k \in \nn$ \ts and let \ts $\bz\in X^k$.
Denote by \ts $\Om(P,t \ts ; \ts \bz,\bc  \ts ; \ts x,a)$ \ts the number of order preserving
maps \ts $h: X\to [t]$, such that \ts $h(x)=a$, and \ts $h(z_i)=c_i$ \ts
for all \ts $1\le i \le k$.  Then,
for all integer \ts $t>a>1$, we have:
\begin{equation}\label{eq:ineq-DDP-gen}
	\Om(P,t\ts ; \ts \bz,\bc \ts ; \ts  x,a)^2 \, \ge \, \Om(P,t \ts ; \ts \bz,\bc  \ts ; \ts x,a+1) \. \cdot \. \Om(P,t\ts ; \ts \bz,\bc  \ts ; \ts  x,a-1).
\end{equation}
Additionally, the defect of this inequality is in~$\SP$.
\end{thm}

Graham believed that there should exist a proof based on the FKG or AD~inequalities.
He lamented: ``such a proof has up to now successfully eluded all attempts to find it''
\cite[p.~129]{Gra83}.  Such proof was given in \cite{CP-corr}, which also gave a
generalization of the \ts $q$-log-concavity \eqref{eq:basic-OP-q} and
$\bq$-log-concavity \eqref{eq:basic-OP-bq} to this setting:

\begin{thm}[{\rm \defn{$\bq$--DDP inequality}~\cite[Thm~9.3]{CP-corr}}{}]\label{t:DDP-multi}
	Let \ts $P=(X,\prec)$ \ts be a finite poset, let $t\in \nn$, and let \ts $x \in X$.
	Then, for every \ts $t>a>1$, we have:
	\begin{equation}\label{eq:DDP-multi}
        \Omega_\bq(P,t \ts ; \ts x,a)^2 \ \geqslant_\bq \ \Omega_\bq(P,t \ts ;\ts  x,a+1) \. \cdot \. \Omega_\bq(P,t \ts ; x,a-1),
	\end{equation}
where the inequality holds coefficient-wise as a polynomial in \ts $\bq=(q_1,\ldots,q_n)$.
\end{thm}

\smallskip

\subsection{Cross-product conjecture}\label{ss:ineq-CPC}
Let \ts $P=(X,\prec)$ \ts be a poset with \ts $|X|=n$ \ts elements.
Fix distinct elements \. $x,y,z\in X$.
For \ts $a,b \geq 1$, let \. $\cF_{xyz}(P,a,b) \. := \. \cF(P,x,y,a) \cap \, \cF(P,y,z,b)$.
Equivalently,
\[
\cF_{xyz}(P,a,b) \ : = \ \big\{\ts f \in \Ec(P) \, : \, f(y)-f(x)=a, \.  f(z)-f(y)=b \ts \big\}.
\]
Denote \. $\aF_{xyz}(P,a,b) \. := \. \big|\cF_{xyz}(P,a,b)\big|$.
By Theorem~\ref{t:Sta-gen-vanish}, we have \. $\big\{\aF_{xyz}(P,a,b)=^?0\big\}\in \poly$,
since there are at most $n$ choices for \ts $f(x)$, which then determine \ts $f(y)$ \ts
and \ts $f(z)$.

\smallskip

\begin{conj}[{\rm \defn{Cross--product conjecture}, Felsner--Trotter \cite[Conj.~8.3]{FT93}}{}]\label{conj:CPC}
We have:
\begin{equation}\label{eq:ineq-CPC} \tag{CPC}
\aFr_{xyz}(P,a+1,b) \.\cdot\. \aFr_{xyz}(P,a,b+1) \ \ge \ \aFr_{xyz}(P,a,b) \.\cdot\. \aFr_{xyz}(P,a+1,b+1) \ts.
\end{equation}
\end{conj}


The following result give a summary of known special cases

\begin{thm} \label{t:CPC-summary}
Conjecture~\ref{conj:CPC} holds in the following cases:
\begin{enumerate}
			[{label=\textnormal{({\alph*})},
		ref=\textnormal{\alph*}}]
\item[$(1)$] \ $a=b=1$, \ts see \. \cite[Thm~3.2]{BFT95},
\item[$(2)$] \ $\width(P)=2$, \ts see \. \cite[Thm~1.4]{CPP-CP},
\item[$(3)$] \ $\aFr_{xyz}(P,a,b+2) = \aFr_{xyz}(P,a+2,b) = 0$, \. $\aFr_{xyz}(P,a,b)>0$ \. and \. $\aFr_{xyz}(P,a+1,b+1)>0$, \ts see \.
\cite[Thm~1.2]{CPP-quant}.
\end{enumerate}
\end{thm}


The proof of {\small $(1)$} is based on the AD~inequality
(see~$\S$\ref{ss:proof-FKG}).  The authors lamented:  ``something
more powerful seems to be needed'' to prove the general form
of~\eqref{eq:ineq-CPC}.

Note that~\eqref{eq:ineq-CPC} easily implies~\eqref{eq:KS}, by taking \ts
$y\gets z$ \ts and \ts $P\gets P+y$, see e.g.\  \cite[$\S$3.1]{CPP-CP}.
In fact, \eqref{eq:GYY} also follows from~\eqref{eq:ineq-CPC}, by a more
involved argument (ibid., $\S$3.4). Of course, the value of these implications
is low given that \eqref{eq:ineq-CPC} remains an open problem.

For posets of width two, the $q$-analogue of \eqref{eq:ineq-CPC} and the
equality conditions are given in \cite[Thm~1.7 and~1.8]{CPP-CP}.
In fact, a stronger inequality holds in this case:
\begin{equation}\label{eq:ineq-GCPC} 
\aF_{xyz}(P,a,b) \cdot \aF_{xyz}(P,c,d) \. \le \, \aF_{xyz}(P,c,b) \cdot \aF_{xyz}(P,a,d),
\ \ \text{for all} \ \ a\le c \ \text{and} \ b \le d,
\end{equation}
see \cite[Thm~1.6]{CPP-CP}.
For \. $c=a+1$ \. and \. $d=b+1$, where \. $a,b\geq 1$, this gives \eqref{eq:ineq-CPC}.
The inequality~\eqref{eq:ineq-GCPC} fails already for posets of width three
\cite[Thm~1.6]{CPP-quant}.

\smallskip

When \. $\aF_{xyz}(P,a,b)\cdot \aF_{xyz}(P,a+1,b+1)=0$, the inequality \eqref{eq:ineq-CPC}
holds trivially.  Note that this assumption as well as the assumptions
in~{\small $(3)$} can be verified in polynomial time.  For the remaining
possible cases, we have the following weak version of Conjecture~\ref{conj:CPC}.

\smallskip

\begin{thm}[{\rm \cite[Thm~1.2]{CPP-quant}}{}] \label{t:ineq-CPC-quant}
Let \ts $P=(X,\prec)$ \ts be a poset on \ts $|X|=n$ \ts elements.
Fix distinct elements \ts $x,y,z\in X$.
Suppose that \. $\aFr_{xyz}(P,a,b+2)\.\aFr_{xyz}(P,a+2,b) > 0$. Then:
\begin{equation}\label{eq:main-thm-1}
{\aFr_{xyz}(P,a+1,b) \, \aFr_{xyz}(P,a,b+1)} \, \geq \,  \Bigl( \tfrac12 \, + \,
\tfrac{1}{4\ts n \ts \sqrt{a\ts b}}\Bigr) \. \aFr_{xyz}(P,a,b) \, \aFr_{xyz}(P,a+1,b+1).
\end{equation}
Alternatively, suppose that \. $\aFr_{xyz}(P,a,b+2)=0$ \. and \. $\aFr_{xyz}(P,a+2,b) > 0$.  Then:
\begin{equation}\label{eq:main-thm-2}
{\aFr_{xyz}(P,a+1,b)\, \aFr_{xyz}(P,a,b+1)} \, \geq \, \, \Bigl( \tfrac12  \, + \,
\tfrac{1}{16 \ts n \ts a \ts b^2} \Bigr) \. \aFr_{xyz}(P,a,b) \, \aFr_{xyz}(P,a+1,b+1).
\end{equation}
Finally, suppose that \. $\aFr_{xyz}(P,a,b+2)\.\aFr_{xyz}(P,a+2,b) = 0$.  Then:
\begin{equation}\label{eq:main-thm-3}
\aFr_{xyz}(P,a,b) \, \aFr_{xyz}(P,a+1,b+1) \, = \, 0.
\end{equation}
\end{thm}


Note that \eqref{eq:main-thm-3} implies part \ts $(3)$ \ts in Theorem~\ref{t:CPC-summary}.
The theorem is proved using geometric inequalities (see~$\S$\ref{ss:proof-geom}).

\smallskip

\subsection{Order polynomial version of CPC}\label{ss:ineq-CPC-OP}
The following is a natural generalization of the DDP inequality \eqref{eq:ineq-DDP}
to the setting of CPC~\eqref{eq:Sta}.

Let \ts $P=(X,\prec)$ \ts be a poset \ts on \ts
$|X|=n$ \ts elements, and let \ts $X=\{x_1,\ldots,x_n\}$.
Fix \ts $t\ge 0$ \ts and distinct elements \ts $x,y,z \in X$.
For  integers \ts $a,b \geq 0$\ts, let
$$\cP(P,t \ts ;\ts x,y,z \ts ; \ts a,b) \ := \ \big\{h\in \cP(P,t) \ : \ h(y)-h(x)=a \ \ \text{and} \ \ h(z)-h(y)=b\big\}.
$$
Denote
$$
\aligned
\Lambda(P,t \ts ;\ts x,y,z \ts ; \ts a,b) \ & := \, \big|\cP(P,t \ts ;\ts x,y,z \ts ; \ts a,b)\big|\ts,  \quad \text{and}\\
\Lambda_{\bq}(P,t \ts ;\ts x,y,z \ts ; \ts a,b) \, & := \, \sum_{f\ts \in \ts \cP(P, \ts t \. ;\. x, \ts y, \ts z \. ; \. a, \ts b)} \. q_1^{f(x_1)}  \ts \cdots \, q_n^{f(x_n)}\ts.
\endaligned
$$
\smallskip

\begin{thm}[{\rm \defn{Cross-product inequality for $P$-partitions}~\cite[Thm~9.3]{CP-multi}}{}]
\label{thm:CPC-OP}
	Let \ts $P=(X,\prec)$ \ts be a finite poset, let $x,y,z \in \cP$, and let \ts $t\ge 1$ \ts be a positive integer.
Then, for every \ts $a,b \geq 0$, we have:
\begin{equation}\label{eq:CPC-OP-1}
\aligned  &  \Lambda(P,t \ts ;\ts x,y,z \ts ; \ts a+1,b) \cdot \Lambda(P,t \ts ;\ts x,y,z \ts ; \ts a,b+1) \\
& \hskip1.cm \geq \ \Lambda(P,t \ts ;\ts x,y,z \ts ; \ts a,b) \cdot \Lambda(P,t \ts ;\ts x,y,z \ts ; \ts a+1,b+1).
\endaligned
\end{equation}
More generally:
\begin{equation}\label{eq:CPC-OP}
\aligned  &   \Lambda_{\bq}(P,t \ts ;\ts x,y,z \ts ; \ts a+1,b) \cdot  \Lambda_{\bq}(P,t \ts ;\ts x,y,z \ts ; \ts a,b+1) \\
& \hskip1.cm  \geqslant_\bq \ \Lambda_{\bq}(P,t \ts ;\ts x,y,z \ts ; \ts a,b) \cdot \Lambda_{\bq}(P,t \ts ;\ts x,y,z \ts ; \ts a+1,b+1).
\endaligned
\end{equation}
\end{thm}
The proof uses a generalization of the AD~inequality (see~$\S$\ref{ss:proof-FKG}).

\smallskip

\subsection{Conjectural generalization}\label{ss:ineq-Sta-conj}
Recall \eqref{eq:ineq-f-min-def}, that \. $f_\mn(A):=\min\{f(x)\. : \. x\in A\}$.  The
following is the natural generalization of Stanley's inequality \eqref{eq:Sta}.

\begin{conj}[{\rm\defn{extended Stanley inequality}~\cite[Conj.~1.5]{CP-corr}}{}]\label{conj:poset-Stanley-subset}
Let \ts $P=(X,\prec)$ \ts be a poset with \ts $|X|=n$ \ts elements. Fix a nonempty
subset \ts
$A \subseteq  X$, and let \. $2 \leq k \leq n-1$. Then:
\begin{equation}\label{eq:Stanley-conj}
    \Pb[f_{\min}(A)=k]^2 \  \geq \  \Pb[f_{\min}(A)=k+1]  \. \cdot \. \Pb[f_{\min}(A)=k-1]\ts.
\end{equation}
\end{conj}

This conjecture also implies Theorem~\ref{t:ineq-deletion-Sta},  
see \cite[$\S$7.1]{CP-corr}.\footnote{Maxwell Aires recently informed us that 
Conjecture~\ref{conj:poset-Stanley-subset} fails for a poset $P=C_3+C_1$, where $A$ 
consists of two maximal elements of~$P$ and $k=2$ (personal communication, February~6, 2025).  
Independently, Jonathan Leake and Shayan Oveis Gharan informed us that this is essentially
the only counterexample to the conjecture (personal communication, February~20, 2025). 
}

\smallskip

\subsection{Second moment conjecture}\label{ss:ineq-second-moment-conj}
In Theorem~\ref{t:poset-cov}, letting \ts $y= x$ \ts
gives the following curious bound on the second moment:

\begin{cor}[{\rm\defn{second moment inequality}~\cite[Cor.~3.5]{CP-corr}}{}]\label{c:ineq-second-moment}
Let \ts $P=(X,\prec)$ \ts be a finite poset, and let \ts $x \in X$ \ts be a fixed element.
Then:
\begin{equation}\label{eq:ineq-second-moment}
1 \, \le \	\frac{\Eb[f(x)^2 ]}{\Eb[f(x)]^2} \ < \, 2 \ts.
\end{equation}
\end{cor}

The lower bound is trivial and holds for every random variable.
The (non-strict) upper bound also follows from Stanley's inequality,
since for every log-concave random variable \ts $Z$, we have: \. $\Eb[Z^2] \le 2 \ts \Eb[Z]^2$,
see \cite[Prop.~3.7]{CP-corr}.  The following conjecture improves upon the upper bound
in~\eqref{eq:ineq-second-moment}.\footnote{Maxwell Aires and Jeff Kahn recently
pointed out that Conjecture~\ref{conj:ineq-second-moment} is false and that the lower bound~$2$
is asymptotically tight (personal communication, July~15, 2024). In their example,
\ts $P=C_m + C_n$ \ts s.t.\ \ts $m, n/m \to \infty$ \ts and $x=\min(C_m)$.
}

\begin{conj}[{\rm\defn{second moment conjecture}~\cite[Conj.~3.8]{CP-corr}}{}]\label{conj:ineq-second-moment}
Let \ts $P=(X,\prec)$ \ts be a finite poset, and let \ts $x \in X$ \ts be a fixed element.
Then:
\begin{equation}\label{eq:ineq-second-moment-conj}
\frac{\Eb[f(x)^2 ]}{\Eb[f(x)]^2} \ \le \, \frac43 \..
\end{equation}
\end{conj}

In fact, the inequality is probably always strict as the following example
suggests.

\begin{ex}\label{ex:ineq-second-moment-43}
Let \ts $P:=C_{n-1}+\{x\}$ \ts be a disjoint sum of two chains.  We have:
\[\frac{\Eb[f(x)^2 ]}{\Eb[f(x)]^2}
 \ = \ \frac{\frac{1}{n} \sum_{k=1}^n \. k^2}{\big(\frac{1}{n} \sum_{k=1}^n \. k \big)^2} \ = \    \frac{4n+2}{3n+3} \
\to \ \frac{4}{3} \ \ \quad \text{as \ \, $n \to \infty$\ts.}
\]
Thus, the constant in the upper bound \eqref{eq:ineq-second-moment-conj} must be at least~$4/3$.
\end{ex}



\medskip

\section{Equality conditions of Stanley type inequalities}\label{s:eq-Sta}

\subsection{Stanley inequality}\label{ss:ineq-Sta-equality}
Recall that \. $\al(x) := \left|x\!\downarrow\right|$ \. and \.
$\be(x) := \left|x\!\uparrow\right|$ \. denote the sizes of
the lower and upper order ideals, respectively.


\begin{thm}[{\rm \defn{vanishing conditions}, Daykin and Daykin~\cite[Thm~8.2]{DD85}}{}] \label{t:Sta-vanish}
Let \ts $P=(X,\prec)$ \ts be a poset with \ts $|X|=n$ \ts elements, let \ts
$x \in X$ \ts and \ts $a \in [n]$.
Then \. $\aNr(P, x,a)>0$ \. \underline{if and only if} \.
$$\al(x)\le a \quad \text{and} \quad \be(x)\le n-a+1.$$
Moreover, if \. $\aNr(P, x,a)>0$, then
a linear extension \. $f \in \Ec(P,x,a)$ \. can be found in polynomial time.
\end{thm}


The original proof uses promotion/demotion operators (under a different name,
cf.~$\S$\ref{ss:proof-promo}),
This results was rediscovered in \cite[Lem.~15.2]{SvH20}.
By Theorem~\ref{t:ineq-Sta}, Stanley's inequality~\eqref{eq:Sta}
is an equality whenever \. $\aN(P, x,a)=0$.  The other equality cases
are given by the following result:


\begin{thm}[{\rm \defn{equality conditions}, Shenfeld and van~Handel~\cite[Thm~15.3]{SvH20}}{}]\label{t:Sta-equality}
Let \ts $P=(X,\prec)$ \ts be a poset with \ts $|X|=n$ \ts elements, let \ts
$x \in X$ \ts and \ts $a \in [n]$.  Suppose that
\. $\aNr(P, x,a)>0$.
\underline{The following are equivalent}$\ts :$
\begin{enumerate}
			[{label=\textnormal{({\alph*})},
		ref=\textnormal{\alph*}}]
\item[$(1)$] \ $\aNr(P, x,a)^2 \. = \. \aNr(P, x,a+1) \. \cdot \. \aNr(P, x,a-1)$,
\item[$(2)$] \ $\aNr(P, x,a+1) \. = \. \aNr(P, x,a) \. = \. \aNr(P, x,a-1)$,
\item[$(3)$] \ we have \. $\al(y)>a$ \. for all \. $y\succ x$, \. and \. $\be(y)>n-a+1$ \. for all \. $y\prec x$.
\end{enumerate}
\end{thm}


The original proof uses a technical geometric argument (see~$\S$\ref{ss:proof-geom}).
The result was reproved in \cite[Thm~1.39]{CP-atlas} using the combinatorial atlas
technology, and extended to equality conditions  of the weighted Stanley
inequality~\eqref{eq:Sta-weighted}.

\begin{cor} \label{c:Sta-equality-P}
The equality of Stanley's inequality~\eqref{eq:Sta} can be verified in polynomial time:
$$\big\{\aNr(P, x,a)^2 =^? \aNr(P, x,a+1) \cdot \aNr(P, x,a-1)\}\in \poly\ts.$$
\end{cor}

This follows from Theorem~\ref{t:Sta-vanish} in the vanishing case,
since the equality always holds, and from \. $(1) \Leftrightarrow (3)$ \.
in Theorem~\ref{t:Sta-equality} in the nonvanishing cases.

\smallskip

\subsection{CPC implies equality conditions}\label{ss:ineq-CPC-cons}
We start with the following surprising inequality:

\begin{conj}\label{conj:ineq-CPC-cons}
Let \ts $P=(X,\prec)$ \ts be a poset on \ts $|X|=n$ \ts elements.
Fix an element \ts $z\in X$.  Then,
for all  integer \ts $a,i\geq 1$, we have:
	\begin{equation}\label{eq:ineq-CPC-SS} 
\aligned
		& (a+i-1) \, \aNr(P,z,a+i-1) \. \cdot \aNr(P,z,a)  \\
& \hskip.6cm \geq \ (a-1) \, \aNr(P,z,a-1)\. \cdot \. \aNr(P,z,a+i)  \  + \ i \, \aNr(P,z,a+i) \.\cdot\. \aNr(P,z,a).
\endaligned
	\end{equation}
\end{conj}

The following results were left on the cutting floor from~\cite{CPP-quant}:

\begin{thm}[{\rm Chan--Pak--Panova, see~$\S$\ref{ss:proof-CPC-cons}}{}]\label{t:ineq-CPC-cons}
The Cross--product Conjecture~\ref{conj:CPC} implies Conjecture~\ref{conj:ineq-CPC-cons}.
\end{thm}

\begin{prop}[{\rm Chan--Pak--Panova, see~$\S$\ref{ss:proof-CPC-cons}}{}]\label{p:ineq-CPC-cons}
Inequality \eqref{eq:ineq-CPC-SS} implies that \. $(1) \. \Leftrightarrow \. (2)$ \.
in Theorem~\ref{t:Sta-equality}.  Additionally, inequality \eqref{eq:ineq-CPC-SS} implies
Stanley's inequality~\eqref{eq:Sta}.
\end{prop}

Combined, these two result show that  Conjecture \ref{conj:CPC} implies the first part of the equality
conditions of Stanley's inequality \eqref{eq:Sta} given in Theorem~\ref{t:Sta-equality}.
Since the only known proofs of the latter are rather difficult (using either convex
geometry or the combinatorial atlas technology), this suggests that Conjecture~\ref{conj:CPC} is
also very difficult to prove.  Another possibility is that Conjecture~\ref{conj:CPC} is false
for posets of large width (cf.~$\S$\ref{ss:finrem-next}).

\smallskip

\subsection{Kahn--Saks inequality}\label{ss:ineq-KS-equality}
Denote \. $\ga(u,v) \ts := \ts \#\{ y\in X, \text{ s.t. } u\prec y\prec v\}$.


\begin{thm}[{\rm \defn{vanishing conditions}~\cite[Thm~8.5]{CPP-effective}}{}] \label{t:KS-vanish}
Let \ts $P=(X,\prec)$ \ts be a poset with \ts $|X|=n$ \ts elements, let \. $x,y\in X$ \. and \. $a\in [n]$.
We have: \. $\aFr(P, x,y,a)>0$ \. \underline{if and only if} \.
$$\ga(x,y) \, < \,  a \, < \, n \. - \. \al(x) \. -  \.\be(y)\ts.
$$
Moreover, if \, $\aFr(P, x,y,a)>0$, then a linear extension \.
$f \in \cF(P,x,y,a)$ \. can be found in polynomial time.
\end{thm}


The original proof uses a variation on promotion/demotion operators (see~$\S$\ref{ss:proof-promo}).
See also \cite{vHYZ} for an alternative proof of Theorem~\ref{t:KS-vanish}.
The equality conditions for the Kahn--Saks inequality \eqref{eq:KS} were
completely resolved in \cite{vHYZ}, but too cumbersome to state here.
The following theorem is a compilation of several results in that paper.


\begin{thm}[{\rm \defn{equality conditions}, van~Handel, Yan and Zeng~\cite{vHYZ}}{}] \label{t:KS-equality}
Let \ts $P=(X,\prec)$ \ts be a poset with \ts $|X|=n$ \ts elements,
let \. $x,y\in X$ \. and \. $a\in [n]$.
Suppose that \. $\aFr(P, x,y,a)>0$.   Then
\begin{equation} \label{eq:KS-equal}
\aFr(P, x,y,a)^2 \, = \, \aFr(P, x,y,a+1) \.\cdot \.  \aFr(P, x,y,a-1)
\end{equation}
\underline{if and only if} \.
$$\aligned \text{either} \quad  & \aFr(P, x,y,a+1) \. = \.  \aFr(P, x,y,a) \. = \.  \aFr(P, x,y,a-1) \\
\text{or} \quad  &  \aFr(P, x,y,a+1) \. = \. 2\cdot  \aFr(P, x,y,a) \. = \. 4\cdot \aFr(P, x,y,a-1).
\endaligned
$$
Additionally, the equality \eqref{eq:KS-equal} can be verified in polynomial time.
\end{thm}


Note the asymmetric structure of the three-term geometric progression with ratio~$2$,
a phenomenon which does not occur for the Stanley inequality (Theorem~\ref{t:Sta-equality}).
Nor does it occur for posets of width two when the equality conditions are especially simple,
see \cite[Thm~1.7~and~$\S$8.4]{CPP-KS}.  Two equivalent conjectural characterizations of
the first part (the complete equality) were given in \cite[Conj.~8.7~and~Thm~8.9]{CPP-KS}.

\smallskip

\subsection{Vanishing conditions}\label{ss:ineq-Sta-vanish}
The following result generalizes Theorem~\ref{t:Sta-vanish} to vanishing conditions
of the generalized Stanley inequality \eqref{eq:Sta-gen}.


\begin{thm}[{\rm \defn{vanishing conditions}, Daykin~and~Daykin~\cite[Thm~8.2]{DD85}}{}] \label{t:Sta-gen-vanish}
Let \ts $P=(X,\prec)$ \ts be a poset with \ts $|X|=n$ \ts elements.
Let \.  $\bz =(z_1,\ldots,z_k)\in X^k$, \. $\bc =(c_1,\ldots,c_k)\in [n]^k$,
and assume that \. $c_1<\cdots< c_k$\ts.
Let \. $\Ec_{\bz\bc}(P)$ \. denotes the set of linear extensions
\. $f\in \Ec(P)$, s.t.\ \. $f(z_i)=c_i$ \. for all \. $1\le i \le k$.
We have \. $\aNr_{\bz\bc}(P) > 0$ \, \underline{if and only if} 
\begin{equation}\label{eq:Sta-gen-vanish}
\aligned
& \al(z_i)\. \le \. c_i \.,  \quad \be(z_i)\le n-c_i+1\,, \quad \  \text{for all} \quad 1\.\le \. i \. \le \. k\ts, \ \ \, \text{and} \\
& c_j \. - \. c_i \, > \, \ga(z_i,z_j) \quad \  \text{for all} \quad 1\.\le \. i \. < \. j \. \le \. k\ts.
\endaligned
\end{equation}
Consequently, the vanishing problem \. $\big\{\aNr_{\bz\bc}(P)=^?0\big\} \in \poly$.
Moreover, if \. $\aNr_{\bz\bc}(P)>0$, then
a linear extension \. $f \in \Ec_{\bz\bc}(P)$ \. can be found in polynomial time.
\end{thm}

The original proof used promotion/demotion operators (see~$\S$\ref{ss:proof-promo}).
This result was rediscovered in \cite[Thm~1.11]{CPP-effective} and \cite[Thm~5.3]{MS22},
where the latter used a geometric argument.

\smallskip

\subsection{Uniqueness conditions}\label{ss:ineq-Sta-unique}
The uniqueness conditions of the generalized Stanley inequality \eqref{eq:Sta-gen}
provide another special case of a polynomial time decision problem.

Let \. $v_i:=f^{-1}(c_i-1)$ \. and \. $w_i:=f^{-1}(c_i+1)$ \.
for \. $1\le i \le k$.
We adopt the convention that \ts $v_1=\widehat{0}$ \ts if \ts $c_1=1$,
and \ts $w_k=\widehat{1}$ \ts if \ts $c_k=n$.  For \. $1 \leq i \leq  j \leq n$,
let
\[ f^{-1}[i,j] \ := \ \bigl\{\ts f^{-1}(i)\., \.\ldots \.,  f^{-1}(j)\ts  \bigr\}.
\]


\begin{thm}[{\rm \defn{uniqueness conditions}, \cite[Thm~7.5]{CPP-effective}}{}] \label{t:gen-Stanley-unique}
Let \ts $f\in \Ec_{\bz\bc}(P)$.
Then we have \ts $\aNr_{\bz\bc}(P)=1$ \, \underline{if and only if} \,
	the following conditions hold$:$
	\begin{equation*}
\aligned
		& (1) \quad \text{$f^{-1}[c_i+1,c_{i+1}-1]$ \. forms a chain in $P$  for every \. $1\le i \le k$, \ts and } \\
	& (2) \quad  \text{there are no \. $1 \leq i \leq j \leq k$\ts, \. such that \. $\{v_i,w_j\} \ts \parallel \ts f^{-1}[c_i,c_j]$.}	
\endaligned
\end{equation*}
Consequently, the uniqueness problem \. $\big\{\aNr_{\bz\bc}(P)=^?1\big\} \in \poly$.
\end{thm}

The first part is proved using  promotion/demotion operators (see~$\S$\ref{ss:proof-promo}).
The last part follows from the first part of the theorem, since by Theorem~\ref{t:Sta-gen-vanish},
a linear extension \ts $f\in \Ec_{\bz\bc}(P)$ \ts can be found in polynomial time.

\begin{conj}\label{conj:gen-Stanley-m}
For every fixed integer \ts $m\in \nn$, the decision problem \. $\big\{\aNr_{\bz\bc}(P)=^?m\big\} \in \poly$.
\end{conj}

\smallskip

\subsection{Equality conditions, positive results}\label{ss:ineq-Sta-gen-equality}
In contrast with equality conditions for the Kahn--Saks inequality (Theorem~\ref{t:KS-equality}),
the equivalence of \. $(1) \Leftrightarrow (2)$ \. in Theorem~\ref{t:Sta-equality}
generalizes to all \ts $k\ge 0$.


\begin{thm}[{\rm \defn{complete equality property}, Ma--Shenfeld~\cite[Thm~1.3~and~1.5]{MS22}}{}]\label{t:ineq-Sta-gen-MS-equality}
Let \ts $P=(X,\prec)$ \ts be a poset with \ts $|X|=n$ \ts elements,
let \. $x,z_1,\ldots,z_k\in X$ \. and \. $a,c_1,\ldots,c_k\in [n]$.
We have
\begin{equation}\label{eq:Sta-gen-eq-MS}
\aNr_{\bz\bc}(P, x,a)^2 \, = \, \aNr_{\bz\bc}(P, x,a+1) \.\cdot \.  \aNr_{\bz\ts\bc}(P, x,a-1)
\end{equation}
\underline{if and only if} \,
\begin{equation}\label{eq:Sta-gen-equal}
\aNr_{\bz\bc}(P, x,a+1) \, = \, \aNr_{\bz\bc}(P, x,a) \, = \,  \aNr_{\bz\ts\bc}(P, x,a-1).
\end{equation}
\end{thm}


For \ts $k=1$, the equality cases of \eqref{eq:Sta-gen-eq-MS}
are given by the following result:


\begin{thm}[{\rm \cite[$\S$9.1]{CP-AFequality}}{}] \label{t:Sta1-equality}
Let \ts $P=(X,\prec)$ \ts be a poset on \ts $n=|X|$ \ts elements, let
\ts $x,z\in X$,  and \ts $a,c \in [n]$.  Then
\begin{equation*}
\aNr_{z\ts c}(P, x,a)^2 \, = \, \aNr_{z\ts c}(P, x,a+1)\cdot  \aNr_{z\ts c}(P, x,a-1).
\end{equation*}
\underline{if and only if} \.
\begin{equation}\label{eq:sapporo-1}
		 \aNr_{zy,\ts cb'}(P,x,a') \. = \. 0  \quad \text{for all} \quad y \in \comp(x) \ \ \ \text{and} \ \ \  a',b' \in \{a-1,a,a+1\}.
\end{equation}
\end{thm}

\smallskip

Since checking \eqref{eq:sapporo-1} can be done in polynomial time by Theorem~\ref{t:Sta-gen-vanish},
we easily have:


\begin{cor}[{\rm \cite[Thm~1.4]{CP-AFequality}}{}] \label{c:Sta1-equality-P}
For $k=1$, the equality verification of the generalized Stanley inequality~\eqref{eq:Sta-gen-eq-MS}
can be done in polynomial time:
$$\big\{\aNr_{z\ts c}(P, x,a)^2 =^? \aNr_{z\ts c}(P, x,a+1) \cdot \aNr_{z\ts c}(P, x,a-1)\}\in \poly\ts.
$$
\end{cor}


Corollary~\ref{c:Sta1-equality-P} trivially implies Corollary~\ref{c:Sta-equality-P}.

\smallskip

\subsection{Equality conditions, Ma--Shenfeld theory}\label{ss:ineq-Sta-gen-equality-MS}
Let \. $x,z_1,\ldots,z_k\in X$ \. and \. $a,c_1,\ldots,c_k\in [n]$;
we write \. $\bz =(z_1,\ldots,z_k)$ \. and \. $\bc =(c_1,\ldots,c_k)$.
Throughout the section, we assume that \. $x\prec z_1 \prec \ldots \prec z_k$ \. and
$c_1<\cdots< c_k$\ts.
Suppose poset \ts $P$ \ts has
elements \ts $z_0= \wh 0$ \ts and \ts $z_{k+1}=\wh 1$. Denote
$$\La \ := \ \big\{\ts (r,s) \, : \, 0 \le r < s \le k+1, \, c_r<a<c_s \,, \, (r,s) \ne (0,k+1)\ts\big\}.
$$
Pairs \ts $(r,s)\in \La$ \ts are called \defnb{splitting pairs}.\footnote{In~\cite[Def.~5.2]{MS22},
these are called \emph{$\ell$-splitting pairs}, which are instead written as \ts $(r+1,s)$.}

\smallskip

\begin{dl}\label{dl:ineq-MS-crit-simp}
Suppose that
\begin{align}\label{eq:ineq-MS-Pos} 
	\aNr_{\bz\bc}(P, x,a-1), \ \aNr_{\bz\bc}(P, x,a), \ \aNr_{\bz\bc}(P, x,a+1) \ > \ 0 \quad \text{ and } \quad n >k+3.
\end{align}
We say that \. $(P,x,a,\bz,\bc)$ \. is \. \defnb{subcritical},
\. \defnb{critical}, and \. \defnb{supercritical} \. \underline{if and only if} \.
for every  splitting pair \ts $(r,s)\in \La$, we have, respectively:
\begin{equation}\label{eq:ineq-MS-subcrit}\tag{subcrit}
    \ga(z_r,z_s) \, \leq \,  c_s\. - \. c_r \. - \. 1\ts,
\end{equation}
\begin{equation}\label{eq:ineq-MS-crit}\tag{crit}
	\ga(z_r,z_s)   \, \leq \,  c_s\. - \. c_r \. - \. 2\ts,
\end{equation}
\begin{equation}\label{eq:ineq-MS-supercrit}\tag{supercrit}
	\ga(z_r,z_s)   \, \leq \,  c_s\. - \. c_r \. - \. 3\ts.
\end{equation}
In particular, each membership problem is in~$\ts \poly$.
\end{dl}

The definition in \cite{MS22} is cumbersome, so we use this definition which
is more transparent.  We prove that it is equivalent to the original definition
in~$\S$\ref{ss:proof-crit-simp}.  By Definition~\ref{dl:ineq-MS-crit-simp}, we have:
\[  \{\text{subcritical}\}  \quad \supseteq \quad   \{\text{critical}\}
\quad \supseteq \quad \{\text{supercritical}\}.
\]
Note that \eqref{eq:ineq-MS-Pos} implies \eqref{eq:ineq-MS-subcrit},
since for every \ts $f \in \aN_{\bz\bc}(P,x,a)$ \ts and \ts $(r,s)\in \La$, we have:
\begin{equation}\label{eq:ineq-MS-basic}
 \ga(z_r,z_s) \, \leq \, \big|\big\{f^{-1}(i) \, : \, c_r< i <c_s  \big\}\big| \, = \,  c_s \. - \. c_r \. - \. 1\ts.
\end{equation}
Note also that for \ts $k=0$, every quintuple \. $(P,x,a,\bz,\bc)$ \.  is supercritical,
since \ts $\La=\emp$.

\begin{thm}[{\cite[Thm~1.3, \ts Rem.~1.7]{MS22}}]\label{t:ineq-MS-supercrit}
Suppose the positivity conditions \eqref{eq:ineq-MS-Pos} hold and that
quintuple \ts $(P,x,a,\zb,\cb)$ \ts is supercritical.
Then the equality \eqref{eq:Sta-gen-eq-MS} holds \. \underline{if and only if} \.
\begin{equation}\label{eq:ineq-MS-supercrit-conditions}
\left\{\.\aligned
& 	\forall  \.  y  \in x\!\uparrow \ \exists \. r \in \{0,\ldots,k+1\}
		\quad \text{s.t.} \quad  y \. \succ \. z_{r} \ \text{ and } \ \ga(z_r,y) \. > \. a \ts - \ts c_{r}\\
&\forall  \,  y  \in x\!\downarrow \ \exists \. s \in \{0,\ldots,k+1\}
		\quad \text{s.t.} \quad  y \. \prec \. z_{s} \ \text{ and } \ \ga(y, z_{s}) \. > \. c_{s} \ts - \ts a
\endaligned\right.
\end{equation}
Additionally, verifying \eqref{eq:ineq-MS-supercrit-conditions} is in~$\ts\poly$.
\end{thm}

Theorem~\ref{t:ineq-MS-supercrit} shows that deciding whether a supercritical quintuple
gives the equality case for the generalized Stanley inequality \eqref{eq:Sta-gen}
can be done in polynomial time.
Note that this generalizes Theorem~\ref{t:Sta-equality}, but not
Theorem~\ref{t:Sta1-equality} since there are critical cases for \ts $k=1$ \ts
as the following example shows.

\begin{ex}\label{ex:ineq-Sta-gen-crit-k=1}
Let  \ts $k=1$. Consider a poset  \. $P=(X,\prec)$, where $X := \{x,y_1,\ldots, y_{n-2},z\}$,
and \. $z \prec x$, $z \prec y_{n-2}$ \. are the only relations.  Let \ts $a:=n-1$,
\ts $b :=n-3$.  We have:
\[
\aN_{z \ts b}(P,x,a) \, = \, \aN_{z \ts b}(P,x,a+1) \, = \, \aN_{z \ts b}(P,x,a-1) \ = \  2 \. (n-3)!\ts,
\]
so \eqref{eq:Sta-gen} is an equality.
Note that \. $(P,z,a,z,b)$ \. is critical but not supercritical.  Indeed,  in this
case we have \. $\La = \{(1,2)\}$, \. $c_1=b$, \. $c_2=n+1$, \. and
\[ \ga(z,\wh 1)  \, = \,  \big|\{x, y_{n-2}\}\big| \, = \, 2 \, = \, (n+1) \. - \. (n-3) \. - \. 2 \,  = \,  c_{2} \. - \. c_1 \. - \. 2.   \]
\end{ex}

\smallskip

\subsection{Equality conditions, negative results}\label{ss:ineq-Sta-gen-equality-not-PH}
Corollary~\ref{c:Sta1-equality-P} is in sharp contrast with the following result:

\smallskip

\begin{thm}[{\rm \cite[Thm~1.3]{CP-AFequality}}{}] \label{c:Sta-gen-equality-not-PH}
Fix \. $k\ge 2$.  Then the equality verification of the generalized Stanley inequality~\eqref{eq:Sta-gen}
is not in the polynomial hierarchy unless the polynomial hierarchy collapses
to a finite level:
$$\big\{\aNr_{\bz \bc}(P, x,a)^2 =^? \aNr_{\bz \bc}(P, x,a+1) \cdot \aNr_{\bz \bc}(P, x,a-1)\}\in \PH \ \ \Longrightarrow \ \ \PH=\Sigmap_m  \quad \ \text{for some} \ \, m\ts.
$$
\end{thm}


The proof is based on the following result of independent interest.

\begin{thm}\label{t:Sta-gen-equality-FlatLE}  \
$\big\{\aNr(P,x,a) =^? \aNr(P,x,a+1)\big\}\in \PH \ \ \Longrightarrow \ \ \PH=\Sigmap_m  \quad \ \text{for some} \ \, m\ts.$
\end{thm}

Compare this with the following consequence of \. $(2) \Leftrightarrow (3)$ \.
in Theorem~\ref{t:Sta-equality}\.:

\begin{cor} \label{c:Sta-equality-FlatLE}
Deciding whether the following holds is in~$\ts \poly${}\ts{}$:$
$$\aNr(P, x,a-1) =\aNr(P, x,a) = \aNr(P, x,a+1).$$
\end{cor}

In particular, deciding if the distribution \. $\big\{\aN(P,x,a)\big\}$ \.
has a unique mode is not in~$\PH$, unless $\PH$ collapses.



\medskip

\section{Examples and applications}\label{s:ex}

In this section, we present a selection of poset classes for which the number of
linear extensions is interesting either because it is easy to compute, or
 because it is provably hard to compute.

\subsection{Bounded width} \label{ss:ex-width-two}
Posets \ts $P=(X,\prec)$ \ts of bounded width are especially elegant and have many
properties that general posets do not have.  In this case computing the number
of linear extensions \ts $\LE\in \FP$ \ts via dynamic programming.  \defng{Young diagrams} \ts
of bounded height represent especially nice examples of posets of bounded width,
see below.

For posets of width two with a fixed partition of $P$ into two chains,
the number $e(P)$ of linear extensions can be viewed as the number of
certain grid walks (lattice paths) in~$\zz^2$, see~$\S$\ref{ss:proof-lattice}.
For posets of width three, beside Young diagrams (see below), there
is an interesting \defng{Kreweras--Niederhausen poset} \ts $(A_2\oplus C_1)\times C_n$
\ts \cite{KN81,HR22}.  This poset
has a \defng{Kreweras number} \ts of linear extensions, see also
\cite[\href{http://oeis.org/A006335}{A006335}]{OEIS}.

%
%

\smallskip

\subsection{Series-parallel and $N$-free posets} \label{ss:ex-SP}
The class of \defn{SP~posets} \ts generalizes forests, and has an
extremely easy structure.  Given the \defng{SP~decomposition},
one can the following formulas to compute the number of linear
extensions:
$$e\big(P\oplus Q\big) \, = \, e(P) \, e(Q) \quad \text{and} \quad
e\big(P + Q\big) \, = \, \tbinom{m+n}{m} \, e(P) \, e(Q),
$$
where $P$ and~$Q$ have $m$ and~$n$ elements, respectively.
Given a SP~poset,  the SP~decomposition can be found in polynomial
time \cite{VTL82}.

Recall that SP posets can be characterized by not having poset~$N$
as an induced subposet (see~$\S$\ref{ss:def-posets}).  There is a
closely related class of \defn{$N$-free posets} \ts whose cover
graph does not contain~$N$.  There is a more general notion of
decomposition in this case.  We refer to \cite{HJ85,Moh89} for the
introduction to this class of posets.  See \cite{FM14} for bounds
on the numbers of linear extensions and a simple dynamic
programming algorithm.

\smallskip

\subsection{Young diagrams} \label{ss:ex-yd}
For a subset \ts $S\ssu \nn^2$, denote by \ts $P_S:=(S, \preccurlyeq)$, where \ts
$(i,j) \preccurlyeq (i',j')$ \ts if and only if \ts $i\le i'$ \ts and
\ts $j\le j'$.
Let \ts $\la=(\la_1,\ldots,\la_\ell)$ \ts be a \defn{partition of~$n$}, i.e.\ \ts
$\la_1 \ge \ldots \ge \la_\ell >0$, $|\la|:=\la_1+\ldots + \la_\ell=n$.
\defn{Young diagram} \ts is a subset \.
$S_\la:=\{\ts (i,j)\in \nn^2 \.: \. 1\le i \le \ell, \. 1\le j \le \la_i\ts\}$.
Let \ts $P_\la:=(S_\la,\preccurlyeq)$.
For example, \defn{Catalan poset} \ts $\Cat_m$ \ts corresponds to the partition
\ts $\la=(m,m)$, with $n=2m$.  Famously, \ts $e(\Cat_m)=\frac{1}{m+1}\binom{2m}{m}$ \ts
in the \ts \defn{Catalan number}, see e.g.\ \cite[$\S$6.2]{Sta-EC} and \cite{Sta-catalan}.

Linear extensions of \ts $P_{\la/\mu}$ \ts are called \defn{standard Young tableaux}
\ts of shape \ts $\la/\mu$.  We use \ts $e(\la/\mu) := e(P_{\la/\mu})$ \ts to
simplify the notation.
The \defn{hook-length formula} \ts
by Frame, Robinson and Thrall \cite{FRT54}, states:
\begin{equation}\label{eq:HLF} 
e(\la) \ = \ n!\, \prod_{(i,j)\in \lambda}\, \frac{1}{h_\la(i,j)} \,,
\end{equation}
where \ts $h_\la(i,j)=\la_i+\la'_j-i-j+1$ \ts is the \defn{hook-length}
in~$\la$.  This implies that \ts $\LE \in \FP$ \ts for Young diagram shapes.
See \cite[$\S$11.2]{Pak-OPAC} for a list of different proofs.

More generally, \defn{skew Young diagram} \ts $\la/\mu$ \ts is a pair of
partitions with \ts $S_\mu \subseteq S_\la$\ts. We use \ts $S_{\la/\mu}:=S_\la\sm S_\mu$.
Denote by \ts $P_{\la/\mu}:=(S_{\la/\mu},\preccurlyeq)$ \ts the corresponding
subposet of \ts $P_\la$\ts, and let \ts $|\la/\mu|:=|S_{\la/\mu}|$.
The \defn{Aitken--Feit determinant formula} \ts \cite{Ait,Feit}, states:
\begin{equation}\label{eq:AFDF} 
e(\la/\mu) \ = \ n!\.
\det \left( \frac{1}{(\la_i-\mu_j -i+j)!}\right)_{i,j=1}^{\ell}\,.
\end{equation}
This implies that \ts $\LE \in \FP$ \ts for skew Young diagram shapes as well.

There are several notable positive summation formulas for \ts $e(P_{\la/\mu})$,
called the \defng{Naruse hook-length formula} (NHLF), see \cite{Kon20,MPP17,MPP18a},
\defng{Okounkov--Olshanski formula}  \cite{MZ22,OO98}, and the
\defng{flipped hook-length formula} \cite[$\S$9.1]{Pak21}.  Among the
implications let us single out two inequalities:
\begin{equation}\label{eq:ex-NHLF-ineq}
e(\la/\mu) \ \ge \ n! \, \prod_{(i,j)\in \lambda}\, \frac{1}{h_\la(i,j)}
\ \quad \text{and} \ \quad \prod_{(i,j)\in \lambda}\, h_\la(i,j) \ \le \
\prod_{(i,j)\in \lambda}\, h_{\la^\ast}(i,j)\ts,
\end{equation}
where \ts $h_{\la^\ast}(i,j) := i+j-1$ \ts denotes the \defn{dual hook-length}
(hook-length in a skew shape rotated \ts $180^\circ$, corresponding to the
dual poset).  The first inequality in \eqref{eq:ex-NHLF-ineq}
is a strong improvement over \eqref{eq:BW}
in this case \cite{MPP18b}.  The second follows from the first, and
is a variation on \eqref{eq:BW-anti-hooks}
in the case of Young diagrams, see \cite[$\S$12.1]{MPP18b} and \cite{PPS20}.

Many inequalities for the numbers of linear extensions become surprising
in the language of standard Young tableaux.  For example,
Fishburn's inequality \eqref{eq:Fish-ineq} in this case states:
\begin{equation} \label{eq:ex-Bjo}
e(\la) \cdot e(\mu) \, \le \, e(\la\cup \mu) \cdot e(\la \cap\mu).
\end{equation}
See \cite{Bjo11} for a direct proof based on the HLF~\eqref{eq:HLF}.
Similarly, the generalized Fishburn
inequality \eqref{eq:Fish-ineq-CP} in this case is due to
Lam--Pylyavskyy \cite{LP07}, and states:
\begin{equation} \label{eq:ex-LP}
e(\la/\al) \cdot e(\mu/\be) \, \le \, e\big((\la\cup \mu) \ts / \ts (\al\cup \be)\big)
\cdot e\big((\la \cap\mu)\ts / \ts (\al\cap \be)\big).
\end{equation}

Let us single out the following
immediate corollary from Theorem~\ref{t:ineq-four-elements}.
For a partition~$\la$, a \defn{conjugate partition} \ts $\la'$ \ts is
obtained by reflection of \ts $S_\la$ \ts across the \ts $i=j$ \ts line.
We say that \ts $\la$ \ts is a \defn{self-conjugate partition}, if \ts $\la=\la'$.

\begin{cor}[{\rm \cite[Cor.~4.1]{CP-corr}}{}] \label{c:ex-YT}
Let \ts $\la/\mu$ \ts be a skew shape, let \ts $x,y\in S_{\la/\mu}$ \ts be
corners, and let \ts $v,w\in S_{\la/\mu}$ \ts be a boundary square adjacent
to~$x$ and~$y$, respectively. Then we have:
\begin{equation}\label{eq:ex-YT}
e\big(\la/\mu-x-y\big)^2 \ \ge \ e\big(\la/\mu-x-v\big) \.\cdot \. e\big(\la/\mu-y-w\big).
\end{equation}
In particular, when \ts $\la$ \ts and \ts $\mu$ \ts are self-conjugate,
\ts $x=(i,j)$ \ts and \ts $y=(j,i)$, we have:
\begin{equation}\label{eq:ex-YT-sc}
e\big(\la/\mu-x-y\big) \, \ge \, e\big(\la/\mu-x-v\big)\ts.
\end{equation}
\end{cor}

Special cases of \eqref{eq:ex-YT} when dominos \ts $(x,v)$ \ts and \ts $(y,w)$
are in the same directions also follow from the Schur function inequalities
in~\cite{LP07}, but \eqref{eq:ex-YT-sc} does not extend to a Schur functions.
We refer to \cite[$\S$4]{CP-corr} for further applications of poset
inequalities to the number of standard Young tableaux.

\begin{rem}\label{r:ex-HCF}
We also have product formulas for the order polynomial and for the GF for $P$-partitions:
\begin{equation}\label{eq:ex-SCF}
\Om(P_\la,t) \, = \, \prod_{(i,j)\in \lambda}\, \frac{t+i-j}{h_\la(i,j)} \qquad
\text{and} \qquad
 \Om_q(P_\la) \, = \, \prod_{(i,j)\in \lambda}\, \frac{1}{1-q^{h_\la(i,j)}} \,.
\end{equation}
Both formulas are special cases of \defng{Stanley's hook-content formula} \ts for \ts $\Om_q(P_\la,t)$,
see \cite[$\S$7.21]{Sta-EC}.  We refer to \cite{HG76,Kra99,Pak01} for bijective proofs of these
formulas, and to \cite{Hop20} for other product formulas for the order polynomial.
Note that the NHLF was extended to \ts $\Om_q(P_{\la/\mu})$ \ts in \cite{MPP18a}.
Finally, both \eqref{eq:ex-Bjo} and \eqref{eq:ex-LP} extend to
inequalities for Schur functions, see \cite{CP-multi,LP07,LPP07}.
\end{rem}

\begin{rem}\label{r:ex-shifted}
Define \defn{shifted Young diagrams} \ts as posets given by intersections of the
usual Young diagrams and cone \ts $i\le j$.  Much of the work on the (usual) Young
diagrams and standard Young tableaux directly translates to this case.  We refer
to \cite[$\S$7.20]{Sta-EC} as the starting point and further references.
\end{rem}

\smallskip

\subsection{Ribbon posets} \label{ss:ex-Euler-numbers}
Let \ts $Z_m:=P_{\la/\mu}$ \ts be a height two poset with \ts $n=2m-1$ \ts elements,
corresponding to the skew Young diagram \ts $\la/\mu:=\de_m/\de_{m-2}$\ts,
where \ts $\de_m:=(m,\ldots,2,1)$.  These are called \defn{zigzag posets}.
Linear extensions of \ts $Z_m$ \ts are in bijection with
\defn{alternating permutations} \ts $\si\in S_{2m-1}$ \ts s.t.\ \ts
$\si(1)>\si(2)<\si(3)>\si(4)<\ldots$ \ts  Then \ts $e(Z_m)=E_{2m-1}$
\ts are the \defn{Euler numbers}, see e.g.\
\cite[\href{http://oeis.org/A000111}{A000111}]{OEIS} and \cite{Sta-alt}.

In the context of Sidorenko's inequality~\eqref{eq:sid-BT},
it is easy to see that there is a width two poset \ts $Q_m$ \ts such that \.
$\Ga(Q_m) = \ov{\Ga(Z_m)}$ \. and \. $e(Q_m)=F_n$ \. is the
\defng{Fibonacci number}. Now \eqref{eq:sid-BT} gives \.
$E_n \cdot F_n \ge n!$ \. in this case \cite{MPP-phi}.
Moreover, this inequality was proved in \cite[Lem~4]{MPP-phi} by
an explicit surjection.

Fix \ts $x:=(m,1)$.  It is easy to see that triangle of numbers \ts
$a(m,k):=\aN(Z_m,x,k)$ \ts are \defn{Entringer numbers}
\cite[\href{http://oeis.org/A008282}{A008282}]{OEIS}.
Stanley's inequality~\eqref{eq:Sta} proves their log-concavity:
\begin{equation}\label{eq:ex-Entringer}
a(m,k)^2 \. \ge \. a(m,k+1) \. \cdot \. a(m,k-1) \quad \text{for} \ \ 1\le k \le 2m-2\ts.
\end{equation}
See \cite[Ex.~1.38]{CP-atlas} for a $q$-analogue of this example
in the style of Theorem~\ref{t:Sta-weighted}; see also \cite{B+,G+}
for other generalizations.

We note that zigzag posets are special cases of \defng{ribbon posets} \ts $P_{\la/\mu}$,
which correspond to skew shapes \ts $\la/\mu$ \ts with at most one square in every diagonal.
Linear extensions of such \ts $P_{\la/\mu}$ \ts are in bijection with permutations which have
a given set of descents, so \ts $e(\la/\mu)$ \ts satisfies \defng{MacMahon's determinant formula},
see e.g.\ \cite[$\S$2.2.4]{Sta-EC}.  This formula was further generalized to
\defng{mobile posets} \ts defined in \cite{GGMM}.  Note also that among all ribbon
posets, zigzag posets maximize the number of linear extensions \cite[Cor.~1.6.5]{Sta-EC}.
This result is due to Niven (1968) and  de~Bruijn (1970), further generalized
in \cite{Sta88,Iri17}.

%

\smallskip

\subsection{Permutation posets} \label{ss:ex-perm}
Fix \ts $\si\in S_n$.  The \defn{permutation poset} \ts $P_\si = ([n],\prec)$
\ts is defined as:
$$i \precc j \quad \Leftrightarrow \quad i\le j \ \ \text{and} \ \ \si(i)\le \si(j) \ts.
$$
Permutation posets are also called \defn{two-dimensional posets} \ts in the literature,
see e.g.\ \cite{Tro95,West21}. This class includes all posets \ts $P_{\la/\mu}$ \ts and
all posets of width two.  Note that the height and the width in \ts $P_\si$ \ts are
longest increasing and the longest decreasing subsequence in~$\si$, which are extremely
well studied, see e.g.\ \cite{Rom15}.  It is known that \ts $\LE$ \ts is $\SP$-complete
in this case \cite{DP18}.

The (weak) \defn{Bruhat order} \ts $\cB_n=(S_n,\lhd)$ \ts is defined as follows:
\ts $\tau \unlhd \pi$ \ts if and only if \. $\tau \cdot \ups = \pi$ \. for some
\ts $\ups \in S_n$ \ts such that \.
$\inv(\tau) + \inv(\ups) = \inv(\pi)$.  It is known that as a subset of \ts $S_n\ts$,
the set of linear extensions \ts $\Ec(P_\si)$ \ts is the lower ideal \ts $\si\!\downarrow$
\ts in \ts $\cB_n\ts$, see \cite[Lem.~5]{FW97} (see also \cite{BW91a,Reu96}).

Denote by \ts $\ov \si := \big(\si(n),\ldots,\si(1)\big)$ \ts the \defn{reverse permutation}.
Note that \ts $P_\si$ \ts and \ts $P_{\, {\ov \si}}$ \ts satisfy condition \eqref{eq:sid-chain-condition}
of the Sidorenko inequality \eqref{eq:sid}, which gives:
\begin{equation}\label{eq:ex-Sid-perm}
e\big(P_{\si}\big) \. \cdot \. e\big(P_{\, {\ov \si}}\big) \, \ge \, n!.
\end{equation}
To generalize this, denote by \.
$$\INV(\si) \ :=  \ \big\{(i,j) \. : \. \si(i) > \si(j), \. 1\le i< j \le n\big\}
$$
the set of inversions of~$\si$, so \. $\inv(\si) = |\INV(\si)|$. Suppose that
$$
\INV(\tau) \. \cap \. \INV(\ups) \, = \, \emp \quad \text{and} \quad
\INV(\tau) \. \cup \. \INV(\ups) \, = \, \INV(\pi).
$$
Taking \ts $i\gets 2$, \. $P_1\gets P_\tau$\ts, \. $P_2\gets P_\ups$ \. and \. $Q\gets P_\pi$ \.
gives conditions \eqref{eq:sid-int-condition} of Theorem~\ref{t:sid-BT}.  Then
\eqref{eq:sid-gen} gives
\begin{equation}\label{eq:sid-perm-BT}
e\big(P_{\tau}\big) \. \cdot \. e\big(P_{\ups}\big) \, \ge \,  e\big(P_{\pi}\big).
\end{equation}
Inequality \eqref{eq:sid-perm-BT} can be
interpreted as defining a metric on the set of permutation \cite[$\S$9.3]{Ilo08},
see also \cite[Cor.~8.7.2]{BT06}.

\smallskip

\subsection{Interval orders and semiorders} \label{ss:ex-int}
Let \ts $X$ \ts be a collection of $n$ intervals \ts $I_1,\ldots,I_n\ssu \rr$.
Define \ts $P = (X,\prec)$, where \ts $I_i \prec I_j$ \ts if \ts $x<y$ \ts
for all \ts $x\in I_i$ \ts and \ts $y \in I_j\ts$.  Such posets are called
\defn{interval orders}.  They are characterized by not having
\ts $(C_2 + C_2)$ \ts as an induced subposet which implies that the
recognition problem of interval orders is in~$\poly$, see e.g.\
\cite[$\S$6]{Moh89}.

Additionally, if all intervals have unit lengths, such posets are called
\defn{semiorders} \ts and \defn{unit interval orders}, see
e.g.\ \cite{Sta-hyper}.  They characterized by not having \ts $(C_2 + C_2)$ \ts
and \ts $(C_3 + C_1)$ \ts as induced subposets (Scott and Suppes, 1958).
We refer to \cite{Fis85} for a thorough treatment and further references.

Let \ts $P=(X,\prec)$ \ts be a poset on \ts $|X|=n$ \ts elements.
Let \ts $\ga(P) := |\{(x,y)\. : \. x\prec y, \. x,y\in X\}|$ \ts denote
the number of comparable pairs in~$P$.  Finally, let \ts $e(n,k)$ \ts be the
maximal number of linear extensions among all posets $P$ on $n$ elements
with \ts $\ga(P)=k$.

\begin{prop}[{\rm Fishburn--Trotter~\cite{FT92}}{}] \label{p:ex-semi}
Let \ts $P=(X,\prec)$ \ts be a poset on \ts $|X|=n$ \ts elements such that \ts
$\ga(P)=k$ \ts and \ts $e(P)=e(n,k)$.  Then \ts $P$ \ts is a semiorder.
\end{prop}

See also \cite[Thm~8.7]{Tro95} for a short proof, and \cite{MPI18}
for the asymptotic analysis of numbers \ts $e(n,k)$.  This is a rare
extremal result on linear extensions of finite posets.  Finally, a larger
class of \ts $(C_3 + C_1)$-free posets is of interest in both Enumerative
and Algebraic Combinatorics, see e.g.\ \cite{GMR14}.

\smallskip

\subsection{Random posets} \label{ss:ex-random}
There are several interesting models of random finite posets studied in the literature,
neither of which is especially satisfactory, at least when compared to random graph models.
We refer to \cite{Bri93} for a survey with many helpful references.

First, one can consider uniform (unlabeled) posets
of $n$ elements.  Kleitman and Rothschild \cite{KR75} gave a sharp
asymptotic estimate on the number \ts $p(n)$ \ts of such posets
\begin{equation}\label{eq:ex-KR}
\log_2 p(n) \, = \, \tfrac{n^2}{4} \. + \. \tfrac{3n}{2} \. + \. O(\log n),
\end{equation}
see  also \cite[\href{http://oeis.org/A000112}{A000112}]{OEIS}.
It follows from the proof that uniform random posets have height three and
can be partitioned into three antichains of sizes roughly \ts  $\frac{n}{4}$\ts , \ts
$\frac{n}{2}$ \ts and \ts $\frac{n}{4}$\ts, respectively.  The number of linear
extensions is very large in this case and concentrated around \. $(\frac{n}{2})! (\frac{n}{4})!^2$,
cf.~\eqref{eq:basic-antichains}.  The number of pairs $(x,y)\in X^2$ such that
\ts $\bP\big(f(x)<f(y)\big)\sim \frac12$ \ts is asymptotically \. $\frac{3\ts n^2}{16}$ \.
\cite{Kor94}.  A major disadvantage of this model is the
difficulty of sampling uniform posets (either labeled or unlabeled).

Next, one can consider random posets from families where the sampling is easy.
These include random \defng{bipartite posets}, defined as posets of height two
with relations given by a random bipartite graph.  Similarly, one can consider
random permutation posets \ts $P_\si\ts,$ or random semiorders
(there are Catalan number of them, see e.g.\ \cite[$\S$3.180]{Sta-catalan}
and references therein).
A curious model is given by the
transitive closure of a random subset of relations $i\prec j$ on $[n]$,
where $1\le i < j\le n$, see an extensive discussion in \cite{Bri93}.

Finally, one can consider random posets \ts $P_\la$ \ts corresponding to
Young diagrams of size~$n$.  The problem of determining \ts $e(P_\la)$ \ts
is well-studied and is especially important in combinatorial representation
theory and again closely related to the study of longest increasing subsequences,
see e.g.\ \cite{Rom15}.  Finally, see \cite{MPI18} for the asymptotics of
random interval orders.



\medskip

\section{Computational aspects}\label{s:CS}

\subsection{Counting complexity}\label{ss:CS-SP}
In \cite{BW91b}, Brightwell and Winkler showed that computing the number of linear
extensions is $\SP$-complete.  This was refined to posets of height two,
posets of dimension two, and to \defng{incidence posets} in~\cite{DP18}.
In \cite{Sta97}, Stachowiak proved that computing \defng{sign-imbalance} \ts
of posets of height two is $\SP$-hard, by giving a simple parsimonious
reduction to computing \ts $e(P)$.

Additionally, Dittmer showed that the parity of the number of linear
extensions is \ts $\oplus P$-complete for posets of dimension two \cite[Thm~1.1.2]{Dit19}.
Combined with Stachowiak's proof above implies that the parity of \ts $e(P)$ \ts is \ts
$\oplus P$-complete for height two posets.  This is in contrast with Soukup's theorem
that deciding whether poset \ts $P$ \ts is \defng{sign-balanced} \ts is in~$\ts \poly$,
see \cite{Sou23+}.

%

\smallskip

In the opposite direction, the are several classes of posets where the number
of linear extensions can be computed in polynomial time.

\begin{thm}  \label{t:CS-LE}
$\LE$ \ts is in \ts $\FP$ \ts for:
\begin{enumerate}
\item[$(1)$] \ bounded width posets,
\item[$(2)$] \ skew Young diagrams, \ts see \eqref{eq:AFDF},
\item[$(3)$] \ series-parallel posets, \ts see~$\S\ref{ss:ex-SP}$ \ts and \ts \cite{VTL82},
\item[$(4)$] \ posets with bounded decomposition diameter \. \cite{HM87},
\item[$(5)$] \ posets whose covering graphs have disjoint cycles \. \cite{Atk89} $($e.g.\ trees \cite{Atk90}$)$,
\item[$(6)$] \ $N$-free posets with bounded activity \. \cite{FM14},
\item[$(7)$] \ posets with bounded treewidth \. \cite{EGKO},
\item[$(8)$] \ mobile posets \. \cite{GGMM}.
\end{enumerate}
\end{thm}

\nin
Some of these results have known $q$-analogues.  Namely, computing polynomial \ts
$e_q(P)\in \nn[q]$ \ts is in \ts $\poly$ \ts for width two posets (see e.g.\ \cite{CPP-CP}),
skew Young diagrams (see e.g.\ \cite{CP-multi}), series-parallel posets (folklore),
and posets whose covering graph is a tree \cite[$\S$6.3]{GGMM}.

In contrast with~$(6)$, for general $N$-free posets, the problem \ts $\LE$ \ts
is conjectured to be \ts $\SP$-complete.  We are not aware if the number of
linear extension of interval order has been studied or conjectured to be
$\SP$-complete.

For the order polynomial, the coefficient \ts $[t^2] \, \Om(P,t)$ \ts is the
number of (lower) order ideals in~$P$, which is $\SP$-complete \cite{PB83}.
We refer to \cite{FS86} for more on complexity of computing \ts $\Om(P,t)$.
Finally, we note that the
probability \ts $\bP\big(f(x)<f(y)\big)$ \ts and the average height
\ts $h(P,x)$ \ts are \ts $\SP$-hard \cite[$\S$5]{BW91b}.

\smallskip

\subsection{Random generation and approximate counting}\label{ss:CS-random}
The problem of \defng{random generation} \ts of linear extensions \ts $f\in \Ec(P)$ \ts
is closely related to \defng{approximate counting}, i.e.\ computing the \ts $(1\pm \ep)$ \ts
approximation of \ts $e(P)$.  The key is the \defng{self-reducibility} \ts property:
if one can obtain a strong approximation of the ratio \ts $e(P)/e(P\cap \{x\prec y\})$,
taking the product of these ratios gives an approximation for~$e(P)$.
We refer to \cite[$\S$28]{Vaz01} for the introduction to this technology.

The  first FPRAS for \ts $\LE$ \ts was obtained by Matthews \cite{Mat91} using a
geometric random walk on the order polytope \ts $\cO_P$\ts, based on a
Markov chain (MC) with the mixing time upper bound \ts $\mix = O\big(n^8(\log n)^3\big)$.
There are now several rapidly mixing Markov chains on \ts $\Ec(P)$ \ts worth discussing.
All Markov chains start at \ts $f\in \Ec(P)$, but have different steps described as follows.

\smallskip

\nin
$(1)$ \ts  \emph{
Choose uniform \ts $k\in [n-1]$, and switch the values \ts $k \leftrightarrow (k+1)$ \ts in~$f$ if possible.} \.
This MC was introduced by Karzanov and Khachiyan \cite{KK91} who proved \ts $\mix=O(n^6\log n)$
bound  using conductance estimates.
This bound was steadily improved down to \ts $O(n^3\log n)$ \cite{Wil04}.  Furthermore,
Wilson showed (ibid.), that for some posets this bound cannot be improved.

\smallskip

\nin
$(1')$ \ts\emph{
Choose  \ts $k\in [n-1]$ \ts with probability proportional to $k(n-k)$.  Proceed as in \ts $(1)$.} \.
This modification is due to Bubley and Dyer \cite{BD99}, who proved the \ts $\mix = O(n^3\log n)$ \ts bound
using a coupling argument.

\smallskip

\nin
$(2)$ \ts \emph{
Choose uniform \ts $x\in X$, and take a partial promotion to~$x$ in~$f$.} \.
This MC was introduced by Ayyer, Klee and Schilling \cite{AKS14a}, who gave
a bound \ts $\mix=O(n^3)$ \ts in \cite{AKS14b}.  A better bound
a \ts $\mix = O(n^2)$ \ts was obtained in \cite{PS18}.

\smallskip

\nin
$(3)$ \ts \emph{
Choose uniform \ts $i<j$.  For all \ts $k\in \{i,\ldots,j-1\}$ \ts in this order,
switch the values \ts $k \leftrightarrow (k+1)$ \ts in~$f$ if possible.} \.
This MC was introduced by Ayyer, Schilling and Thi\'ery \cite{AST17}, who proved
the \ts $\mix =O(n^2\log n)$ \ts upper bound and conjecture that \ts $\mix =O(n\log n)$.

\smallskip

For posets of height two, a conjectured rapidly mixing MC was given in~\cite{CRS09}.
A closely related MC was analyzed in~\cite{Hub14}.  Let us also mention
Huber's algorithm \cite{Hub06} for \defng{perfect sampling} \ts of linear
extensions of general posets.

Finally, the problem of (exact) \defng{uniform generation} \ts is also of interest
in both Combinatorics \cite{NW78} and Theoretical Computer Science \cite{JVV86}.
For classes of posets where the counting problem is polynomial, i.e.\ \ts $\LE\in \FP$,
the self-reducibility gives a polynomial time algorithms for the uniform generation.
Faster algorithms exists for Young diagrams \cite{GNW79,NPS97} (see also~\cite{SS17}),
special skew Young diagrams \cite[$\S$5.6]{H+}, and for series-parallel posets \cite{BDGP}.

\smallskip

\subsection{Graph of linear extensions}\label{ss:CS-graph}
A common consequence of the Markov chains discussed above, is the following basic result:

\begin{prop}[{\rm folklore}{}] \label{p:CS-random-connect}
Let \ts $P=(X,\prec)$ \ts be finite poset, and let \ts $f,g\in \Ec(P)$ \ts
be linear extensions.  Then \ts $f$ \ts and \ts $g$ \ts are connected by
a sequence of \ts \ts $k \leftrightarrow (k+1)$ \ts switches, $1\le k <n$.
\end{prop}

The proposition is a folklore result repeatedly rediscovered in
different contexts.  
For a brief overview of generalizations and further references,
we refer to the discussion which follows \ts \cite[Prop.~1.3]{DK21}.

Consider a graph \ts $\gg(P)$ \ts with vertices \ts $\Ec(P)$, and with
edges corresponding to switches, see e.g.~\cite{Mas09}.
The proposition above proves connectivity of  \ts $\gg(P)$.
In fact, the distance between any two linear extensions can be computed in
polynomial time, see \cite[$\S$6]{BW91a} and \cite[Prop.~2.2]{Naa00}.
On the other hand, the diameter of
\ts $\gg(P)$ \ts is $\NP$-hard to compute \cite[Thm~5]{BM13}.

Suppose \ts $e(P)$ \ts is even.
Ruskey noted \cite{Rus92}, that if \ts $\gg(P)$ \ts
has a \defng{Hamiltonian path}, then \ts $P$ \ts is sign-balanced (see~$\S$\ref{ss:def-LE}).
Motivated by the problem of \defng{listing} \ts all linear extensions,
he stated:

\begin{conj}[{\rm Ruskey~\cite[$\S$5]{Rus92}}{}]\label{conj:Ruskey}
Let \ts $P$ \ts be a finite sign-balanced poset.  Then \ts $\gg(P)$ \ts has a
Hamiltonian path.
\end{conj}

We refer to \cite[$\S$5.10]{Rus03} for the introduction,
to~\cite{CW95} for algorithmic aspects of this problem,
and to a recent survey \cite[$\S$5.5]{Mut23} for further references.

\smallskip

\subsection{Coincidence problem and concise functions}\label{ss:CS-coincidence}
Following \cite{CP-domino}, consider the \defn{coincidence problem}:
$$
\text{\sc C}_e \ := \ \big\{\ts e(P) \. = ^? \ts e(Q) \ts \big\}.
$$
We conjecture that this problem is $\CEP$-complete under Turing reductions.

\begin{thm}[{\rm \cite[Thm~1.4]{CP-domino}}{}] \label{t:CS-coincidence} \
$\text{\sc C}_e \ts \in \ts \PH \ \Longrightarrow \ \PH \ts = \ts \Sigmap_m$  \
for some \ $m \ge 1$.

\nin
Moreover, this holds when \ts $e$ \ts is restricted to permutation posets.
\end{thm}

A key lemma in the proof is the following result of independent interest.

\begin{thm}[{\rm \cite[Thm~1.1]{KS21}}{}] \label{t:CS-KS}
Let
\begin{equation*} 
\cT_e(n) \ := \ \big\{\. e(P) \, \. : \, \. P=(X,\prec), \, \, |X|=n, \, \. \width(P)=2 \.\big\}.
\end{equation*}
Then:
\begin{equation}\label{eq:KS-poset}
\cT_e(n)  \, \supseteq \, \big \{1,\ldots, c^{n/(\log n)} \big\} \quad
\text{for some} \ \ c>1.
\end{equation}
\end{thm}

The authors believe that the upper end in \eqref{eq:KS-poset} can be
improved to \ts $c^n$.  

\begin{conj}[{\rm \cite[Conj.~7.4]{KS21}}{}]\label{conj:KS-exp}
We have:
\begin{equation}\label{eq:KS-poset-conj}
\cT_e(n)  \, \supseteq \, \big \{1,\ldots, c^{n} \big\} \quad
\text{for some} \ \ c>1.
\end{equation}
\end{conj}

In \cite{CP-CF}, we prove that this conjecture follows from 
\defng{Zaremba's conjecture} and an array of somewhat weaker conjectures.  
We also show that a density~1 version of the conjecture follows from 
the \defng{Bourgain--Kontorovich theorem} (2014). 

Note that since $e(P) \le n!$ is the only natural upper 
bound for posets on $n$ elements, one can ask if the  bound 
\eqref{eq:KS-poset-conj} can be further extended to \ts $e^{n\log n}$.  
The following conjecture implies this bound, and thus also Conjecture~\ref{conj:KS-exp}.  
It also implies
that Theorem~\ref{t:CS-coincidence} holds for posets of height two, see
\cite[Prop.~5.18]{CP-domino}.

\begin{conj}[{\rm \cite[Conj~5.17]{CP-domino}}{}] \label{conj:CS-coincidence-two}
For all sufficiently large \ts $m$, there is a poset $P$ of height two,
such that \ts $e(P)=m$.
\end{conj}

\smallskip

\subsection{Combinatorial interpretation}\label{ss:CS-comb}
For every inequality \ts $f\geqslant g$ \ts where \ts $f,g\in \SP$ \ts
are counting functions, one can ask if the defect \ts $(f-g)$ \ts
is in~$\SP$.  Informally, this is a question whether \ts $(f-g)$ \ts has
a combinatorial interpretation.  For example, the proof of
Theorem~\ref{t:CS-coincidence} implies that
$$
\big(e(P) \. - \. e(Q)\big)^2 \, \notin \, \SP \quad \text{unless \quad $\PH \ts = \ts \Sigmap_2$\ts.}
$$
We refer to \cite{Pak19} for the introduction to the problem of combinatorial
interpretation, to \cite{Sta-positivity} for a review of open problems on
combinatorial interpretation in Algebraic Combinatorics,
to \cite{IP22} for a careful treatment of polynomial inequalities,
and to \cite{Pak-OPAC} for a detailed survey.



\medskip

\section{Sorting probability}\label{s:sort}

In this section we summarize partial results and several variations on the
\ts $\frac13-\frac23$ \ts Conjecture.

\subsection{The \ts $\frac13-\frac23$ \ts Conjecture} \label{ss:sort-1323}
The following conjecture remains a major challenge in the area.
It was originally stated by Kislitsyn \cite{Kis68}, and independently
by Fredman \cite{Fre75}.

\begin{conj}[{\rm \ts \defn{$\frac13 - \frac23$ \ts conjecture}}{}] \label{conj:1323}
In every finite poset \ts $P=(X,\prec)$ \ts that is not a chain, there exist
two elements \ts $x,y\in X$, such that
\begin{equation}\label{eq:sort-1323}
\frac13 \ \le \ \bP\big[\ts f(x) \. < f(y)\ts\big] \ \le \ \frac23\,,
\end{equation}
where the probability is over uniform random \ts $f\in \Ec(P)$.
\end{conj}

Note that the constant \ts $\ve = \frac13$ \ts is optimal for \ts $P=C_2+ C_1$\ts.
A poset \ts $P=(X,\prec)$ \ts on $n$ elements is called \defn{$k$-thin} \ts
if \ts $\al(x)+\be(x)> n-k$, for all \ts $x\in X$.  In the opposite direction, poset
\ts $P$ \ts is called \defn{$(\epsilon,\de)$-dense} \ts if there is \ts $Y \subset X$,
$|Y|\ge \epsilon \ts n$, s.t.\ \. $\al(x)+\be(x) < \de \ts n$, for all \ts $x\in Y$.

\begin{thm} \label{t:sort-1323-summary}
Conjecture~\ref{conj:1323} holds for the following posets with \ts $n$ \ts elements:
\begin{enumerate}
\item[$(1)$] \ width two posets \. \cite{Lin84},
\item[$(2)$] \ posets with a symmetry \. \cite{GHP87},
\item[$(3)$] \ semiorders \. \cite{Bri89} \ts $($a concise proof was given in \ts \cite[Thm.~2.3]{Bri99}$)$,
\item[$(4)$] \ height two posets \. \cite{TGF92},
\item[$(5)$] \ posets with \. $|\min(P)| > C \sqrt{n}$ \. for some $C>0$ \. \cite{Fri93},
\item[$(6)$] \ posets with no chains of length \. $> 2 \log_2 \log n - C$, ibid.
\item[$(7)$] \ $(\epsilon,\de)$-dense posets, for all \ts $\epsilon>0$ \ts and some \ts $\de=\de(\epsilon)>0$, ibid.
\item[$(8)$] \ $6$-thin posets \. \cite{Pec08} \ts $($a weaker $5$-thin version was given in \cite{BW92}$)$,
\item[$(9)$] \ series-parallel and $N$-free posets \. \cite{Zag12},
\item[$(10)$] \ skew Young diagram posets \ts $P_{\la/\mu}$  \. \cite{OS18},
\item[$(11)$] \ posets whose cover graph is a forest  \. \cite{Zag19}.
\end{enumerate}
\end{thm}

\smallskip

Note that these classes are not completely disjoint.  For example,
for sufficiently large $n$, parts \ts $(5)$, \ts $(6)$ \ts and \ts $(7)$ \ts imply \ts $(4)$.
Part~$(4)$ is proved via Koml\'os theorem (see below).  Parts \ts $(5)-(7)$ \ts
are proved using geometric tools.  An alternative proof of $(10)$ given
in~\cite[$\S$3]{CPP-sort} uses the Naruse hook-length formula (cf.~$\S$\ref{ss:ex-yd}).
We refer to \cite{Bri99} for a well-written survey of early results and ideas.

\smallskip

\subsection{Weaker general bounds} \label{ss:sort-gen}
Currently, the best general bound is given by the following result:

\begin{thm}[{\rm Brightwell--Felsner--Trotter~\cite{BFT95}}{}] \label{t:sort-BFT}
In every finite poset \ts $P=(X,\prec)$ \ts that is not a chain, there exist
two elements \ts $x,y\in X$, such that
\begin{equation}\label{eq:sort-BFT}
\frac12\. - \. \frac1{2\sqrt{5}} \ \le \ \bP\big[\ts f(x) \. < f(y)\ts\big] \ \le \ \frac12\. + \. \frac1{2\sqrt{5}}\,,
\end{equation}
where the probability is over uniform random \ts $f\in \Ec(P)$.
\end{thm}

Here the lower bound is \. $\ve\approx  0.2764$, just shy of~$\frac13$\ts.
This is a small improvement over the first bound of the type \. $\ve \le \bP[f(x)<f(y)] \le 1-\ve$ \.
proved by Kahn and Saks \cite{KS84}, with \ts $\ve = \frac3{11} \approx 0.2727$.
The latter result uses the Kahn--Saks inequality \eqref{eq:KS},
applied to two elements \ts $x,y\in X$ \ts with
a small difference of average heights: \ts $|h(P,x)-h(P,y)|<1$, cf.~$\S$\ref{ss:ineq-height}.
The proof in~\cite{BFT95}  uses the \ts $a=b=1$ \ts
case of the cross-product conjecture (Theorem~\ref{t:CPC-summary}), and builds on~\cite{KS84}.

We note that an easier proof of a weaker bound with \ts $\ve = \frac{1}{2e} \approx 0.1840$,
was obtained by Kahn and Linial in \cite{KL}, using a variation on \defng{Gr\"unbaum's theorem} (1960),
which in turn is proved using the \defng{Brunn--Minkowski inequality} (rather than
the Alexandrov--Fenchel inequality used in the proof of the Kahn--Saks inequality).  We refer to
\cite[$\S$12.3]{Mat02} for a clean presentation of this proof, and to \cite{C+} for a survey of
algorithmic applications of the sorting probability, and to \cite{H+24,HR24} for best recent bounds.

\smallskip

\subsection{Stronger specialized bounds} \label{ss:sort-gen-KS}
For a poset \ts $P=(X,\prec)$, the \defn{sorting probability} \ts
is defined as
\begin{equation}\label{eq:sort-def-prob}
\de(P) \, := \, \min_{x,\ts y \ts \in \ts X} \, \big|\.\bP\big(f(x)<f(y)\big) \. - \. \bP\big(f(x)>f(y)\big)\ts\big|.
\end{equation}
In this notation, the \. $\frac13 - \frac23$ \ts Conjecture~\ref{conj:1323} claims that
the sorting probability \ts $\de(P) \le \frac13$ \ts unless \ts $P$ \ts is a chain,
while the bound \eqref{eq:sort-BFT} gives \ts $\de(P) \le \frac1{\sqrt{5}}$\ts.
In the cases when the conjecture is established, one can ask for better bounds.
For example, Friedman showed that for all \ts $\epsilon>0$ \ts and \ts
$C=C(\epsilon)$, we have
\. $\de(P) \. < \. \big(1 \ts - \ts \frac{2}{e} \ts +  \ts \epsilon\big)$ \.
in the cases \ts $(5)-(7)$ in Theorem~\ref{t:sort-1323-summary}.

\begin{conj}[{\rm Kahn--Saks \cite{KS84}}{}] \label{conj:1323-KS}  We have:
\begin{equation}\label{eq:1323-KS}
\de(P) \to 0 \quad \text{as} \quad \width(P)\to \infty.
\end{equation}
\end{conj}

Informally, this conjecture says that the sorting probability is small
for all posets of sufficiently large width.  Beside common sense,
there is relatively small evidence in favor of this conjecture.
Saks \cite{Saks85} suggests that \ts $\de(P)\le \frac{14}{39}$ \ts
when \ts $\width(P) \ge 3$, and gives an example where this bound
is tight.  Koml\'os \cite{Kom90} proved that \ts $\de(P)\to 0$,
for \ts $|\min(P)| = n/s(n)$ \ts and some \ts $s(n) = \om(1)$.

Curiously, it is known that \ts $\de(P)\to 0$ \ts for several large posets
of bounded width, such as several families of skew Young diagrams \cite{CPP-sort}.
For example, Panova and the authors proved \ts
$\de(P_\la) = O_\epsilon(1/\sqrt{n})$ \ts for \. $\la=(\la_1,\ldots,\la_\ell)$ \.
with \ts $\ell$ \ts fixed and \ts $\la_\ell > \epsilon\ts n$ \. \cite[Thm~1.3]{CPP-sort}.

For posets of width two which cannot be written as a linear sum
of chains and posets \ts $(C_2 + C_1)$, Sah~\cite{Sah21}
improved Linial's \ts $\de(P)\le \frac13$ \ts upper bound
to \ts $\de(P)< 0.3225$.  Sah also conjectured in \cite[Conj.~5.3]{Sah21},
that this approach extends to general posets (of any width), i.e.\ under
mild assumptions Conjecture~\ref{conj:1323}
can be made strict: \. $\de(P)< \frac{1}{3} - \ve$ \ts for some constant
\ts $\ve>0$.

Finally, for the extremely well studied \defn{Catalan poset} \ts
$\Cat_n$ \ts of width two (see~$\S$\ref{ss:ex-yd}), Panova and the authors
showed in \cite{CPP-Cat},
that \ts $\de(\Cat_n) = O(n^{-5/4})$, where the constant
$5/4$ is conjectured to be sharp.

\smallskip

\subsection{Gaps between heights} \label{ss:sort-height}
For a poset \ts $P=(X,\prec)$ \ts on \ts $|X|=n$ \ts elements, let
$$
\eta(P) \, := \, \min_{x \ts | \ts y} \. \big|\ts h(P,x) \. - \. h(P,y)\ts\big|
$$
denote the smallest gap between heights of incomparable elements in~$P$.
As we mentioned earlier, the proof by Kahn--Saks of \ts $\de(P) \le \frac{5}{11}$ \ts follows from
the observation that \ts $\eta(P) <1$.  Improving the latter bound directly translates
to a sharper for the sorting probability \ts $\de(P)$.

Unfortunately, there is a limit for this approach. Let
$$\vartheta \, := \, \frac14 \. \prod_{k=1}^\infty \left(1 - \frac{1}{2^k}\right)^{-1} \. \approx \. 0.8657\ts.
$$
Saks in \cite{Saks85}, gives a construction of posets \ts $P_n$ \ts such that \ts $\eta(P_n) \to \vartheta$.
He conjectures that this example is optimal:

\begin{conj}[{\cite{Saks85}}{}] \label{conj:Saks}
For every poset \ts $P$ \ts that is not a chain, we have \ts $\eta(P) \le \vartheta$.
\end{conj}

\smallskip

\subsection{Voting preferences} \label{ss:sort-preferences}
Let \ts $P=(X,\prec)$ \ts be a poset on \ts $n=|X|$ \ts elements and
let \ts $\bP$ \ts be defined over uniform \ts $f\in \Ec(P)$.
One can think of the average height \ts $h(P,x)$ \ts a way to
rank all elements in~$X$ (possibly, with ties).  This is not the
only natural approach, of course.

For elements \ts $x,y\in X$, we write \. $x \mapsto y$ \. if
\. $\bP\big(f(x)<f(y)\big) > \frac12$\ts.  Heuristically, this means
of the random linear ordering of $X$ which respects partial order ``$\prec$'',
element~$y$ typically has a smaller rank than~$x$.  In~\cite[$\S$4.2]{Kis68}, Kislitsyn
speculated that \. ``$\mapsto$'' \. is transitive.  This was disproved by
Fishburn soon after:\footnote{Fishburn was unaware of~\cite{Kis68} and
independently discovered the problem motivated by voting paradoxes. }

\begin{prop}[{\rm Fishburn~\cite{Fis74b}}{}]\label{p:sort-pref-Fishburn}
Relation \. ``$\mapsto$'' \. is not transitive, i.e.\ there exist a finite
poset \ts $P=(X,\prec)$ \ts and three elements \ts $x,y,z\in X$, s.t.\
\ts $x\mapsto y$, \ts $y\mapsto z$, and \ts $z\mapsto x$.
\end{prop}

The original example by Fishburn has \ts $n=31$ \ts elements.  It was shown
in \cite{FG90}, that the smallest such poset has $n=9$ elements.  See also
 \cite{EFG90} for an example of a poset of height two with \ts $n=15$ \ts elements.
Below is a quantitative version of this problem:

\begin{thm}[{\rm Fishburn~\cite{Fis86} and Kahn--Yu~\cite{KY98}}{}] \label{t:sort-KY}
Let \ts $P=(X,\prec)$ \ts be a finite poset, and let \ts $x\mapsto y$,
\ts $y\mapsto z$ \ts for some  \ts $x,y,z\in X$.  Then:
\begin{equation}\label{eq:sort-KY}
\bP\big(\ts f(x) \ts < \ts f(z)\ts \big) \, > \, \tfrac14\ts.
\end{equation}
On the other hand, the constant \ts $\frac14$ \ts in the RHS cannot be replaced
with \ts $\frac1e$\ts.
\end{thm}

Here the second part is due to Fishburn.   The bound \eqref{eq:sort-KY} is due
to Kahn--Yu, who write that Fishburn's  \ts $\frac1e$ \ts is
``likely to be the correct value'' to appear in \eqref{eq:sort-KY}.

\begin{question} \label{q:sort-KY}
What is the optimal constant in the RHS of \eqref{eq:sort-KY}?
\end{question}

We conclude with a weaker notion of the preference order which happens to be
transitive.  For elements \ts $x,y\in X$ \ts and \ts $\rho\ge \frac12$ \ts we write
\. $x\mapsto_\rho y$ \. if $\bP\big(f(x)<f(y)\big) > \rho$.

\begin{thm}[{\rm Yu~\cite{Yu98}}{}]\label{t:sort-Yu}
Fix \ts $\rho > 0.78005$.  Let \ts $P=(X,\prec)$ \ts be a finite poset,
and  let \ts $x\mapsto_\rho y$, \ts $y\mapsto_\rho z$ \ts for some \ts $x,y,z\in X$.
Then \. $x\mapsto_\rho z$.
\end{thm}

Let us also mention Friedman's observation \cite[$\S$2]{Fri93},
that there always exist a linear extension
\ts $g\in \Ec(P)$ \ts such that \. $\bP\big(f(x)<f(y)\big) \ge \frac12$ \. for
all \ts $g(x)=k$, \ts $g(y)=k+1$, \ts $1\le k < n$.  This follows from the fact
that every tournament has a Hamiltonian path.

\begin{rem}\label{r:sort-preference}
Fishburn's original motivation for Proposition~\ref{p:sort-pref-Fishburn} comes
from voting paradoxes, see \cite{Fis74,Geh06}.  There is also a parallel study of
\defng{intransitive dice} \ts which exhibits similar phenomena and has been
studied quantitatively in recent years, see \cite{HMRZ20,Poly22} and references
therein.
\end{rem}



\medskip

\section{Tools and ideas}\label{s:proof-ideas}

\subsection{Direct injections}\label{ss:proof-inject}
In contrast with many objects in enumerative and algebraic combinatorics,
posets inequalities are incredibly difficult to prove by a direct
combinatorial arguments. There are essentially two main tools:
lattice paths for posets of width two and promotion/demotion maps
for general posets (see below). Here is a quick list of ad hoc
direct injections uses in the proof of results in this survey:

\smallskip

\nin
$\bu$ \.
Theorem~\ref{t:basic-OP} with a lower bound for the values of the order polynomials,

\nin
$\bu$ \. Theorem~\ref{t:basic-OP} with Brenti's log-concavity for the order polynomials,

\nin
$\bu$ \. Gaetz--Gao surjection in \cite{GG20} proving Theorem~\ref{t:sid} (but not Theorem~\ref{t:sid-SP}),

\nin
$\bu$ \. Proposition~\ref{p:KS-mono-multiple} proving a weak version of the Kahn--Saks Conjecture~\ref{conj:KS-mon},

\nin
$\bu$ \. The original proof of the Bj\"orner--Wachs inequality (Theorem~\ref{t:ineq-BW}),

\nin
$\bu$ \. Reiner's proof of the $q$-BW inequality (Theorem~\ref{eq:ineq-BW-OP-Reiner}),

\nin
$\bu$ \. Lam–-Pylyavskyy generalizations of Fishburn inequalities (see Theorems~\ref{t:LP} and~\ref{t:LP-multi}).

\smallskip

Let us emphasize that these injections are relatively straightforward and
completely explicit, so the defect of the corresponding inequalities are all in~$\SP$.

\smallskip

\subsection{Promotion and demotion}\label{ss:proof-promo}
Let \ts $X=(X,\prec)$ \ts be a poset on \ts $|X|=n$ \ts elements.
\defn{Promotion} \ts is a bijection \. $\partial: \Ec(P) \to \Ec(P)$,
which we denote using the operator notation \ts $\partial: f\to \partial f$.
For \ts $f\in \Ec(P)$, let \. $x_1 \prec \ldots \prec x_\ell$ \.
be a maximal chain in~$P$ such that the sequence \ts $f(x_1),\ldots,f(x_\ell)$ \ts is
lexicographically smallest.  Equivalently, we have \ts $f(x_1)=1$, \ts element \ts $x_2$ \ts
is smallest cover of~$x_1$, \ts element \ts $x_3$ \ts
is smallest cover of~$x_2$, etc.  Define \. $\partial f\in \Ec(P)$ \. as
\begin{equation}\label{eq:proof-promo}
\partial  f \ts (z) \ := \
\begin{cases} \ f(x_{i+1})-1 & \ \ \text{ if \ }z=x_i\text{ \, for some } \, i<\ell\., \\
\ n &  \ \ \text{ if \ $z=x_\ell$}\.,\\
\ f(z)-1\ts &  \ \ \ \text{otherwise}\ts.
\end{cases}
\end{equation}
\defn{Partial promotion} \. $\partial_i$ \. is defined as the promotion on a poset
obtained by restriction to elements with $f$-values \. $1,\ldots,i$, so that \ts
$\partial_n=\partial$ \ts and \ts $\partial_1=1$.  \defn{Demotion} \ts and
\defn{partial demotions} \ts are defined as inverse bijections.

These operators were introduced and initially studies by Sch\"utzenberger in \cite{Sch}.
For \emph{standard Young tableaux} (linear extensions of Young diagram posets), the
promotion is called the \defng{jeu-de-taquin}, and is fundamental in the whole
of Algebraic Combinatorics (see e.g.\ \cite{Sag01,Sta-EC}). It is closely related to the
\defng{Robinson--Schensted--Knuth $($RSK$)$ correspondence} (ibid.), the
\defng{Edelman--Greene bijection} \cite{EG87}, and the \defng{NPS algorithm} \cite{NPS97}.
For \emph{increasing trees} (connected forest posets), the enumerative applications
were given in \cite[$\S$6]{KPP94}.

Partial promotion operators have algebraic relations which were investigated
by Lascoux and Sch\"utzenberger for Young tableaux.  In full generality,
they were studied by Haiman~\cite{Hai92} and Malvenuto--Reutenauer \cite{MR94}.
See \cite{Sta-promo} for an extensive survey.

\smallskip

The results in this survey whose proofs use explicit applications of the
promotion and demotion operators include:

\nin
$\bu$ \.
Theorem~\ref{t:basic-EHS} proving the summation inequality for antichains,

\nin
$\bu$ \.
Theorems~\ref{t:ineq-DDP} and~\ref{t:ineq-DDP-gen} proving the DDP log-concave inequality for the order polynomial,

\nin
$\bu$ \. Theorem~\ref{t:ineq-CPC-quant} proving weak quantitative version of \eqref{eq:ineq-CPC},

\nin
$\bu$ \. Theorems~\ref{t:Sta-vanish}, \ref{t:KS-vanish}, \ref{t:Sta-gen-vanish} and~\ref{t:gen-Stanley-unique},
proving various vanishing and uniqueness conditions for Stanley type inequalities, and

\nin
$\bu$ \. Theorem~\ref{t:Sta1-equality} proving equality conditions for the $k=1$ case of the generalized
Stanley inequality~\eqref{eq:Sta-gen}.

\smallskip

For the first two items, because the promotion/demotion proofs are completely explicit, the
defect of the corresponding inequalities are all in~$\SP$.  For the last three items, the
corresponding decision problems are in~$\poly$.

\smallskip

\subsection{Lattice paths}\label{ss:proof-lattice}
Let \ts $P=(X,\prec)$ \ts be a poset of width two with \ts $|X|=n$ \ts elements.
Fix a partition \. $X=\rC\sqcup \rC'$ \. into two chains, where
\. $\rC = \{u_1\prec \ldots \prec u_k\}$, and
\. $\rC'=\{v_1\prec \ldots \prec v_{\ell}\}$, where \ts $n=k+\ell$.
Denote by \ts $h, h'\in \Ec(P)$ \ts lexicographically minimal
and maximal linear extensions in~$P$, respectively.

For a linear extension \ts $f\in \Ec(P)$, denote by \ts $\ga(f)$ \ts
a \defn{lattice path} \ts in $\nn^2$, defined as follows.  Let \. $\ga(f): (0,0) \to (k,\ell)$ \.
start at \ts $(0,0)$, end at \ts $(k,\ell)$, and take \ts {\sf Up} \ts
steps for elements in~$\rC$, and \ts {\sf Right} \ts steps for elements in~$\rC'$.
Consider a region \ts $\Ga(P) \ssu \nn^2$ \ts enclosed between paths \ts $\ga(h)$ \ts
and \ts $\ga(h')$.  It is easy to see that all \ts {\sf Up-Right} \ts paths
\. $\ga: (0,0) \to (k,\ell)$ \. are in bijection with \ts $\Ec(P)$, i.e.\ \ts
$\ga=\ga(f)$ \ts for some \ts $f\in \Ec(f)$.

This connection has been frequently used to give estimates and prove
inequalities for width two posets, see e.g.\
\cite{BG96,CPP-CP,CPP-KS,CFG,GYY80}.  It is worth noting that
in \cite{CPP-RW} the injective proof of Stanley's inequality \eqref{eq:Sta}
for width two posets was extended to log-concavity of hitting probabilities
of general walks in the plane.  These walks can be viewed as
``oscillating linear extensions'', generalizing
\defng{oscillating Young tableaux}, see e.g.\ \cite{PP96,Sag90}.

For posets of larger (bounded) width, in principle the same general approach,
but direct injective arguments are harder to obtain.  We are aware of only
\cite{CPP-sort} which was able to combine it with anti-concentration
inequalities to obtain sharp bounds for the sorting probabilities of (skew)
Young diagrams (see~$\S$\ref{ss:sort-gen-KS}).

\smallskip

The results in this survey whose proofs employ lattice paths, include:

\smallskip

\nin
$\bu$ \. Theorem~\ref{t:GYY} proving \eqref{eq:GYY} inequality,

\nin
$\bu$ \. Special case of Theorem~\ref{t:ineq-Sta} proving \eqref{eq:Sta} inequality for width two posets,

\nin
$\bu$ \. Theorems~\ref{t:Sta-q} and \ref{t:Sta-KS-q} with $q$-analogues of~\eqref{eq:Sta} and~\eqref{eq:KS},

\nin
$\bu$ \. Theorem~\ref{t:CPC-summary}~$(2)$, proving \eqref{eq:ineq-CPC} for width two posets.

\smallskip

Let us emphasize that because lattice walks based proofs are completely explicit, the
defect of the corresponding inequalities are all in~$\SP$.

\smallskip

\subsection{Correlation inequalities}\label{ss:proof-FKG}
Let \ts $\cA$ \ts be a collection of subsets of $[n]$.  We say that $\cA$ is
\emph{closed upward}, if \ts $B\in \cA$ \ts for every \ts $B \supseteq A$ \ts
and \ts $A\in \cA$.  Similarly, \ts $\cA$ \ts is \emph{closed downward}, if
\ts $B\in \cA$ \ts for every \ts $B \subseteq A$ \ts and \ts $A\in \cA$.
The notions are easiest to understand when \ts $\cA$ \ts is a \emph{graph property},
i.e. a collection of (spanning) subgraphs of a complete graph.  Then, for example,
connectivity, Hamiltonicity, having all degrees at least $\De$, non-planarity,
non-$k$-colorability, or containing an $r$-clique \ts $K_r$ \ts are upward closed
properties.  The negations, such as $k$-colorability, planarity or disconnectivity
are downward closed properties.

The \defn{Harris--Kleitman inequality} \cite{Har60,Kle66}, states that for the
upward closed
collections \ts $\cA,\cB \subseteq 2^{[n]}$, we have positive correlation:
\begin{equation}\label{eq:HK}
|\cA\cap \cB| \. \cdot \. 2^n \ \ge \ |\cA| \. \cdot \. |\cB|\ts.
\end{equation}
Similarly, for \ts $\cA \subseteq 2^{[n]}$ \ts upward closed and \ts
$\cB \subseteq 2^{[n]}$ \ts downward closed, we have negative correlation:
$$
|\cA\cap \cB| \. \cdot \. 2^n \ \le \ |\cA| \. \cdot \. |\cB|\ts.
$$
These inequalities have a clear probabilistic meaning, e.g.\ \eqref{eq:HK}
can be rewritten as
\begin{equation}\label{eq:HK-prob}
\bP(A\in \cA\cap \cB) \ \ge \ \bP(A\in \cA) \. \cdot \. \bP(A\in \cB)\ts,
\end{equation}
where the probability is over uniform \ts $A\subseteq [n]$.

The Harris--Kleitman inequality had a series of generalizations, including
the celebrated \defng{FKG inequality} \ts by Fortuin, Kasteleyn and Ginibre \cite{FKG71}.
Here we only present the AD~inequality that is aptly called the \defng{four functions theorem} \ts
in~\cite{AS16}, and which implies the FKG inequality.
For every \. $\rho: Z\to \Rb_{+}$ \. and every \. $X\subseteq Z$, denote
\begin{equation}\label{eq:def-rho-sum}
 \rho(X) \ := \ \sum_{x \ts\in \ts X} \. \rho(x)\ts.
\end{equation}


\begin{thm}[{\rm \defn{Ahlswede--Daykin inequality}~\cite{AD78}}{}]\label{t:AD}
	Let \ts $\LL=(L,\vee,\wedge)$ \ts be a  finite distributive lattice on the ground set~$L$,
and let \. $\alpha,\beta,\gamma,\delta: L \to \Rb_{+}$ \. be nonnegative functions on $L$.
	Suppose we have:
\begin{equation}\label{eq:AD-def}
 \alpha(x) \.\cdot \. \beta(y) \ \leq \ \gamma(x \vee y) \.\cdot \. \delta(x \wedge y) \quad \text{ for every } \ \ x, \ts y \in L\ts.
 \end{equation}
	Then:
\begin{equation}\label{eq:AD}
 \alpha(X) \.\cdot \. \beta(Y) \ \leq \ \gamma(X \vee Y) \.\cdot \. \delta(X \wedge Y) \quad \text{ for every } \ \ X, \ts Y \subseteq L\ts.
\end{equation}
\end{thm}

\smallskip

The \defng{$q$-FKG inequality} \ts was introduced by Bj\"orner in \cite[Thm~2.1]{Bjo11}, who used it
to prove a $q$-analogue of \eqref{eq:ex-Bjo}.  Christofides \cite[Thm~1.5]{Chr09} gave the \defng{$q$-AD inequality},
while the multivariate   \defng{$\bq$-AD inequality} \ts we given in \cite[Thm~6.1]{CP-multi}.

\smallskip

These correlation inequalities play a key role in many poset inequalities throughout
the survey.  Notably, they are used in proofs of the following results:

\nin
$\bu$ \. Theorems~\ref{t:basic-OP-log-concave-strict}, \ref{t:basic-OP-q},
\ref{t:basic-OP-bq} and~\ref{t:KS-mono-reverse-width} on general inequalities for
the order polynomial,

\nin
$\bu$ \. Theorems~\ref{t:ineq-BW-OP-lower-bound} and \ref{t:ineq-BW-OP-lower-bound-conj} on the order polynomial
version of the BW inequality~\eqref{eq:BW},

\nin
$\bu$ \. Fishburn's inequality (Theorem~\ref{t:Fish}) and its generalizations in~$\S$\ref{s:Fish},
including the most general Theorem~\ref{t:main-OP-multi},

\nin
$\bu$ \. GYY inequality (Theorem~\ref{t:GYY}) and its generalization Shepp's inequality (Theorem~\ref{t:ineq-Shepp}),
obtained as a consequence of the order polynomial versions Theorems~\ref{t:ineq-Shepp-OP} and~\ref{t:ineq-Shepp-OP-q},

\nin
$\bu$ \. XYZ inequality (Theorem~\ref{t:ineq-XYZ}), and its applications Theorems~\ref{t:ineq-height} and~\ref{t:ineq-height-product}.

\smallskip

\subsection{Combinatorial optimization}\label{ss:proof-comb-opt}
Both the order polytope \ts $\cO_P$ \ts and the chain polytope~$\cS_P$ were defined
by Stanley, see \cite{Sta-two}.  The volume equality \ts $\vol \cO_p = e(P)/n!$ \ts in
\eqref{eq:order-def} is straightforward
via a simple triangulation of \ts $\cO_P$ \ts into congruent \defng{orthoschemes}
(also called \emph{path-simplices}) whose volume is \ts $1/n!$ \ts and which are
in bijection with linear extensions \ts $f\in \Ec(P)$.  Equality between volumes and
between Ehrhart polynomials of \ts $\cO_P$ \ts and \ts $\cS_P$ \ts in \eqref{eq:two-poset-OP},
follows from an explicit continuous piecewise-linear volume-preserving map
\ts $\xi: \cO_P \to \cS_P$\ts.

Chain polytope \ts $\cS_P$ \ts is especially useful in applications.

\smallskip

\nin
$\bu$ \. Proposition~\ref{p:LYM-antichains} and Theorem~\ref{t:LYM-BT} are proved using
the description of the dual polytope~$\cS_P^\circ\ts,$

\nin
$\bu$ \. Theorem~\ref{t:basic-entropy-KK} and Theorem~\ref{t:basic-height-two-BT}
are proved using the entropy functional on~$\cS_P\ts.$

\smallskip

An interesting approach was introduced by Sidorenko in \cite{Sid}, which combines
combinatorial duality and network flows.  Think of the comparability graph \ts $\Ga(P)$ \ts
as a directed network and enlarge it by adding bidirected edges for incomparable pairs.
Now define the \defn{Sidorenko flow} \ts by sending a unit along every linear extension.
Observe that the flow along these bidirected edges is equal in both directions,
so these edges can be deleted.  In fact, the flow can be described using promotions,
see e.g.\ \cite[$\S$8]{CPP-effective}.  The detailed analysis of this flow implies
several closely related inequalities:

\smallskip

\nin
$\bu$ \.
Theorem~\ref{t:basic-EHS} proving the summation inequality for antichains,

\nin
$\bu$ \. Sidorenko inequality (Theorem~\ref{t:sid}), and its generalizations
(Theorems~\ref{t:sid-gen} and~\ref{t:sid-BT}).

\smallskip

We refer to \cite{Schr03} for a very extensive discussion of combinatorial
optimization and applications to chain polytopes.

\smallskip

\subsection{Geometric inequalities}\label{ss:proof-geom}
Applying known geometric inequalities to order and chain polytopes~$\cO_P$ and~$\cS_P$
gives surprisingly strong implications.  These include:

\smallskip

\nin
$\bu$ \. The reverse Sidorenko inequality (Theorem~\ref{t:sid}), proved by applying
the \defng{Saint-Raymond inequality} \cite{StR81}.  This is a special case of
\defng{Mahler's conjecture}, see e.g.\ \cite[$\S$1.3]{Tao}.

\nin
$\bu$ \. The reverse Sidorenko inequality (Theorem~\ref{t:sid-BBS}), proved by applying
the \defng{Blaschke--Santal\'o inequality} \ts to polytope~$\cS_P$ and its dual~$\cS_P^\circ$,
see e.g.\ \cite[$\S$24.5]{BZ-book}.

\nin
$\bu$ \. The mixed Sidorenko inequality (see Remark~\ref{r:sid-mixed}),
proved by applying \defng{Godberson's conjecture} \ts established in \cite{AASS}
for anti-blocking polytopes, including \ts $\cS_P\ts.$

\nin
$\bu$ \. Stanley inequality (Theorems~\ref{t:ineq-Sta} and~\ref{t:ineq-Sta-gen}) and
Kahn--Saks inequality (Theorem~\ref{t:ineq-KS}), proved by applying the
\defng{Alexandrov--Fenchel inequality} \ts to sections of~$\cO_P\ts.$

\nin
$\bu$ \. Special cases and weak versions of Conjecture~\ref{conj:CPC} given in
Theorem~\ref{t:CPC-summary}~$(3)$ and Theorem~\ref{t:ineq-CPC-quant}, are proved
by using \defng{Favard's inequality}.

\nin
$\bu$ \. Kahn--Linial's proof of a weak version of Conjecture~\ref{conj:1323}
and  parts \ts $(5)-(7)$ \ts of Theorem~\ref{t:sort-1323-summary}, are proved
using \defng{Mityagin's inequality}, which is a generalization of
\defng{Gr\"unbaum's inequality}, see~$\S$\ref{ss:sort-gen}.

\nin
$\bu$ \. Kahn--Yu and Yu inequalities (Theorems~\ref{t:sort-KY} and~\ref{t:sort-Yu})
are proved using the \defng{Brunn--Minkowski inequality} \ts combined with the
\defng{K.~Ball inequality}.

\smallskip

Additionally, the geometric analysis of equality conditions of geometric
inequalities can be translated to equality conditions of poset inequalities.
This is a very recent direction of research pioneered by Shenfeld and van Handel
\cite{SvH20}, in their study of equality cases of the Alexandrov--Fenchel inequality.

\smallskip

\nin
$\bu$ \. Equality cases of Stanley inequality in Theorem~\ref{t:Sta-equality},

\nin
$\bu$ \. Equality cases of the Kahn--Saks inequality in Theorem~\ref{t:KS-equality},

\nin
$\bu$ \. Ma--Shenfeld's proof of vanishing of the generalized Stanley inequality in Theorem~\ref{t:Sta-gen-vanish},

\nin
$\bu$ \. Equality cases of the generalized Stanley inequality in Theorem~\ref{t:Sta1-equality}, $k=1$ case,

\nin
$\bu$ \. Equality conditions of the generalized Stanley inequality in Theorems~\ref{t:ineq-Sta-gen-MS-equality}
and~\ref{t:ineq-MS-supercrit}.

\smallskip

In contrast with the previous list of applications, all of these results are very
technical and difficult to obtain.  Even the definitions can be difficult as
a direct translation from the geometry can pose challenges.  Streamlining one
such definition is the point of our Definition/Lemma~\ref{dl:ineq-MS-crit-simp}
proved in~$\S$\ref{ss:proof-crit-simp}.

\smallskip

\subsection{Combinatorial atlas}\label{ss:proof-LA}
In~\cite{CP-atlas}, we presented a new linear algebraic approach to
log-concave inequalities, based on a structure we call \defng{combinatorial
atlas}.  Roughly, this is a combinatorial setup of vectors and matrices
related to each other according to certain directed graphs.  These matrices
are associated with posets and contain counting of numbers of linear
extensions.  The setup allows one to prove by induction that these matrices
are hyperbolic, starting with two element posets as the base of induction.

Formally, a \ts $d\times d$ \ts real matrix \ts $\bM$ \ts is called \defn{hyperbolic},
if
\begin{equation}\label{eq:Hyp}\tag{Hyp}
\langle \vb, \bM \wb \rangle^2 \  \geq \  \langle \vb, \bM \vb \rangle  \langle \wb, \bM \wb \rangle \quad \text{for all \ \, $\vb, \wb \in \Rb^{d}$ \ \  s.t.\ \ \ $\langle \wb, \bM \wb \rangle > 0$}.
\end{equation}
It is not hard to see that \ts $\bM$ \ts is hyperbolic if and only if it has
at most one positive eigenvalue (including multiplicities).  While the eigenvalue
conditions allows to obtain the step of induction, the inequality \eqref{eq:Hyp}
implies a number of correlation and Stanley type inequalities:

\smallskip

\nin
$\bu$ \. Deletion correlations (Theorem~\ref{t:ineq-delete}),

\nin
$\bu$ \. Subset correlations (Theorems~\ref{t:ineq-deletion-Sta} and~\ref{t:ineq-deletion-Sta-arrow}),

\nin
$\bu$ \. Covariance inequalities (Theorems~\ref{t:poset-cov}, \ref{t:poset-cov-multiple} and~\ref{t:ineq-cov-arrow}),

\nin
$\bu$ \. Unique covers special cases (Theorem~\ref{t:ineq-four-elements} and~\ref{c:ineq-four-elements-three}),

\nin
$\bu$ \. Weighted Stanley inequality (Theorem~\ref{t:Sta-weighted}),

\nin
$\bu$ \. Equality conditions for the Stanley inequality (Theorem~\ref{t:Sta-equality}).

\smallskip

We refer to \cite{CP-intro} for the introduction to the combinatorial atlas
technology and to \cite{CP-corr} for applications to correlation inequalities.



\medskip

\section{Proofs of technical results}\label{s:proof-crit-simp}

\subsection{Proof of Theorem~\ref{t:ineq-CPC-cons} and Proposition~\ref{p:ineq-CPC-cons}}\label{ss:proof-CPC-cons}
\begin{proof}[Proof of Theorem~\ref{t:ineq-CPC-cons}]
In notation of Conjecture~\ref{conj:CPC}, let \ts $P:=(X,\prec)$ \ts be a poset with \ts $|X|=n$ \ts elements and
fixed  element \ts $z\in X$.
Let \ts $Q:=(Y,\prec')$ \ts given by \ts $Y:=X\cup \{x,y,w\}$,
with relations \. $u\prec' v \. \Leftrightarrow \. u\prec v$ \. for all \ts $u,v\in X$, \. $
x \ \prec'  y  \prec' u$ \. for all \. $u \in X$, and \. $x \ \prec' \ w$.

For all \ts $a\geq 1$, we have:
	\begin{align}\label{eq:F1k}
		\aF_{xyz}(Q,1,a) \ = \ (a-1) \. \aN(P,z,a-1) \ + \. (n+1-a) \. \aN(P,z,a).
	\end{align}
This follows from the observation that \ts $\aF_{xyz}(Q,1,a)$ \ts in the number of
linear extensions \ts $f \in \Ec(P)$ \ts for which \. $f(x)=1$, \ts $f(y)=2$, \ts $f(z)=a+2$,
while \ts $f(w)\in \{3,\ldots, n+3\}$.

Also note that, for all $k\geq 2$,	
\begin{align}\label{eq:F2k}
	\aF_{xyz}(Q,2,a) \ = \  \aN(P,z,a).
\end{align}
Indeed, this follows from the observation that \ts $\aF_{xyz}(Q,2,a)$ \ts counts those
linear extensions for which \. $f(x)=1$, \ts $f(w)=2$, \ts $f(y)=3$ \ts and \ts $f(z)=a+3$.

Now, \eqref{eq:ineq-CPC} implies that for all \ts $a,i\geq 1$, we have:
\begin{equation}\label{eq:appB1}
	\aF_{xyz}(Q,1,a) \cdot  \aF_{xyz}(Q,2,a+i) \ \leq \  \aF_{xyz}(Q,1,a+i) \cdot \aF_{xyz}(Q,2,a).
\end{equation}
By \eqref{eq:F1k} and \eqref{eq:F2k}, the LHS of \eqref{eq:appB1} is equal to
\[ 	(a-1) \, \aN(P,z,a-1)  \cdot  \aN(P,z,a+i) \,
+ \, (n+1-a) \, \aN(P,z,a)  \cdot  \aN(P,z,a+i),
\]
while the RHS of \eqref{eq:appB1} is equal to
\[ (a+i-1) \. \aN(P,z,a+i-1)  \cdot  \aN(P,z,a) \, + \, (n+1-a-i) \, \aN(P,z,a+i)  \cdot \aN(P,z,a).
\]
Thus, \eqref{eq:ineq-CPC-SS} follows from \eqref{eq:appB1} and the two equations above.
\end{proof}

\smallskip

We restate Proposition~\ref{p:ineq-CPC-cons} for clarity, writing out all the inequalities:

\smallskip

\begin{prop}[{\rm = Proposition~\ref{p:ineq-CPC-cons}}{}] 
Let \ts $P=(X,\prec)$ \ts be a poset with \ts $|X|=n$ \ts elements, let \ts
$x \in X$ \ts and \ts $a \in [n]$.  Suppose that
\. $\aNr(P, x,a)>0$. Then \eqref{eq:ineq-CPC-SS} implies:
$$
\aNr(P, x,a)^2 \. = \. \aNr(P, x,a+1) \cdot \aNr(P, x,a-1) \
\Leftrightarrow \
\aNr(P, x,a+1)  =  \aNr(P, x,a) = \aNr(P, x,a-1).
$$
Additionally, we have Stanley's inequality:
$$
\aNr(P, x,a)^2 \. \ge \. \aNr(P, x,a+1) \cdot \aNr(P, x,a-1).
$$
\end{prop}

\begin{proof}
For the first part, the \. $\Leftarrow$ \. direction is trivial.  For the \. $\Rightarrow$ \. direction,
combining the inequality \ts \eqref{eq:ineq-CPC-SS} (with $z \gets x$ and $i\gets 1$) and the assumption
\. $\aN(P, x,a)^2 = \aN(P, x,a+1) \cdot \aN(P, x,a-1)$ \. gives \.
$$\aN(P, x,a+1) \cdot \aN(P,x,a-1)  \, \ge \, \aN(P, x,a) \. \cdot \. \aN(P, x,a+1).
$$
With the assumption
\. $\aN(P, x,a)>0$ (and thus $\aN(P, x,a+1)>0$), this gives \. $\aN(P, x,a-1) \ge \aN(P, x,a)$.
The same argument for the dual poset \ts $P^\ast$ \ts gives \.
$\aN(P, x,a+1) \ge \aN(P, x,a)$. This implies the first part.

For the second part, note that \eqref{eq:ineq-CPC-SS}  (with $z \gets x$ and $i\gets 1$) can rewritten as:
$$
(a-1) \big(\aN(P, x,a)^2 \ts - \ts \aN(P, x,a-1) \. \aN(P, x,a+1)\big)
\, \geq \, \aN(P, x,a) \big(\aN(P, x,a+1) \ts - \ts \aN(P, x,a)\big).
$$
Suppose \. $\aN(P, x,a) < \aN(P, x,a+1)$.  Then the RHS of the above
inequality is $\ge 0$, and the LHS implies Stanley's inequality.
Similarly, for \. $\aN(P, x,a) < \aN(P, x,a-1)$,  the same
argument for the dual poset \ts $P^\ast$ \ts also implies Stanley's
inequality.  Therefore, we have \. $\aN(P, x,a) \ge \aN(P, x,a+1)$ \. and \.
$\aN(P, x,a) \ge \aN(P, x,a-1)$, which immediately implies the result.
\end{proof}

\smallskip

\subsection{Proof of Lemma~\ref{dl:ineq-MS-crit-simp}}\label{ss:proof-crit-simp}
We now present the definitions as stated in \cite{MS22}, and show that
they are in fact equivalent to our Definition~\ref{dl:ineq-MS-crit-simp}.
What follows is a technical argument which ordinarily would not fit a
general survey.  We include it here in order to show that the
subcritical/critical/supercritical poset characterization is in~$\poly$.

\smallskip

Assume that \eqref{eq:ineq-MS-Pos} holds, and that \. $c_{\ell-1}<a < c_\ell$.
It follows from \. $\aN_{\bz\bc}(P, x,a-1)>0$ \. and \. $\aN_{\bz\bc}(P, x,a+1) > 0$, that
\. $c_{\ell-1}+1 < a <c_{\ell}-1$.  Let us slightly simplify the notation:
		\begin{align*}
			(y_0,y_1,\ldots, y_{k},y_{k+1},y_{k+2}) \ & \longleftarrow \  \bigl(z_0=\wh 0, z_1,\ldots, z_{\ell-1}, x, z_\ell, \ldots, z_{k}, z_{k+1}=\wh 1\big), \\
(d_0,d_1,\ldots, d_{k+1},d_{k+2})  \ & \longleftarrow \   (0,c_1,\ldots, c_{\ell-1}, a, c_\ell,\ldots, c_k,n+1).
		\end{align*}
%
%
For \ts $0\le r < s \le k+1$,  let \ts $\aD(r,s)  := \ga(y_r,y_s)$.
Denote by \ts $\rR_p$ \ts a collection of disjoint nonempty intervals
\. $[r_1,r_2],\ldots,[r_{2p-1},r_{2p}]$, such that $[r_1,r_2] \neq [0,k+2]$.

\smallskip

\begin{definition}[{\rm cf.\ \cite[Def.~2.11]{MS22}}{}]
We say that a quintuple \. $(P,x,a,\zb,\cb)$ \. is\footnote{This definition is more combinatorial in flavor that the one in \cite[Def~2.11]{MS22}, which has a more geometric flavor. It is easy to show that these definitions are equivalent, we omit the details.}

\nin
$\bu$ \ \defnb{subcritical}  \ if \.
		for every $p\geq 1$ and every \. $\rR_p$ \. as above, we have:
			\begin{align}\label{eq:subcrit-MS}\tag{subcrit-MS}
			\aD(r_{1}, r_{2}) \. + \. \ldots \. + \. \aD(r_{2p-1}, r_{2p}) \, \leq \, \de  \. + \.  \De \.,
		\end{align}
where \. $\de := 2-\xi-\ze$,
$$\De \, : = \, (d_{r_{2}}\ts - \ts d_{r_{1}}\ts - \ts 1) \. + \. \ldots  \. + \.  (d_{r_{2p}}\ts - \ts d_{r_{2p-1}}\ts - \ts 1),
$$
and \. $a:=a(r_1,\ldots,r_{2p})$, \. $b:=b(r_1,\ldots,r_{2p})$ \. are given by
	\begin{equation}\label{eq:ab-def}
\aligned		& \xi := \left\{
\aligned
& \. 1 \quad \text{if} \quad [\ell-1,\ell]  \subseteq  [r_{2i-1},r_{2i}] \ \ \text{for some} \ \ i \in [p],\\
& \. 0 \quad \text{otherwise}, \endaligned \right.
\\	& \ze := \left\{
\aligned
& \. 1 \quad \text{if} \quad [\ell,\ell+1]  \subseteq  [r_{2i-1},r_{2i}] \ \ \text{for some} \ \ i \in [p],\\
& \. 0 \quad \text{otherwise}. \endaligned \right.
\endaligned
	\end{equation}

		%
\nin
$\bu$ \ \defnb{critical} \ if \.
		for every $p\geq 1$ and every \. $\rR_p$ \. as above, we have:
			\begin{align}\label{eq:crit-MS}\tag{crit-MS}
			\aD(r_{1}, r_{2}) \ts + \ts \ldots \ts + \ts \aD(r_{2p-1}, r_{2p}) \, \leq \,
\de \ts -\ts 1  \. + \.  \De.
		\end{align}
\nin
$\bu$ \ \defnb{supercritical} \ if \.
		for every $p\geq 1$ and every \. $\rR_p$ \. as above, we have:
			\begin{align}\label{eq:supercrit-MS}\tag{supercrit-MS}
			\aD(r_{1}, r_{2})  \. + \.  \ldots  \. + \. \aD(r_{2p-1}, r_{2p}) \, \leq \, \de \ts -\ts 2  \. + \.  \De.
		\end{align}
\end{definition}

\smallskip

We now proceed to the proof which is separated into parts:

\smallskip

\begin{proof}[\eqref{eq:ineq-MS-subcrit} \ $\Leftrightarrow$ \ \eqref{eq:subcrit-MS}] \.
By the argument in \eqref{eq:ineq-MS-basic}, the inequality \eqref{eq:ineq-MS-subcrit}  always holds.
It suffices to show the same for \eqref{eq:subcrit-MS}.
By \eqref{eq:ab-def}, we have \. $\xi,\ze \leq 1$, so \. $2-\xi-\ze\geq 0$.
Since \. $\aN_{\zb \cb}(P,x,a)\ge 1$, for every \. $f \in \aN_{\zb \cb}(P,x,a)$ \. and \. $i \in [p]$, we have:
	\begin{equation}\label{eq:D-basic}
	\begin{split}
	 \aD(r_{2i-1}, r_{2i}) \ 
	 & \leq  \  \big|\{f^{-1}(i) \, : \, d_{r_{2i-1}}< i <d_{r_{2i}}  \}\big| \ = \ d_{r_{2i}}\. - \. d_{r_{2i-1}} \. - \. 1.
	\end{split}
	\end{equation}
Summing these inequalities proves the claim.
\end{proof}

\begin{proof}[\eqref{eq:ineq-MS-crit} \ $\Leftarrow$ \ \eqref{eq:crit-MS}] \.
Let \ts $(r,s)\in \La$ \ts be a splitting pair such that \ts $c_r< a <c_s$\ts.
Let \. $r_1:=r$ \. and \. $r_2:=s+1$\..
Note that $(z_r,z_{s})=(y_{r_1},y_{r_2})$, and it follows that $\xi=\ze=1$.
We have:
	\[ \ga(z_r,z_s) \, = \,  \aD(r_1,r_2) \, \leq_{\eqref{eq:crit-MS}}  \,
(1-\xi-\ze) \. + \. (d_{r_2}-d_{r_1}-1)  \, = \,  c_s \. - \. c_r \. - \. 2,
\]
as desired.
\end{proof}
		
\begin{proof}[\eqref{eq:ineq-MS-crit} \ $\Rightarrow$ \ \eqref{eq:crit-MS}] \.
Let \ts $\rR_p$ \ts be as above.
	 If $\xi+\ze \leq 1$, then \eqref{eq:crit-MS} follows from the fact that \. $ \aD(r_{2i-1}, r_{2i})  \leq d_{r_{2i}}-d_{r_{2i-1}}-1$ \. by \eqref{eq:D-basic}.
	 Thus, we can assume that \ts $\xi+\ze=2$.
	 This implies that there exists $j \in [p]$ such that $[\ell-1, \ell+1] \subseteq [r_{2j-1},r_{2j}]$, which in turn implies that
	\[ c_{r_{2j-1}} \ = \  d_{r_{2j-1}} \ \leq \  d_{\ell-1} \  < \ d_\ell= a  \  < \ d_{\ell+1} \ \leq  \  d_{r_{2j}} \ = \ c_{r_{2j}-1}\..  \]
	 It then follows that
\[  \aD(r_{2j-1},r_{2j}) \, = \, \ga\big(z_{r_{2j-1}},z_{r_{2j}-1}\big) \,  \leq_{\eqref{eq:ineq-MS-crit}} \,  c_{r_{2j}-1} \. - \. c_{r_{2j-1}} -2 \, = \, d_{r_{2j}} \. - \. d_{r_{2j-1}} \. - \. 2.
\]
	On the other hand, for every \. $i \in [p]$, \. $i \ne j$, we have:
	 \[  \aD(r_{2i-1},r_{2i})  \, \leq \, d_{r_{2i}} \. - \. d_{r_{2i-1}} \. - \. 1. \]
	 Combining these two inequalities with the assumption that \ts $\xi+\ze=2$,
	 we get \ts \eqref{eq:crit-MS}.
\end{proof}

\begin{proof}[\eqref{eq:ineq-MS-supercrit} \ $\Leftarrow$ \ \eqref{eq:supercrit-MS}] \.
Let \ts $(r,s)\in \La$, so that \ts $c_r< a <c_s$.
Let \. $r_1:=r$ \. and \. $r_2:=s+1$\..
Note that $(z_r,z_{s})=(y_{r_1},y_{r_2})$, and it follows that $\xi=\ze=1$.
It then follows  that
\[ \ga(z_r,z_s) \, = \,   \aD(r_1,r_2) \, \leq_{\eqref{eq:supercrit-MS}}  \,  (-\xi-\ze) \. + \. (d_{r_2}-d_{r_1}-1)  \, = \.  c_s \. - \. c_r \. - \. 3,
\]
as desired.
\end{proof}

\begin{proof}[\eqref{eq:ineq-MS-supercrit} \ $\Rightarrow$ \ \eqref{eq:supercrit-MS}] \.
Let \ts $\rR_p$ \ts be as above.
We split the proof into three cases, depending on the values of \ts $\xi+\ze$.  First,
		for \ts $\xi+\ze =0$, the inequality \eqref{eq:supercrit-MS} follows from  \.
$ \aD(r_{2i-1}, r_{2i})  \leq d_{r_{2i}}-d_{r_{2i-1}}-1$ \. by \eqref{eq:D-basic}.
		
		Second, for \ts $\xi+\ze =1$, there exists \. $j \in [p]$ \. such that \. $\ell=r_{2j-1}$ \. or \.  $\ell=r_{2j}$\..
		Without loss of generality, we can assume that   \.  $\ell=r_{2j-1}$\..
		For every \. $f \in \aN_{\zb \cb}(P,x,a-1)$, we have:
		\begin{align*}
		   \aD(r_{2j-1},r_{2j}) \ & = \   \ga\big(y_{r_{2j-1}}, x\big)    \ \leq \
\big|\{f^{-1}(i) \. : \. d_{r_{2j-1}} \. < \. i  \. < \. a-1  \}\big|\\
		  & = \ a-1-d_{r_{2j-1}}-1 \ = \ d_{r_{2j}} - d_{r_{2j-1}} -2.
		\end{align*}
	On the other hand, for every \. $i \in [p]$, \. $i \ne j$, we have:
\[  \aD(r_{2i-1},r_{2i})  \ \leq \ d_{r_{2i}} -  d_{r_{2i-1}} -1. \]
Combining these two inequalities with the assumption that \ts $\xi+\ze=1$,
we get \eqref{eq:supercrit-MS}, as desired.

Third, for \ts $\xi+\ze =2$, there exists \ts $j \in [p]$, such that \.
$[\ell-1, \ell+1] \subseteq [r_{2j-1},r_{2j}]$. This implies that
\[ c_{r_{2j-1}} \ = \  d_{r_{2j-1}} \ \leq \  d_{\ell-1} \  < \ d_\ell= a  \  < \ d_{\ell+1} \ \leq  \  d_{r_{2j}} \ = \ c_{r_{2j}-1}\..  \]
It then follows that
\[  \aD(r_{2j-1},r_{2j}) \ = \ \ga\big(z_{r_{2j-1}},z_{r_{2j}-1}\big) \  \leq_{\eqref{eq:ineq-MS-supercrit}} \  c_{r_{2j}-1} - c_{r_{2j-1}} -3 \ = \ d_{r_{2j}} -  d_{r_{2j-1}} -3. \]
On the other hand,  for every \. $i \in [p]$, \. $i \ne j$, we have:
\[  \aD(r_{2i-1},r_{2i})  \ \leq \ d_{r_{2i}} -  d_{r_{2i-1}} -1. \]
Combining the two inequalities above with the assumption that $a+b=0$,
we get \eqref{eq:supercrit-MS}, as desired.
\end{proof}

\medskip

\section{Final remarks and open problems}\label{s:finrem}

\subsection{The nature of linear extensions}\label{ss:finrem-nature}
How is the number of linear extensions different from all other combinatorial
counting functions?  This question is worth addressing since we devoted so
much space to the subject.  It is also a very difficult question which has
more than one answer.

First, we clarify that counting linear extensions is as
much a distinct area of Poset Theory as counting colorings, spanning
trees, or counting perfect matchings are distinct areas of Graph Theory.
These three areas are covered by a large number of surveys and monographs,
see e.g.\ \cite{LP86,MR02,Moon70} for our favorites.  By contrast,
this is probably the first survey dedicated to linear extensions of
general posets.  Of course, special cases such as standard Young tableaux
and increasing trees are extensively covered in the literature, see e.g.\
\cite{AR15,KPP94,Sta-EC}.

Second, from the complexity point of view, linear extensions are
in the middle of the spectrum: counting is $\SP$-complete, while the
existence is trivially in~$\poly$.  Of course, for colorings these
problems are $\SP$-complete and $\NP$-complete, respectively.  For
spanning trees, these problems are in $\FP$ and~$\poly$, respectively.
For perfect matchings, the counting is $\SP$-complete and while the
existence is in $\poly$, for a nontrivial reason.

Third, the $(1\pm \ve)$ approximation counting is $\NP$-hard for the
number of $3$-colorings and sufficiently small $\ve>0$ \cite{Vaz01},
while for the number of perfect matchings there is a Markov chain based algorithm
with a difficult analysis \cite{JSV04}.   This puts linear extensions
in the middle of the spectrum again --- it is a hard counting problem
with an easy Markov chains algorithm, see the discussion in~$\S$\ref{ss:CS-random}.


It is not surprising that for counting functions which are harder to
compute it is easier to prove that they are hard to compute, and it
is harder if not impossible to obtain good bounds.  On the other hand,
even when counting is easy like for the spanning trees, perfect
matchings in planar graphs, or standard Young tableaux of skew shapes,
the inequalities can be quite interesting, see e.g.~$\S$\ref{ss:ex-yd}
and \cite{Gor21,Gri76}.

In summary, there is no real pattern among these counting problems.
Each of them is its own mini-universe with its unique advances and
challenges.  Unfortunately, counting linear extensions is the least
explored of these counting problems.  Hopefully, this survey will
pave a way for further studies in the area.


\subsection{Early history of linear extensions}\label{ss:finrem-hist}
The area of linear extensions of posets was started in 1930, with
Szpilrajn's \defng{extension theorem} \cite{Szp30}, which is nontrivial
for infinite posets and gives \ts $e(P)\ge 1$ \ts for finite
posets.  Birkhoff's
\defng{fundamental theorem for finite distributive lattices} \cite{Bir33},
states that every finite distributive lattice \ts $\cL$ \ts is isomorphic to
a lattice \ts $J(P)$ \ts of lower order ideas of a finite poset~$P$,
see e.g.\ \cite[$\S$3.4]{Sta-EC}. Stanley noticed that this implies
that the number of maximal chains in \ts $\cL$ \ts is~$\ts e(P)$,
see \cite[Prop.~4.1]{Sta-thesis}.
For example, the number of maximal chains in the Boolean algebra \ts $B_n$ \ts
is \ts $e(A_n)=n!$, cf.\ Example~\ref{ex:LYM-Boolean}.

In Enumerative Combinatorics, enumeration of linear extensions arose
naturally in the context of standard Young tableaux in MacMahon's
classical work \cite{Mac15}, which led to the hook-length formula~\eqref{eq:HLF}
and its generalizations, see~$\S$\ref{ss:ex-yd}.
Building in part on Knuth's paper \cite{Knu70}, Stanley's thesis
on \defng{$P$-partition theory} \cite{Sta-thesis} generalized this
work both relating and unifying it with the study of \defng{symmetric functions}.
 Soon after, Sch\"utzenberger's work \cite{Sch} on \defng{promotion operators} \ts
led to development of \defng{RSK} \ts and other combinatorial tools used to this day.
For further references, see Stanley's historical notes in \cite[pp.~383--386]{Sta-EC}.

In Computer Science, linear extensions of posets arose naturally in connection
with \defng{sorting problems}.  Notably, Kislitsyn~\cite{Kis68} which stated the \ts
$\frac{1}{3}-\frac23$ \ts Conjecture~\ref{conj:1323} and the unimodality of \ts
$\{\aN(P,x,a)\}$, which was later rediscovered by Rivest and eventually proved by
Stanley \cite{Sta-AF}.
Soon after, Fredman~\cite{Fre75} and Sch\"onhage~\cite{Scho76} introduced
(different, but related) sorting problem in the West, and gave the information
theoretic lower bounds.


\subsection{What's next?}\label{ss:finrem-next}
This survey may seem extensive, but to resolve long open conjectures
we would need new tools and ideas beyond those in Section~\ref{s:proof-ideas}.
The latest entrants --- the combinatorial atlas and the geometric approach
to Alexandrov--Fenchel inequalities, led to a great deal of progress and
there is hope for further advances.

Let us single out a few open problems mentioned in the survey.  First,
we personally find the Kahn--Saks conjecture~\eqref{conj:1323-KS} more
important than the \. $\frac{1}{3}-\frac23$ \. Conjecture~\ref{conj:1323},
although both remain out of reach with existing technology.
Next, we believe that the Cross Product Conjecture~\ref{eq:CPC-OP} is
false already for width three posets. We tried to disprove it by
finding a counterexample to Conjecture~\ref{conj:ineq-CPC-cons} (which follows from CPC),
but the extensive computer experiments have yet to succeed.\footnote{We
are grateful to Greta Panova and Andrew Sack for their help in designing
and running these experiments.}

In a different direction, it would be interesting to make even a small
improvement towards Conjecture~\ref{conj:ineq-second-moment}.  In fact,
any constant \ts $(2-\epsilon)$ \ts in the RHS of
\eqref{eq:ineq-second-moment-conj} would already be a major progress.
Finally, Conjectures~\ref{conj:XYZ-SP} and~\ref{conj:ineq-Sta-SP}
remain a major challenge both in Combinatorics and Computational
Complexity. 

\vskip.6cm

{\small

\subsection*{Acknowledgements}
We are grateful to Karim Adiprasito, Max Aires, Luis Ferroni, Nikita Gladkov, Jeffry Kahn,
Alejandro Morales,
Greta Panova, F\"edor Petrov, Yair Shenfeld, David Soukup and
Ramon van Handel for many helpful discussions and remarks on the
subject. Special thanks to Richard Stanley for comments on the
draft of this survey.  Both authors were partially supported
by the~NSF.
}

\newpage





\end{document}